\documentclass[11pt]{amsart}
\usepackage{amsmath,amsfonts,amsthm,amssymb,amscd}
\RequirePackage[OT1]{fontenc}

\DeclareMathOperator*{\esssup}{ess\,sup}

\newtheorem{definition}{Definition}[section]
\newtheorem{defn}{Definition}[section]
\newtheorem{theorem}{Theorem}
\newtheorem{lemma}[defn]{Lemma}
\newtheorem{prop}[defn]{Proposition}
\newtheorem{proposition}[defn]{Proposition}

\newtheorem{remark}[defn]{Remark}
\newtheorem{corollary}{Corollary}

\newcommand{\beginsec}{\setcounter{equation}{0}}

\def\R{{\mathbb R}}
\def\N{{\mathbb N}}
\def\C{{\cal C}}
\def\Z{{\mathbb Z}}
\def\1{{\mathbb I}}
\def\E{{\mathbb E}}
\def\P{{\mathbb P}}

\def\Q{{\mathbb Q}}
\def\J{{\mathcal J}}
\def\S{\widehat{{\mathcal U}}}

\newcommand{\vect}{v}

\newcommand{\stopr}{\theta}

\newcommand{\ptt}{\bar{y}}
\newcommand{\coeff}{s}

\def\ind{{\mathbb I}}

\def\ra{\rightarrow}
\def\hyp{{\cal S}} 

\def\fntwo{\bar{h}}
\def\Cone{\Theta}

\def\barz{\bar{z}}
\def\mart{S}
\def\martf{S^{f}}
\def\id{\chi} 
\def\shift{S^\id}
\def\stop{\tau_0}
\def\stopin{\varsigma}
\def\stopout{\tau}
\def\stopoutt{\stopout}
\def\stopouti{\tau_1}
\def\stopinnew{\varrho}
\def\stopoutnew{\iota}
\def\stopinnewt{\varrho}
\def\stopinf{\varrho}
\def\stopoutf{\iota}

\def\MA{{\mathcal A}}

\def\MH{{\mathcal H}}

\def\MM{{\mathcal M}}
\def\MB{{\mathcal B}}

\def\MI{{\mathcal I}}
\def\MF{{\mathcal F}}
\def\ME{{\mathcal E}}
\def\MK{{\mathcal K}}
\def\MV{{\mathcal V}}
\def\MU{{\mathcal U}}

\def\MT{{\mathcal T}}
\def\ML{{\mathcal L}}
\def\MC{{\mathcal C}}
\def\MJ{{\mathcal J}}
\def\MO{{\mathcal O}}
\def\MN{{\mathcal N}}
\def\MP{{\mathcal P}}
\def\MU{{\mathcal U}}

\def\cal{\mathcal}

\def\dist{{\rm{dist}}}
\def\supp{{\rm{supp}}}
\def\sm{\setminus}
\def\csto{\eta}
\def\cstt{\lambda}

\def\martin{M}
\def\fvproc{A}
\def\conset{\Lambda}
\def\mvect{\Upsilon}
\def\pushproc{R}
\def\vect{v}
\def\Proc{\bar{H}}

\newcommand{\conv}{\overline{\mbox{\rm co}}}
\newcommand{\noi}{\noindent }
\newcommand{\ba}{\begin{array}}
\newcommand{\ea}{\end{array}}
\newcommand{\bea}{\begin{eqnarray}}
\newcommand{\eea}{\end{eqnarray}}
\newcommand{\beas}{\begin{eqnarray*}}
\newcommand{\eeas}{\end{eqnarray*}}
\newcommand{\be}{\begin{equation}}
\newcommand{\ee}{\end{equation}}
\newcommand{\bt}{\begin{theorem}}
\newcommand{\et}{\end{theorem}}
\newcommand{\bc}{\begin{center}}
\newcommand{\ec}{\end{center}}
\newcommand{\ben}{\begin{enumerate}}
\newcommand{\een}{\end{enumerate}}
\newcommand{\lan}{\langle}
\newcommand{\ran}{\rangle}
\newcommand{\ei}{\end{itemize}}

\newcommand{\ds}{\displaystyle}

\newcommand{\ve}{\varepsilon}

\newcommand\cc{{\cal C}(\left[0,\infty\right):\R^J)}
\newcommand\ccspace{\cal C}

\newcommand{\brm}{W}
\newcommand{\exx}{X}
\newcommand{\y}{Y}
\newcommand{\zee}{Z}
\newcommand{\smart}{\shift}
\newcommand{\eh}{H}

\def\rv{\xi}
\def\intset{{\mathcal R}}

\title[The Submartingale Problem for Reflected Diffusions]{On the
  Submartingale Problem for Reflected Diffusions in Domains with
  Piecewise Smooth Boundaries}
\author{Weining Kang}
\address{Department of Mathematics \& Statistics\\
University of Maryland, Baltimore County\\
1000 Hilltop Circle\\
Baltimore, MD 21250 } \email{wkang@umbc.edu}

\author{Kavita Ramanan}
\address{Division  of Applied Mathematics\\
Brown University\\
Providence, RI 02118 \\
USA}
\email{Kavita\_Ramanan@brown.edu}

\thanks{The second author was partially supported by NSF grants
 CMMI-1052750 (formerly 0928154), CMMI-1234100 and DMS 1407504}


\date{}

\begin{document}

\begin{abstract}  
Two  frameworks that have been used to characterize reflected diffusions
include   stochastic differential equations 
with reflection and the so-called submartingale problem. 
We introduce  a general formulation of the submartingale problem for 
  (obliquely) reflected   diffusions in domains with piecewise ${\mathcal C}^2$ 
  boundaries and piecewise continuous reflection vector fields. 
Under suitable assumptions, 
 we show that well-posedness of the submartingale problem is
equivalent to  existence and uniqueness in law of  weak solutions to
the corresponding stochastic differential equation with reflection.   
Our result generalizes  to the case of reflecting diffusions a
classical  result due to Stroock and Varadhan on the equivalence of 
well-posedness of martingale problems and well-posedness of weak solutions of stochastic differential
equations in $d$-dimensional Euclidean space.  
The analysis in the case of  reflected diffusions in domains with
non-smooth boundaries 
 is considerably more subtle and requires a careful 
analysis of the behavior of the reflected diffusion on the boundary of
the domain.   
In particular, the equivalence can fail to hold when our assumptions are not satisfied.
 The equivalence we establish allows one to transfer results on 
reflected diffusions characterized by one approach to reflected
diffusions analyzed by the other approach.  
As an application,  we  provide a characterization of  
stationary distributions of  a large class of reflected diffusions in 
convex polyhedral domains.  
\end{abstract}

\maketitle

\tableofcontents

\beginsec

\section{Introduction}

\subsection{Background and Motivation}

A reflected diffusion in a non-empty, connected domain $G$ with a vector field
$d(\cdot)$ on the boundary $\partial G$ and measurable drift and
dispersion coefficients $b: \bar{G} \mapsto \R^J$ and
$\sigma: \bar{G} \mapsto \R^{J \times J}$ defined on the closure
$\bar{G}$ of the domain is a continuous Markov 
  process that, roughly speaking, behaves like a
  diffusion with (state-dependent) drift $b(\cdot)$ and dispersion $\sigma(\cdot)$
  inside  the domain and that is restricted to stay in $\bar G$ by a constraining force that is only allowed
to act along the  directions specified by the vector field on the
boundary.  For historical reasons, this constrained process is 
referred to as a reflected diffusion.  Two approaches to  providing a precise mathematical
characterization of this intuitive description
 are the framework of stochastic differential equations with
 reflection (SDER), which is used, for example, in 
\cite{Sko61}, \cite{LioSzn84}, \cite{Cos92},
\cite{DupIsh93},  \cite{BasHsu00} and \cite{Ram06}, and  the submartingale problem formulation introduced by Stroock and Varadhan in
\cite{StrVar71}. 
These two approaches are respective generalizations of the 
stochastic differential equation (SDE) and martingale problem
formulations commonly used to analyze diffusions in
$\R^J$.  
  In the case of (unconstrained)  
diffusions,  under fairly general conditions, there is a well
established equivalence between  existence and
 uniqueness in law of weak solutions to SDEs 
and well-posedness of the martingale problem 
(see \cite{StrVar72}, and also \cite{Kur11} for a recent generalization in the
unconstrained setting).  
  Somewhat surprisingly, there appears to be no such general
  correspondence 
available in the case of obliquely reflected diffusions, particularly 
in non-smooth domains. 
 Such reflected diffusions arise in a variety 
of applications, ranging from queueing theory and  mathematical
finance to the study of random matrices. 
The  goal of the current work is to establish an analogous 
equivalence  (see Theorem \ref{thm:general})  between well-posedness of the submartingale problem and
well-posedness of  the associated stochastic differential equation with reflection (SDER) for  a
general class of semimartingale reflected diffusions in piecewise
smooth domains (see Section \ref{sec-back} for precise 
definitions).  
 The results of this paper are extended in
\cite{KanRam14b} to 
 diffusions that are not necessarily semimartingales, which 
arise in many situations \cite{DeBTob93,DupRam98,RamRei03, RamRei08,KanRam09,BurKanRam09}.

There are several  motivations for establishing such a
correspondence. 
Firstly, the submartingale problem was originally formulated only for smooth
 domains and continuous reflection \cite{StrVar71}.   
 Extensions to domains with non-smooth boundaries had previously been  considered 
only in special cases \cite{VarWil85, Wil87, Kwo92, KwoWil91, DeBTob93}. 
With the exception of
 \cite{Wil87}, in each of these cases,   the boundary of the domain 
 has a single point of non-smoothness. 
  The work  \cite{Wil87} considered the class of skew-symmetric reflected Brownian
 motions (RBMs) in polyhedral domains, which have the special property 
that they  almost surely do not hit the non-smooth parts of the
boundary. 
For general domains with non-smooth boundaries and oblique reflection, even the formulation of the
submartingale problem is somewhat subtle and a correct formulation in
multidimensional non-smooth domains had been  a longstanding open problem 
\cite{Wil-95Surv} (see comment (iii) of Section 4 therein). 
In Definition \ref{def-smg}, we introduce a general formulation 
of the submartingale problem in domains with piecewise smooth
boundaries.   The equivalence result established here provides  validation that
this formulation, which was also used in \cite{KanRam14},
is a reasonable formulation. 

A second motivation arises from the fact that whereas  some properties of reflected diffusions such as 
 existence and uniqueness in law have been established using  the SDER 
framework in \cite{LioSzn84, DupIsh93, TayWil93, Ram06},   other properties
such as  boundary properties and characterizations of  stationary distributions
have been established for reflected diffusions associated
with a well-posed submartingale problem  \cite{Wei81,KanRam14}. 
The equivalence allows one to transfer results proved in one
setting to the other setting.  As an illustration, in  Corollary
\ref{cor-stat} we use our equivalence result to characterize stationary distributions
for a large class of solutions to well-posed SDER. 
In addition, we also apply our result to show that 
(under fairly general conditions) to establish well-posedness of the
 submartingale problem one can without loss of generality assume 
that the drift is zero (see Remark \ref{rem-Girsanov}). 

Even for cases where the submartingale problem was well formulated, 
establishing a correspondence between  the submartingale
problem and weak solutions to SDERs had  been deemed a
challenging problem \cite[p. 149]{DeB96}. 
In the case of non-smooth domains, only 
very special cases seem to have been previously considered, 
such as, for example,  normal reflection in the $d$-dimensional 
nonnegative orthant, which essentially reduces to a one-dimensional
problem (see \cite[Theorem V.1.1]{BasBook97} and
\cite[Proposition 2.1]{BasLav07} for a brief discussion of this
case).   Our main equivalence result, Theorem \ref{thm:general}, is established for reflected
diffusions with a measurable, locally bounded drift and  measurable,  
locally bounded and uniformly elliptic diffusion coefficient in  piecewise
${\mathcal C}^2$ domains with (piecewise) continuous reflection that satisfy 
 a certain geometric
condition, which is a generalization of the 
so-called completely-${\mathcal S}$ condition that is used 
for RBMs in the orthant  \cite{BerElk91}.

 The proof of 
Theorem \ref{thm:general}  follows from two results, established 
in Theorem \ref{th-exi} and Theorem \ref{th:existence}, respectively. 
With a view to the extension to the case of non-semimartingale 
reflected diffusions in \cite{KanRam14b},  these two theorems 
are established in slightly greater generality, where the 
generalized completely-${\mathcal S}$-condition is allowed to fail in 
a certain subset of the boundary of the domain.  
First, in Theorem \ref{th-exi}, 
any weak solution to the SDER is shown to satisfy the submartingale problem. 
This is relatively straightforward, with  some additional arguments 
to deal with the fact that the weak solution need not be a
submartingale. 
The second  step (Theorem \ref{th:existence}),  which constructs 
a weak solution  to an SDER from a  solution to the 
corresponding submartingale problem, is the main step. 
It  is significantly more complicated  than in the unconstrained 
case due to the presence of boundaries and the  geometry of the
directions of reflection. 
In particular,  it involves the use of a
boundary property of solutions to submartingale problems in domains 
with piecewise smooth domains that was established in
\cite{KanRam14} (see Proposition \ref{lem-bdary} here) and 
the construction of 
special test functions that satisfy certain oblique derivative
conditions (see  Lemma \ref{lem:cutoff} and Appendix
\ref{ap-test}). 
Moreover, due to  the multi-valued nature of
the normal and reflection directions at non-smooth points, 
 to identify the local time term,  we use 
an integral representation of the candidate local time
process,  which we establish using functional analytic arguments 
(see Section \ref{subs-integrep}  and  Appendix \ref{apsub-integrep}).

\subsection{Outline of the Paper}

In Section \ref{sec-back} we
recall some basic definitions related to reflected diffusions.  The SDER and 
submartingale formulations and their properties are introduced in Sections
\ref{subs-sder} and \ref{sec:SP}, respectively, and Section 
\ref{subs-bdary} defines the class of domains 
and reflection directions that we consider. 
Section \ref{sec-mainres} contains the main result, 
Theorem \ref{thm:general}, its two auxiliary results 
 Theorem \ref{th-exi} and Theorem \ref{th:existence}, and 
some discussion of the ramifications, including 
Corollary \ref{cor-stat} and Remark \ref{rem-Girsanov}. 
The proof of Theorem \ref{th-exi}  is presented in Section
\ref{subs-red}, and the proof of Theorem \ref{th:existence}, 
which is considerably more involved, is broken down into may steps 
that are presented in  Sections
\ref{subs:locref}--\ref{sec:case2}. 
The proof of some lemmas are deferred to the Appendices. 
First, in the next section we collect some common notation used throughout
the paper.

\subsection{Common Notation}
\label{subs-not}

Let $\R$ denote the set of real numbers and
$\R_+$ is the set of non-negative real numbers. Given $a, b \in \R$, $a
\wedge b$ ($a \vee b$) 
denote the minimum (maximum)  of $a$ and $b$. For each $J\in \N$,
$\R^J$ is the $J$-dimensional Euclidean space and $|\cdot|$ and $\lan
\cdot,\cdot \ran$, respectively,
denote the Euclidean norm  and the inner product on $\R^J$. For each set $A\subset \R^J$, $A^\circ$, $\partial A$, $\bar
A$ and $A^c$ denote the interior, boundary, closure and complement of $A$, respectively. For
each $x\in \R^J$ and $A\subset \R^J$, $\dist(x,A)$ is the distance
from $x$ to $A$, that is, $\dist(x,A)=\inf\{y\in A:\ |y-x|\}$. For
each $A\subset \R^J$ and $r>0$, $B_r(A)=\{y\in \R^J:\ \dist(y,A)<
r\}$, and given $\ve > 0$ let $A^\ve \doteq \{y \in \R^J: \dist(y,A) <
\ve\}$ denote the (open) $\ve$-fattening of $A$.  If $A=\{x\}$, we
simply denote $B_r(A)$ by $B_r(x)$.  We will use $S_1(0)$ to denote
the unit sphere in $\R^J$.
We also let $\ind_A$ denote the indicator function of the set $A$
(that is, $\ind_A (x) = 1$ if $x \in A$ and $\ind_A(x) = 0$
otherwise). Given integers $i, j$ we let $\delta_{ij}$
denote the Kronecker delta function, $\delta_{ij}  = 1$ if $i=j$ and 
$\delta_{ij} = 0$, otherwise.   
Given a set $A \subset \R^J$, let $\conv [A]$ denote the closure of the convex
hull of $A$, which is defined to be the intersection of all closed convex sets that
contain $A$.   

Given  a domain  $E$ in $\R^n$, for some $n \in \N$,  let $\C(E) = \C^0(E)$
be the space of continuous real-valued
functions on $E$ and, for any
$m\in \Z_+\cup\{\infty\}$, let $\C^m(E)$ be the subspace of
functions in $\C(E)$ that are  $m$ times continuously
differentiable on $E$ with continuous partial derivatives of
order up to and including $m$.  When $E$ is the closure of a domain, $\C^m(E)$ is to be interpreted as
the collection of  functions in $\cap_{\ve > 0} \C^m
(E^\ve)$, where $E^\ve$ is an open $\ve$-neighborhood of $E$, restricted to $E$. Also, let $\C_b^m(E)$ be
 the subspace of $\C^m(E)$ consisting of bounded functions whose partial derivatives of order up to
and including $m$ are also bounded, let $\C_c^m(E)$ be the subspace of
$\C^m(E)$ consisting of functions that vanish outside compact sets.
In addition, let $\C_c^m(E)\oplus\R$ be the direct sum of $\C_c^m(E)$
and the space of constant functions, that is,
the space of functions that are sums of functions in $\C_c^m(E)$ and
constants in $\R$.  
The support of a function $f$ is denoted by $\supp(f)$,  its gradient
of $f$ is denoted by $\nabla f$. For $m\geq 1$  and a sequence of random variables $\{X_n,\ n\geq 1\}$ defined on some common probability space $(\Omega, \MF,\Q)$, we say $X_n$ converges to $X$ in $\mathbb{L}^m(\Q)$ as $n\rightarrow \infty$ if $\E^{\Q}[|X_n-X|^m]\rightarrow 0$ as $n\rightarrow \infty$. 
For $f \in {\mathcal C}^\infty (I)$, with $I \subset \R$ or $f \in
{\mathcal C}^2 (G)$, we let $||f||_\infty$  denote the supremum of 
$f$ on its domain.   

\beginsec

\section{Characterizations of Reflected Diffusions}
\label{sec-back}

Throughout, let $G$ be a nonempty connected domain in $\R^J$ and let $d(\cdot)$ be
a set-valued mapping defined on the closure $\bar G$ of $G$ such that $d(x) = \{0\}$ for $x \in G$,
 $d(x)$ is a non-empty, closed and convex cone in $\R^J$ with vertex at the origin for every
$x$ in $\partial G$,  and the graph of $d(\cdot)$ is closed, that is, the set
$\{(x,v): x \in \bar{G}, v \in d(x) \}$ is a closed subset of
$\R^{2J}$.  
Let $b: \R^J \mapsto \R^J$ and $\sigma: \R^J \mapsto \R^{J \times N}$
be measurable and locally bounded. 
Also,  denote the set of inward normals to $G$ at a point $x\in \partial G$ by
$n(x)$, assuming this is well-defined,  and let 
\be
\label{def-mv}
\MV\doteq \partial G \setminus \MU,
\ee where  $\MU$ is the subset of the boundary $\partial G$ defined by
 \be
\label{def-mu}
  \MU \doteq \{x \in \partial G: \exists\ n \in n(x)  \mbox{ such that }
  \lan n, d \ran > 0 ,\ \forall\ d \in d(x)\sm \{0\} \}. 
\ee
The set $\MV$  will play an important role in the analysis. 
Also, let $\ML$ be  the usual associated second-order  differential operator
\be \label{operL}
\ML f(x)\doteq\sum_{i=1}^J b_i(x)\frac{\partial f}{\partial
x_i}(x) +\frac{1}{2}\sum_{i,j=1}^J a_{ij}(x)\frac{\partial^2f}{\partial
x_i\partial x_j}(x), \qquad f\in \C^2_b(\bar G),
\nonumber \ee
where $\C^2_b(\bar G)$ is the space of twice continuously
differentiable functions on $\bar G$ that, along with their first
and second partial derivatives, are bounded.

We recall the definition of weak solutions to stochastic
differential equations with reflection associated with $(G,d(\cdot))$,
$b$ and $\sigma$ and some of their properties in
Section \ref{subs-sder}. 
We then 
introduce the formulation of the associated submartingale problem in Section
\ref{sec:SP}. 
 Lastly,  in Section \ref{subs-bdary} we 
describe the specific class of piecewise continuous domains
$(G,d(\cdot))$ of interest 
and then state a useful boundary property of 
 reflected diffusions in this class of domains that was established in
\cite{KanRam14}.

\subsection{Stochastic Differential Equations with Reflection}
\label{subs-sder}

The Skorokhod Problem (SP), which was introduced in one dimension by  
\cite{Sko61} and subsequently extended to higher dimensions by
numerous authors \cite{BerElk91,Cos92,DupIsh91,LioSzn84},  
and the  extended Skorokhod problem (ESP) introduced in \cite{Ram06}, 
are convenient tools for the pathwise construction of reflected
diffusions. 
Roughly speaking, given a continuous path $\psi$, the ESP associated with $(G,d(\cdot))$
produces a constrained version
$\phi$ of $\psi$ that is restricted to live within $\bar G$ by adding to it a ``constraining term''
$\eta$ whose increments over any interval lie in the closure of the
convex hull of the union of the allowable directions $d(x)$
at the points $x$ visited by $\phi$ during this interval.
   Let  $\ccspace = \cc$ 
denote the space of  continuous functions from $[0,\infty)$ to
$\R^J$,  equipped with the topology of uniform convergence on compact
sets.  We now rigorously define the  ESP.

\begin{defn}[{\bf Extended Skorokhod Problem}]
\label{def-esp}
Suppose $(G,d(\cdot))$ and  $\psi \in \ccspace$ with $\psi(0)\in \bar G$ are given.
Then the pair $(\phi, \eta) \in \ccspace \times \ccspace$
is said to solve the extended Skorokhod Problem (ESP) for $\psi$
  if $\phi(0) = \psi (0)$,
 and if for all
$t \in [0, \infty)$, the following properties hold:
\begin{enumerate}
\item[1.]
$\phi(t) = \psi (t) + \eta (t)$;
\item[2.]
$\phi (t) \in \bar G$;
\item[3.]
For every $s \in [0, t)$, 
\be
\label{hullprop}
 \eta (t) - \eta(s)  \in \conv \left[ \bigcup_{u \in [s,t]} d(\phi(u)) \right].
\ee
\end{enumerate}
If $(\phi,\eta)$ is the unique solution to the ESP for $\psi$,
then we write $\phi = \Gamma (\psi)$, and refer to
$\Gamma$ as the extended Skorokhod map (ESM).
\end{defn}

 The formulation of the ESP in Definition \ref{def-esp} appears slightly
 different from the original one given in \cite{Ram06} since the ESP
 in \cite{Ram06} was formulated more generally for c\`{a}dl\`{a}g
 paths.  However, as we show below, they coincide for the case of continuous paths, 
which is all that is required for this work.   Indeed,  
for continuous paths, property 4 of Definition 1.2 of \cite{Ram06} holds
automatically, and the following lemma shows that 
property 3 of Definition 1.2  of \cite{Ram06} is  equivalent to
property 3 in Definition \ref{def-esp}.

\begin{lemma}
\label{lem-esp} For every $t \in (0,\infty)$, 
Property (\ref{hullprop}) in Definition
\ref{def-esp} is satisfied if and only if  the following condition
holds: for every $s \in [0, t)$, 
\be
\label{hullprop1}
 \eta (t) - \eta(s)  \in \conv \left[ \huge\cup_{u \in (s,t]} d(\phi(u)) \right].
\ee
\end{lemma}
\begin{proof}  Note that the only  difference between \eqref{hullprop} 
and \eqref{hullprop1} is that the left bracket in $(s,t]$ is open in
\eqref{hullprop1} and closed in \eqref{hullprop}, and thus 
property (\ref{hullprop1}) trivially implies property
\eqref{hullprop}. 
To prove the converse, suppose that 
(\ref{hullprop}) holds. Let $\{s_n,\ n\in \N\}$ be a sequence of real
numbers such that $s<s_n<t$ for each $n\in \N$ and $s_n\rightarrow s$
as $n\rightarrow \infty$.  For each
$n\in \N$, by (\ref{hullprop}) we have 
\[\eta (t) - \eta(s_n)  \in \conv \left[ \bigcup_{u \in
    [s_n,t]} d(\phi(u)) \right]\subseteq \conv \left[ \bigcup_{u \in
    (s,t]} d(\phi(u)) \right].\] Together with the continuity of $\eta$ and
the closedness of the set $\conv \left[ \huge\cup_{u \in (s,t]}
  d(\phi(u)) \right]$, this  implies that the property (\ref{hullprop1})
holds. 
\end{proof}
\bigskip

\begin{remark}
\label{rem-esp}
{\em 
Given $(G,d(\cdot))$ and $\psi$ as in Definition \ref{def-esp}, 
a pair $(\phi,\eta)$  $ \in {\mathcal C} \times {\mathcal C}$ is said
to solve the Skorokhod Problem (SP) for $\psi$ if it satisfies
properties 1 and 2 of Definition \ref{def-esp} and, in addition,
$\eta$ has finite variation on bounded intervals and,  
in addition, there exists a Borel measurable function $\gamma:[0,\infty) \mapsto S_1(0)$ such that
for every $t \in [0,\infty)$,  
\be
\label{etaprop}
 \eta(t) = \int_{[0,t]} \gamma(s) \ind_{\{ \phi(s) \in \partial G\}}
d|\eta|(s),
\nonumber \ee
where $\gamma (s) \in d(\phi(s))$ for $d|\eta|$ almost every $s \in [0,\infty)$, and
$|\eta|(t)$ represents the total variation of $\eta$ on the
interval $[0,t]$ (see, \cite{DupIsh91}).
The ESP is a generalization of the SP that does not
{\em a priori} require the constraining term $\eta$ to have finite
variation, and hence allows for the construction of reflected
diffusions that are not necessarily semimartingales (see Lemma \ref{lem-spesp}
for an elaboration of this point).  
However, it was  shown in Theorem 1.3 of \cite{Ram06} that if the
solution
$(\phi,\eta)$ to the ESP for some $\psi$  is such that $\eta$ has finite variation on
every interval $[0,t]$, then $(\phi, \eta)$ is a solution to the SP
associated with $\psi$.  Sufficient  conditions  for the existence of 
a unique solution to the SP or ESP can be found, for example, in 
\cite{BerElk91,BurKanRam09,LehKruRamShr07,DupIsh91,DupRam99a,DupRam99b,Ram06}.
}
\end{remark}

The ESM  can be used to define solutions to
stochastic differential equations with reflection (SDERs) associated with a given
pair $(G,d(\cdot))$,  and drift and dispersion coefficients $b:\bar{G}
\mapsto \R^J$ and $\sigma: \bar{G} \mapsto \R^{J\times N}$.

\begin{defn}[{\bf Weak Solution}]
\label{def-SDER}
Given  $z  \in \bar G$, a weak solution to the SDER with initial condition $z$ associated with 
 $(G,d(\cdot))$, $b(\cdot)$ and $\sigma (\cdot)$ is a 
triplet $(\Omega, {\cal F}, \{{\cal F}_t\})$,  $\P_z$, $(Z,W)$, where 
$(\Omega, {\cal F}, \{{\cal F}_t\})$  is a filtered space that supports a 
probability measure  $\P_z$ and $(Z,W)$ is a pair of 
continuous, $\{{\cal F}_t\}$-adapted $J$-dimensional processes 
with the following properties that under $\P_z$, 
\begin{enumerate}
\item[1.]
$\{W_t, {\cal F}_t, t \geq 0\}$ is an adapted $J$-dimensional standard Brownian
motion;  
\item[2.] $\P_z \left( \int_0^t \left| b(Z(s)) \right| \, ds +
\int_0^t \left| \sigma (Z(s)) \right|^2 \, ds < \infty \right) = 1,$
\quad  $t \in [0,\infty)$;
\item[3.] there exists a continuous $\{{\cal F}_t\}$-adapted process $Y$ such that  
$\P_z$-almost surely,  $(Z,Y)$ solves the ESP associated with
$(G,d(\cdot))$ for $X$, where
 \be
\label{def-x}
\qquad  X(t) \doteq z + \int_0^t b(Z(s)) \, ds + \int_0^t \sigma
(Z(s)) \, d W(s), \quad  t \in [0,\infty);
\ee
\item[4.] $\P_z$-almost surely, the set $\{t:\ Z(t)\in \partial G\}$ has zero Lebesgue
  measure.  In other words, $\P_z$-almost surely, 
\be \int_0^\infty
  \ind_{\partial G}(Z(s))\,ds = 0.
\nonumber \ee
\end{enumerate}
We say that $(\Omega, {\cal F}, \{{\cal F}_t\})$,  $\{\P_z, z \in
\bar{G}\}$, $(Z,W)$ 
is a weak solution to the SDER associated with 
 $(G,d(\cdot))$, $b(\cdot)$ and $\sigma (\cdot)$  if for each $z\in
 \bar{G}$, $(\Omega, {\cal F}, \{{\cal F}_t\})$,  $\P_z$, $(Z,W)$ is 
  \end{defn}

Note that property 3 of Definition \ref{def-SDER} and the definition
of the ESP imply that under $\P_z$, $Z$
satisfies 
\be \label{def-z} 
Z(t) = X(t) + Y(t) = z + \int_0^t b(Z(s)) \, ds + \int_0^t \sigma (Z(s)) \, d
W(s)+Y(t),\quad  t \in [0,\infty). \ee 

The definition of uniqueness in law for the SDER is analogous to the
case of an SDE. 

\begin{defn}
\label{eq-weakunique}
Uniqueness in law is said to hold for 
the SDER associated with
$(G,d(\cdot))$, $b(\cdot)$ and $\sigma(\cdot)$ if for each $z\in \bar{G}$, given  any two
weak solutions $(\Omega, {\cal F}, \{{\cal F}_t\})$, $\P_z$, $(Z,W)$, and $(\tilde \Omega, \tilde{\cal F}, \{\tilde{\cal F}_t\})$,
 $\tilde \P_z$, $(\tilde Z, \tilde W)$, of the SDER with initial
 condition $z$, 
   the law of $Z$ under $\P_z$ is the same as the law of $\tilde{Z}$ under $\tilde \P_z$.
\end{defn}

We now define well-posedness of the SDER. 

\begin{defn}[{\bf Well-posedness of the SDER}]
\label{SDERu}
The SDER 
associated with $(G,d(\cdot))$,  drift $b(\cdot)$ and
dispersion $\sigma(\cdot)$ is said to be  well posed if  there exists a
weak solution to the SDER and uniqueness in law holds for the SDER. 
\end{defn}

Well-posedness of weak solutions (in some cases even existence of strong
solutions) to SDERs  have also been established in many classes of polyhedral and piecewise smooth domains
(see, e.g., \cite{Cos92,DaiWil95,DupIsh91, DupIsh93, DupRam99a,DupRam99b,LioSzn84, Ram06,TayWil93}). 
As mentioned above, in general, weak solutions defined via the ESP  need not be
semimartingales when $\MV \neq \emptyset$ 
(see \cite{BurKanRam09,DupRam00,KanRam09,Ram06,RamRei03,RamRei08} for examples
where $\MV \neq \emptyset$ and such  non-semimartingales arise).    
However, we now make some observations on the link between weak solutions
and the semimartingale property, which will be used in our subsequent
analysis. 

\begin{lemma}
\label{lem-spesp}
Given $z \in \bar{G}$,  
let $(\Omega, {\cal F},\{{\cal F}_t\})$, $\P_z$, $(Z,W)$, 
be a weak solution of the SDER with initial condition 
$z$, and let $X, Y$ be defined as in Definition \ref{def-SDER}.  Let $\theta_1, \theta_2$ be 
two $\{{\cal F}_t\}$-stopping times such that $\theta_1$ is
$\P_z$-almost surely finite and $\theta_2 \geq \theta_1$, and consider 
the  shifted and stopped processes  
\be
\label{def-tildey}
 \tilde{Y} (\omega, u) \doteq  Y(\omega, (\theta_1 (\omega) + u) \wedge
 \theta_2(\omega)) - Y(\omega, \theta_1(\omega)), \quad   \omega \in \Omega, u \in
 [0,\infty)
\ee
and 
\be
\label{def-tildez}
\tilde{Z} (\omega, u) \doteq Z (\omega, (\theta_1 (\omega)+ u) \wedge
 \theta_2(\omega)),    \quad   \omega \in \Omega, u \in
 [0,\infty). 
\ee
If $\tilde Z(\omega,t) \not \in {\mathcal V}$ for all $t \in [0,
\theta_2 (\omega)-\theta_1 (\omega)]$ 
(which should be interpreted as $t \in [0,\infty)$
 when $\theta_2 (\omega) = \infty$),   
then the total variation of $\tilde{Y}$ is $\P_z$-almost surely finite on every bounded
interval,  and 
there exists a measurable function $\gamma : (\Omega \times \R_+, {\mathcal F}
\times \MB (\R_+)) \mapsto (\R^J, \MB (\R^J))$ such that 
for each $\omega \in \Omega$,  and $0 \leq s \leq t < \infty$, 
\be
\label{eq-espmeas}
\tilde Y(\omega, t)-\tilde Y(\omega, s) = 
\int_{[s,t]}
   \gamma(\omega, u) d |\tilde Y|(\omega,u), 
\ee
and for $\P_z$-almost surely every $\omega$,  $\gamma (\omega, u) \in
d (\tilde{Z}(\omega, u))$  for $d|\tilde Y|(\omega)$-almost every $u \in
[0,\infty)$.  
\end{lemma}

The main new content of this lemma, beyond the pathwise property stated 
in Remark \ref{rem-esp}, is the claim that the process $\gamma$
can be chosen to be jointly measurable in $\omega$ and $t$.  The proof of the lemma is relegated to Appendix 
 \ref{ap-esp}.

We close this section with an observation that will allow 
us to assume without generality that the Brownian motion driving 
the weak solution is $J$-dimensional. 
Recall that the diffusion coefficient $a(\cdot) = \sigma (\cdot) \sigma^T(\cdot)$ is uniformly
elliptic if there exists $ \bar{a} > 0$ such that
\be
\label{eq-ue}
 v^T a(x) v \geq \bar{a} |v|^2 \qquad \mbox{ for all }
v \in \R^J, x \in \bar G.
\ee

 \begin{remark} \label{exist}
{\em  Under the uniform
ellipticity condition (\ref{eq-ue}), 
 existence of a weak solution with an  $\R^{J \times N}$-valued
 dispersion coefficient 
$\sigma (\cdot)$ is equivalent to existence of a weak solution with the 
$\R^{J \times J}$ dispersion coefficient $a^{1/2} (\cdot)$. 
Indeed, for each $z\in \bar{G}$, existence of a weak
solution $(\Omega, {\mathcal F},
\{{\mathcal F}_t\})$, $\P_z$, $(Z, W)$ 
 to the
SDER with initial condition $z$ associated with $(G,d(\cdot))$, $b(\cdot)$ and $\sigma
(\cdot) \in \R^{J \times N}$ implies that 
$(\Omega, {\mathcal F},
\{{\mathcal F}_t\})$, $\P_z$, $(Z, B)$ 
is also a weak solution to the same SDER with initial condition $z$, 
 where, under $\P_z$, $B$ is the $J$-dimensional standard Brownian motion defined by
\be
\label{trans-B}
 B(t) \doteq\int_0^t a^{-1/2}(Z(u))\sigma(Z(u))\, dW(u),  \qquad t \geq 0.
\nonumber \ee   
Conversely, using an argument similar to that used in Proposition 4.6 in  Chapter 5 of
\cite{KarShrBook},  
it can be shown that for each $z\in \bar{G}$, given any weak solution $(\Omega, {\mathcal F},
\{{\mathcal F}_t\})$, $\P_z$, $(Z, W)$ to the SDER with initial condition $z$ 
 associated with 
$(G,d(\cdot))$, 
 $b(\cdot)$ and $a^{1/2}(\cdot) \in \R^{J \times J}$ 
(so that  now $W$ is an  associated $J$-dimensional standard Brownian motion under
 $\P_z$), 
there exists a possibly   extended filtered space
$(\tilde \Omega, {\tilde{\cal F}}, \{{\tilde{\cal F}}_t\})$, a
probability measure $\tilde \P_z$ on that space and 
$(\tilde Z, \tilde W)$, that  is a weak solution to the SDER with initial condition $z$ 
associated with $(G,d(\cdot))$, 
 $b(\cdot)$ and $\sigma(\cdot)$, and where
under $\tilde \P_z$, $\tilde W$ is an $N$-dimensional standard
 Brownian motion.
}
\end{remark}

\subsection{The Submartingale Problem} \label{sec:SP}

Let $(G,d(\cdot))$, $b(\cdot)$, $\sigma(\cdot)$, $\MV$ and 
${\mathcal L}$  be as defined at the beginning of 
Section \ref{sec-back}.  
We first introduce a class of test functions  that arises in the
formulation  of the submartingale problem. 
 Recall that  $\C_c^2(\bar{G})\oplus\R$ be the
space of functions that are sums of functions in  $\C_c^2(\bar{G})$
and constants in $\R$.  Define 
\be \label{dis:H}
\MH \doteq  \ds \left\{\begin{array}{ll} f\in \C_c^2(\bar G)\oplus \R:\ & \mbox{$f$ is constant in a neighborhood of $\MV$}, \\ & \mbox{$\left<d, \nabla
f(y)\right>\geq 0$ for $d\in d(y)$ and $y\in \partial
G$} \end{array}\right\}, \nonumber \ee
where for each function $f$ defined on $\R^J$, we say $f$ is constant in a
neighborhood of $\MV$ if for each $x \in \MV$, $f$ is constant
in some open neighborhood of $x$.  
When $\MV$ is the empty set, the condition that $f$ be
constant in a neighborhood of $\MV$ is understood to be void.  

\begin{remark}
\label{rem-MH}
{\em  The following useful property of $\MH$ was established in 
  Lemma 5.2 of  \cite{KanRam14}:  $\MH$ has a countable subset $\MH_0$ with the property that
for each $f\in \MH$ and each $N \in \N$ such that $B_N(0)$ contains
both an open neighborhood of $\MV$ and an open neighborhood of
$\supp(f-f(\infty))$, there exists a sequence $\{g_k:\ k\in
\N\}\subset \MH_0$ such that \be \label{fgk} \lim_{k\rightarrow
  \infty}\sup_{y\in \bar G \cap
  B_N(0)}\max_{i,j=1}^J|f(y)-g_k(y)|\vee\left|\frac{\partial
    f(y)}{\partial x_i}-\frac{\partial g_k(y)}{\partial
    x_i}\right|\vee \left|\frac{\partial^2 f(y)}{\partial x_i\partial
    x_j}-\frac{\partial^2 g_k(y)}{\partial x_i\partial x_j}\right|=
0. \ee
}
\end{remark}

We now define the submartingale problem associated with the data
$(G,d(\cdot))$, $\MV$, $b(\cdot)$ and $\sigma(\cdot)$. 
Recall that $\ccspace = \cc$ 
denote the space of  continuous functions from $[0,\infty)$ to
$\R^J$,  equipped with the topology of uniform convergence on compact
sets.   Let $\MM$ be the associated Borel $\sigma$-algebra, which is 
generated by sets of the form $\{ \omega \in \ccspace: \omega (t) \in
A\}$ for $t \in [0,\infty)$ and $A \in {\mathcal B} (\R^J)$.
We equip the measurable space $(\ccspace, \MM)$ with the filtration 
$\{\MM_t\}$, where for $t \in [0,\infty)$,  $\MM_t$ is the smallest $\sigma$-algebra with respect 
to which the map $\omega \in \ccspace \mapsto \omega (s) \in \R^J$ is 
measurable for every $s \in [0, t]$. 

\begin{defn}[{\bf Submartingale Problem}] \label{def-smg}
Given $z \in \bar G$,  a probability measure $\Q_z$ on the
measurable space $(\ccspace,\MM)$ is a solution to the submartingale
problem starting from $z$ 
associated with $(G,d(\cdot))$, $\MV$, drift $b(\cdot)$ and
dispersion $\sigma(\cdot)$ if   for each $A\in \MM$,
the mapping $z \mapsto \Q_z(A)$ is $\MB(\bar G)$-measurable and $\Q_z$ satisfies the following four properties:
\begin{enumerate}
\item[1.] $\Q_z(\omega(0)=z)=1$;
\item[2.] $\Q_z(\omega(t)\in \bar G \mbox{ for  every }  t \in
  [0,\infty))=1$;
\item[3.] For every $f\in \MH$, the process
\be
\label{term-mart}
f(\omega(t))-\int_0^t\ML
f(\omega(u))\,du,   \quad t \geq 0,
\nonumber \ee
is a $\Q_z$-submartingale on
$(\ccspace,\MM,\{\MM_t\})$;
\item[4.]  $\Q_z$-almost surely, $\int_0^\infty \ind_{\MV}(\omega(u))\,du=0.$ 
\end{enumerate}
 A family $\{\Q_z, z\in \bar G\}$ of probability measures on $(\ccspace,\MM)$ is a solution to the submartingale
problem if for each $z \in \bar G$, $\Q_z$ is a solution to the
submartingale  problem starting from $z$.  
\end{defn}

\begin{defn}[{\bf Well-posedness of the submartingale problem}]
The submartingale problem
associated with $(G,d(\cdot))$, $\MV$, drift $b(\cdot)$ and
dispersion $\sigma(\cdot)$ is said to be well posed if  there exists
exactly one solution to the  submartingale problem.
\end{defn}

Definition \ref{def-smg} differs slightly from past formulations 
of the submartingale problem in domains with non-smooth boundaries. 
As mentioned in the introduction, essentially all these works 
\cite{VarWil85, Wil87, Kwo92, KwoWil91, DeBTob93} 
consider domains 
that have only a single point of non-smoothness on the boundary and  the 
formulation they use is Definition \ref{def-smg}, but with $\MV$
replaced by the set of non-smooth points on the boundary.  
In either formulation,  since the test functions in property 3 are required to
be constant in a neighborhood of some subset of the boundary, 
an additional property (property 4) needs to be imposed to ensure 
that the reflected diffusion spends zero Lebesgue time on the
boundary.   For the class of domains we consider, it  is shown in 
Proposition \ref{prop-bdary} that 
any solution to the submartingale problem formulated as in Definition \ref{def-smg} 
spends zero Lebesgue time on the boundary.    On the other hand, since
$\MV$ is typically a subset of the non-smooth part of the domain, 
property 3  in
our formulation  has to be satisfied by a larger class of test functions and
hence, it is {\em a
  priori} easier to establish uniqueness and harder to establish  
existence of solutions. 
However, in the cases studied previously, it seems not much harder to
establish existence of solutions for  our formulation of the 
 submartingale problem. 
For example,  for the 
two-dimensional wedge considered in
\cite{VarWil85},  the two formulations coincide when $\MV = \{0\}$, which is
precisely the case when the parameter $\alpha$  in 
\cite{VarWil85} satisfies $\alpha \geq 1$. 
When $\alpha < 1$,  $\MV = \emptyset$, the formulations 
are different and so existence to 
a solution to the submartingale 
problem in Definition \ref{def-smg} does not follow directly from the
results in \cite{VarWil85}.  Nevertheless,  it can be deduced using 
similar arguments or, alternatively, by applying 
 Theorem \ref{th-exi} in conjunction with 
the results of \cite{TayWil93}. 
We believe our formulation  is more convenient for
obtaining results in general piecewise smooth
domains (that have more than one non-smooth point on the boundary). 
In particular, this formulation was used to obtain a 
 characterization of  stationary distributions of a large 
class of reflected diffusions in domains with piecewise smooth 
boundaries in \cite{KanRam14}.    As discussed in \cite{KanRam14b},
the correction formulation is even more subtle for multi-dimensional
domains whose $\MV$ sets have more 
complex geometries.

\subsection{A Class of Domains with Piecewise Smooth Boundary}
\label{subs-bdary}

 We now introduce the general  class of domains and reflection
 directions $(G,d(\cdot))$ covered by our results.    

\begin{defn}[{\bf Piecewise ${\cal C}^2$ with Continuous Reflection}]
\label{ass:G}
The pair $(G,d(\cdot))$ is said to be piecewise ${\cal C}^2$ with
continuous reflection if  it satisfies the following
properties: 
\begin{enumerate}
\item[1.]
$G$ is a non-empty domain in $\R^J$ with representation 
\be \label{G} G=\bigcap_{i\in \MI}G^i, \nonumber \ee 
where $\MI$ is a finite index set and for each $i\in \MI$, $G^i$ is a non-empty domain
with $\C^2$ boundary in the sense that for each $x\in \partial G$,
there exist a neighborhood
${\mathcal O}_x$ of $x$, and functions $\varphi_x^i \in {\cal C}^2(\R^J)$, $i\in
\MI(x)\doteq \{i \in \MI:\ x\in \partial G^i\}$, 
such that
\be
\label{gi-rep}
  {\mathcal O}_x\cap G^i = \{ z \in {\mathcal O}_x: \varphi^i_x (z) > 0 \}, \quad
  {\mathcal O}_x\cap \partial G^i
= \{ z \in {\mathcal O}_x: \varphi^i_x (z) = 0 \},
\nonumber \ee
and $\nabla \varphi_x^i\neq 0$ on ${\mathcal O}_x$.
 For each  $x \in \partial G^i$ and $i \in \MI(x)$, let 
\[ n^i(x) \doteq \frac{\nabla \varphi_x^i(x)}{|\nabla \varphi_x^i(x)|} \]
 denote the unit inward normal
 vector to $\partial G^i$ at $x$. 
\item[2.]
The direction vector field $d(\cdot):\bar{G} \mapsto \R^J$ is given by
$d(x) = 0$ if $x \in G$ and 
\be \label{def-d}  d(x)=\left\{\sum_{i\in \MI(x)}\coeff_i d^i(x):\
  \coeff_i\geq 0, i\in \MI(x) \right\},  \qquad x \in \partial G, \nonumber \ee
where, for each $i \in \MI$, 
$d^i(\cdot)$ is a continuous unit vector field defined on
$\partial G^i$  that satisfies 
 \[\lan n^i(x), d^i(x) \ran > 0 \mbox{ for each $x \in \partial
G_i$.} \]
\end{enumerate}
If $d^i(\cdot)$ is
constant for every $i \in\MI$, then the pair $(G,d(\cdot))$ is said to be piecewise ${\cal C}^2$ with
constant reflection. If, in addition, $n^i(\cdot)$ is
constant for every $i \in\MI$, then the pair $(G,d(\cdot))$ is said to be polyehdral with
piecewise constant reflection.  
\end{defn}

 Note that, with the definition given above, the set of inward normal
 vectors to $G$ takes the form 
 \be
\label{def-n}
n(x)=\left\{\sum_{i\in \MI(x)}\coeff_i n^i(x):\ \coeff_i\geq 0,\ i\in
   \MI(x)\right\}, \qquad x \in \partial G.
\nonumber \ee

Since $(G,d(\cdot))$ is piecewise ${\cal C}^2$ with
continuous reflection, it can readily be verified that $\MU$ is relatively open to $\partial G$ and hence $\MV$ is a closed set. 
We now state a boundary property established 
in Proposition 6.1 of \cite{KanRam14}, which extends 
results established in \cite{ReiWil88} for RBMs in polyhedral
domains.   Note that in \cite{KanRam14} the set $\MV$ in the submartingale
problem was allowed to be any 
arbitrary subset of $\partial G$ and Proposition 6.1 of
\cite{KanRam14} was established under the condition that this set 
$\MV$ satisfies $\partial G \sm \MU
\subseteq \MV$, which is in particular satisfied by the $\MV$
specified in \eqref{def-mv}.  

\begin{prop}[{\bf Boundary Property}]
\label{lem-bdary}
Suppose that $(G,d(\cdot))$ is a piecewise ${\mathcal C}^2$ domain with
continuous reflection,  $b(\cdot)$, $\sigma(\cdot)$ are measurable and
locally bounded,  $a = \sigma \sigma^T$ is uniformly elliptic 
and  the submartingale problem associated with $(G,d(\cdot))$, 
$b(\cdot)$ and $\sigma(\cdot)$ is well posed.  If $\{ \Q_z, z \in
\bar{G}\}$ is a solution to the associated submartingale problem then 
for each $z \in \bar{G}$, we have 
\be
\label{prop-bdary}
\int_0^\infty \ind_{\partial G} (w(u)) \, du = 0,  \quad \Q_z\mbox{-almost surely. }
\ee
\end{prop}

\beginsec

\section{Main Results}
\label{sec-mainres}

We now state our main results.  Recall that 
we will assume throughout, without always stating this explicitly, 
that the drift and dispersion coefficients are measurable and 
locally bounded,  and  that the diffusion coefficient is uniformly elliptic, 
that is, \eqref{eq-ue} holds for some $\bar{a} > 0$.

\begin{theorem}\label{thm:general}
Suppose $(G,d(\cdot))$ is piecewise
$\C^2$ with continuous reflection and $\MV = \emptyset$.
Then the SDER associated with $(G,d(\cdot))$, $b(\cdot)$ and
$\sigma(\cdot)$ is well posed if and only if the corresponding 
submartingale problem is well posed. 
\end{theorem}

Suppose   $G = \R_+^J$ and
$d(x)$ is equal to the vector $d^j$
when $x$ is in the
relative interior of the face $\bar G \cap  \{x: x_j = 0 \}$. 
 Then the condition   $\MV = \emptyset$, 
with $\MV$ defined by \eqref{def-mv}, 
is equivalent to the condition that the so-called reflection matrix
$[d^i_j]_{i,j \in \{1, \ldots, J\}}$ is 
completely-${\cal S}$  (see \cite{BerElk91}).  
Given constant drift and dispersion coefficients 
$b$ and $\sigma$,  for different classes of polyhedral domains with piecewise constant 
reflection $(G,d(\cdot))$, it was shown in   
\cite{DaiWil95,ReiWil88,TayWil93} that the condition 
 $\MU= \partial G$ is sufficient for well-posedness of the associated SDER, and is 
 also necessary for existence of a weak solution that is a
 semimartingale. 
For more general $G$ and $d(\cdot)$,  the condition $\MV = \emptyset$
imposed in Theorem \ref{thm:general} can
be viewed as a generalized completely-${\cal S}$ condition, and it
follows 
from Lemma \ref{lem-spesp} that in this case the reflected diffusion
is a semimartingale. 

Theorem \ref{thm:general} is a direct consequence of Theorems
\ref{th-exi} and \ref{th:existence}, which prove slightly more
general results that do not assume that $\MV = \emptyset$. 

\begin{theorem} 
\label{th-exi}
Given $(G,d(\cdot))$, $\MV$, $b(\cdot)$, $\sigma(\cdot)$,  suppose that for  $z\in \bar{G}$, $(\Omega,
{\mathcal F}, \{{\mathcal F}_t\})$, $\P_z$, 
$(Z,W)$ is  a weak solution to the associated SDER with initial condition $z$, 
and let $\Q_z = \P_z \circ Z^{-1}$ denote the law of $Z$ 
on $(\ccspace, {\mathcal M})$ under $\P_z$.  
If $\MV$ is the union of finitely many closed connected sets, 
 then $\Q_z$ is a solution to the corresponding submartingale
 problem starting from $z$. 
Consequently, if  the submartingale problem has at most one solution,  then uniqueness in law
holds for the associated SDER.  
\end{theorem}

The proof of Theorem \ref{th-exi} is given in Section \ref{subs-red}.    
It is essentially a consequence of It\^{o}'s formula; however, since
we also allow weak solutions that are not necessarily semimartingales, 
the proof requires some additional arguments, which use the results on the ESP
 from \cite{Ram06} that are summarized in Lemma  \ref{lem-spesp}.  The more 
substantial result is its (partial) converse, Theorem \ref{th:existence} below.

\begin{theorem} \label{th:existence}
Suppose $(G,d(\cdot))$ is piecewise $\C^2$ with continuous reflection,
and  the submartingale problem associated with $(G,d(\cdot))$, $\MV$,
$b(\cdot)$ and $\sigma(\cdot)$ has a solution $\{\Q_z, z\in \bar
G\}$. Let  
\be
\label{Z-weak}
   Z(\omega,t) = \omega (t),  \quad t \geq 0,\ \omega \in \ccspace, 
\ee
and consider the ${\mathcal M}_t$-stopping time given by 
\be
\label{def-taumv} 
\tau_\MV = \inf \{ t \geq 0:  \omega (t)  \in \MV \}. 
\ee
Then there exists a process $W$ on $({\mathcal C}, {\mathcal M},
{\mathcal M}_t)$ such that  $({\mathcal C}, {\mathcal M},
{\mathcal M}_t)$, $\{\Q_z, z\in \bar
G\}$, $(W(\cdot \wedge \tau_{\MV}), Z(\cdot \wedge \tau_{\MV}))$ is a 
weak solution to the associated SDER. 
\end{theorem}

The proof of the Theorem \ref{th:existence} is given at the end of Section
\ref{subs:locref}.   
 A broad outline of the proof, broken down into several steps, is first provided at the beginning of Section
\ref{sec-converse}, and details of the various steps are presented  
 in Sections \ref{sec-converse}--\ref{sec:case2}.  
We end this section by discussing two simple consequences of Theorem \ref{thm:general}. 

\begin{remark}
\label{rem-Girsanov}
{\em  If $b(\cdot)$ and $\sigma(\cdot)$ satisfy suitable conditions,  then 
 Girsanov's theorem can be used to show that 
well-posedness of the SDER  associated with $(G,d(\cdot))$, $b(\cdot)$ and $\sigma(\cdot)$ is equivalent to well-posedness of  the SDER
associated with $(G,d(\cdot))$, $b\equiv
0$ and $\sigma(\cdot)$.   In other words, under suitable conditions, to show well-posedness of an SDER one can
assume without loss of generality that $b \equiv 0$. 
Due to Theorem \ref{thm:general}, under the same conditions on $b(\cdot)$ and
$\sigma(\cdot)$,  when establishing well-posedness of a submartingale
problem,  one can also without loss of generality assume $b \equiv
0$. 
This can be a very convenient simplification. 
While it may be natural to expect such an equivalence, in the generality we are considering, 
it does  not seem to be straightforward to establish this result directly 
for the submartingale problem without invoking Theorem
\ref{thm:general}  and the corresponding result for weak solutions to
SDER.  
 For example, in \cite{Wil87} this was established by invoking
 the corresponding result for smooth domains and then using an
 approximation argument and the fact that the RBMs almost surely 
do not hit the non-smooth parts of the boundary. 
}
\end{remark}

Theorem \ref{thm:general} also allows us to transfer results that have been
established for solutions to well-posed submartingale problems to 
reflected diffusions characterized as solutions to well-posed SDER. 
In particular, the latter category includes a large  class of
semimartingale reflected Brownian motions that arise 
as diffusion limits of multiclass queueing networks
\cite{TayWil93,Wil-95Surv}.  

\begin{corollary}
\label{cor-stat}
Suppose that $(G,d(\cdot))$ is piecewise $\C^2$ with continuous
reflection and $\MV = \emptyset$. 
Suppose that the SDER associated with $(G,d(\cdot))$,
$b(\cdot)$ and $\sigma(\cdot)$ is well-posed. 
  Then a probability measure $\pi$ on
$\bar{G}$ is a stationary distribution for the associated reflected
diffusion if and only if $\pi (\partial G) = 0$ and 
\[   \int_{\bar{G}} {\mathcal L} f(x) \pi (dx)  \leq 0,  \]
for all $f \in \MH$. 
\end{corollary}
\begin{proof} By Theorem
  \ref{thm:general}, it follows that the submartingale problem associated
  with  $(G,d(\cdot))$,
$b(\cdot)$ and $\sigma(\cdot)$  is well posed.  The
corollary then follows from  Theorem 2 of \cite{KanRam14}. 
\end{proof}

In particular, due to the result of \cite{TayWil93},  Corollary \ref{cor-stat} applies 
to polyhedral domains with piecewise constant reflection 
that satisfy $\MV = \emptyset$ (that is, the completely-${\mathcal S}$
condition). 

\beginsec

\section{Proof of Theorem \ref{th-exi}}
\label{subs-red}

Given $(G,d(\cdot))$, $b(\cdot)$ and $\sigma(\cdot)$, suppose the associated $\MV$ is
the union of finitely many connected sets, and 
 suppose that for each $z\in \bar{G}$, 
$(\Omega, {\cal F}, \{{\cal F}_t\})$, $\P_z$, $(Z,W)$ is a weak
solution to the associated SDER with initial condition $z$, and let
$\Q_z$ be the law of $Z$ induced by $\P_z$.   
Fix  $z \in \bar{G}$. 
Then the definition of the ESP and properties 3--4 of
Definition \ref{def-SDER} together imply that $\Q_z$ satisfies properties 1, 2
and 4 of the submartingale problem associated with $(G,d(\cdot))$, $b$
and $\sigma$.   
Thus, to prove the first assertion of the theorem, it suffices to
show that 
$\Q_z$ also satisfies property 3 of the submartingale problem. 
Fix $f \in \MH$,  and for some $\bar{L}  \in \N$, 
let  $\MV = \cup_{i=1}^{\bar{L}}\MV_i$ be the unique decomposition of $\MV$ into a finite union of its
connected components, each of which is closed. 
 Any  $f\in \MH$  is the sum of a constant and a
function $\tilde{f}$, where $\tilde{f}$ has
compact support and  is constant in a
neighborhood of every point in $\MV$.   
The set  $\MV_i \cap \supp[ \tilde{f}]$ is compact for every $i = 1,
\ldots, \bar{L}$, and hence, a  standard covering argument shows that 
 there exists $\varepsilon>0$ 
such that for each $i = 1, \ldots, \bar{L}$, 
$f$  is constant on $B_\varepsilon(\MV_i)\cap \bar G$. 
We assume without loss of generality that 
 $\ve$ is smaller than the minimum distance
between any two closed sets $\MV_i$ and $\MV_j$,  $i, j = 1,
\ldots, \bar{L}, i \neq j$.  
 
Now, define $\stopoutf_0\doteq 0$ and for $k\in \N$, let  
\begin{eqnarray*}
\label{def-thetan}
\stopinf_k & \doteq & \inf \{t>\stopoutf_{k-1}:\ Z(t)\in
\bar{B}_{\varepsilon/2}(\MV)\},  \\
\label{def-sigman}
\stopoutf_k & \doteq & \inf \{t>\stopinf_k:\ Z(t)\notin
B_\varepsilon(\MV)\},
 \end{eqnarray*}
 where, by  convention, the infimum over any empty set is taken to be
 infinity.    
Since $\bar{B}_{\varepsilon/2}(\MV)$ and $(B_\varepsilon(\MV))^c$ are
closed sets, $\stopoutf_k$ and $\stopinf_k$ are $\{{\cal F}_t\}$-stopping
times.  Since the process $Z$ is continuous, almost surely, $\stopoutf_k,
\stopinf_k\rightarrow \infty$ as $k\rightarrow \infty$. 
For $t\in [0,\infty)$, 
\be
\label{eqn1}
\begin{array}{rcl}
f(Z(t))- f(Z(0)) & =& \ds \sum_{k=1}^\infty
\left[\ind_{\{\stopoutf_{k-1}\leq t\}}(f(Z(t\wedge
  \stopinf_k))-f(Z(\stopoutf_{k-1})))\right. \\ & & \qquad
+\ds \left.\ind_{\{\stopinf_{k}\leq t\}}(f(Z(t\wedge
  \stopoutf_k))-f(Z(\stopinf_{k})))\right] \\ & = & \ds \sum_{k=1}^\infty \left[\ind_{\{\stopoutf_{k-1}\leq t\}}(f(Z(t\wedge \stopinf_k))-f(Z(\stopoutf_{k-1})))\right],
\end{array}
\ee
where the last equality holds because $f$ is constant 
on each $B_\varepsilon(\MV_{i})$ and the continuity of $Z$ implies
that almost surely, $Z$ lies in the $\ve$-neighborhood of exactly one 
connected component of $\MV$ during each interval $[\stopinf_{k},
\stopoutf_k)$.  
 Fix $k\in \N$.   Now, $(Z,Y)$ is a solution to the ESP for $X$,
 where $X$ is defined by (\ref{def-x}),  and 
$Z(t) \not\in \MV$ for $t \in [\stopoutf_{k-1}, \stopinf_k]$.
Therefore,  on the set $\{\stopinf_k < t\}$, applying  Lemma \ref{lem-spesp}  with $\theta_1 =
\stopoutf_{k-1}$, $\theta_2 = \stopinf_k$, $s= 0$ and $t$ replaced by
$t - \stopoutf_{k-1}$, and defining 
$\tilde{Y}^k (u) \doteq Y((\stopoutf_{k-1} + u)\wedge
\stopinf_k) - Y(\stopoutf_{k-1})$ for $u \in [0,\infty)$, 
it follows that there exists a measurable 
function $\gamma^k: \Omega \times \R_+
\mapsto \R^J$ such that  
\[  Y(t \wedge \stopinf_k)  - Y(\stopoutf_{k-1}) =  \tilde{Y}^k (
t-\stopoutf_{k-1}) - \tilde{Y}^k (0) = \int_{0}^{(t - \stopoutf_{k-1})\wedge (\stopinf_k-\stopoutf_{k-1})} \gamma^k(u) d
|\tilde{Y}^k|(u),\]
where $\gamma^k(u)\in d(Z(\stopoutf_{k-1}+u))$ for $d|\tilde{Y}^k|$ almost
every $u$ (where we have replaced  $d(Z( (\stopoutf_{k-1}+u) \wedge
 \stopinf_k))$ by  $d(Z(\stopoutf_{k-1}+u))$  because $d |\tilde{Y}^k|(u) = 0$ for $u >
 \stopinf_k-\stopoutf_{k-1}$). 
   In turn, this implies that the process $Z(\cdot\wedge \stopinf_k)-Z(\stopoutf_{k-1})$ admits the
following semimartingale 
decomposition: for $t \geq \stopoutf_{k-1}$,   
\begin{eqnarray*} 
Z(t\wedge \stopinf_k)-Z(\stopoutf_{k-1}) & = &  \int_{\stopoutf_{k-1}}^{t\wedge
  \stopinf_k} b(Z(u)) \, du + \int_{\stopoutf_{k-1}}^{t\wedge \stopinf_k}
\sigma(Z(u)) \, dW(u) \\
& & \quad +\int_{0}^{(t - \stopoutf_{k-1})\wedge (\stopinf_k-\stopoutf_{k-1})} \gamma^k(u) d
|\tilde{Y}^k|(u),\end{eqnarray*}
and by It\^{o}'s formula,   on the set $\{\stopoutf_{k-1}\leq t\}$ we have 
\begin{eqnarray*}
f(Z(t\wedge \stopinf_k))-f(Z(\stopoutf_{k-1})) & =&
\int_{\stopoutf_{k-1}}^{t\wedge \stopinf_k} \ML f(Z(u))du+
\int_{\stopoutf_{k-1}}^{t\wedge \stopinf_k} \left<\nabla f(Z(u)),
  \sigma(Z(u))dW(u) \right> \\ & & + \int_{0}^{(t - \stopoutf_{k-1})\wedge (\stopinf_k-\stopoutf_{k-1})} \left<\nabla f(Z(\stopoutf_{k-1}+u)), \gamma^k(u) \right> d|\tilde{Y}^k|(u).
\end{eqnarray*} 
Multiplying both sides of the last display
by $\ind_{\{\stopoutf_{k-1}\leq t\}}$, summing over $k\in \N$ and observing  that $\nabla f$ and $\ML f$ are
identically zero on $B_\varepsilon(\MV)$ because $f$ is constant on
each connected component of $\MV$,  we have the equalities
\[\sum_{k=1}^\infty \ind_{\{\stopoutf_{k-1}\leq
  t\}}\int_{\stopoutf_{k-1}}^{t\wedge \stopinf_k} \left<\nabla f(Z(u)),
  \sigma(Z(u))dW(u) \right> = \int_0^t  \left<\nabla f(Z(u)),
  \sigma(Z(u))dW(u) \right> \] and, likewise,
 \[\sum_{k=1}^\infty \ind_{\{\stopoutf_{k-1}\leq
   t\}}\int_{\stopoutf_{k-1}}^{t\wedge \stopinf_k} \ML f(Z(u))du =
 \int_0^t  \ML f(Z(u))du.  \]
Combining the last three displays with \eqref{eqn1}, we conclude that $\P_z$-almost surely, for every $t \geq 0$, 
\[ 
\begin{array}{l}
\label{eq-z}
\ds f(Z(t))-f(Z(0))  - \int_0^t  \ML f(Z(u))du \\
\quad  =  \ds \int_0^t  \left<\nabla f(Z(u)), \sigma(Z(u))dW(u) \right> \\
 \qquad \quad + \ds \sum_{k=1}^\infty \ind_{\{\stopoutf_{k-1}\leq
  t\}}\int_{0}^{(t-\stopoutf_{k-1})\wedge
  (\stopinf_k-\stopoutf_{k-1})} \left<\nabla f(Z(\stopoutf_{k-1} + u)),
  \gamma^k(u) \right> d|\tilde{Y}^k|(u). 
\end{array}
\]
Since $f\in \MH$, $\gamma^k(u) \in d(Z(\stopoutf_{k-1}+u))$ for 
 $d|\tilde{Y}^k|$ almost every $u$, the second term on the right-hand
 side is almost surely non-decreasing, whereas the local boundedness 
 of $\sigma$ and the fact that $f$ has compact support shows
 that the first term on the right-hand side is a martingale. 
 This implies that  the process described by the right-hand side, and
 therefore the left-hand side,  is a submartingale, and hence, shows that 
 $\Q_z$ satisfies the remaining property (3) of the submartingale problem.  This proves the first
assertion of Theorem \ref{th-exi}. 
The second assertion is a simple consequence of the first. 

\bigskip

\beginsec

\section{Proof of Theorem \ref{th:existence}}
\label{sec-converse}

In this section we provide  the broad outline of the proof of 
Theorem \ref{th:existence}, and establish some preliminary  results.  
The remaining steps of the proof are carried out in 
Sections \ref{sec:case1} and \ref{sec:case2}.   For the rest of the
paper, we consider $(G,d(\cdot))$ that is piecewise 
$\C^2$ with continuous reflection. 
 In light of Remark \ref{exist}, we can (and will) assume that $\sigma = a^{1/2}$, 
where $a = \sigma \sigma^T$.  We also  let $\{\Q_z, z\in \bar G\}$ be a
solution to the associated submartingale problem associated with the
above data, and  let $Z$ be the canonical process on $(\ccspace,
{\mathcal M}, \{{\mathcal M}_t\})$,  defined by \eqref{Z-weak}.

The proof of Theorem \ref{th:existence} consists of three main steps. 
First,  in Section \ref{sec:BM}, for each $z\in \bar{G}$,
we construct  a  continuous adapted stochastic process $\brm$ on  the canonical
filtered
probability space 
$(\ccspace, {\mathcal M}, \{{\mathcal M}_t\})$, and show that under
$\Q_z$, $\{\brm (t), {\mathcal M}_t, t \geq 0\}$ is a $J$-dimensional standard Brownian motion.   Next, in 
Section \ref{subs:locref} (see Proposition \ref{prop-chiomega}), we show that $(\ccspace, {\mathcal M}, \{{\mathcal M}_t\})$, 
$\Q_z$,  $(Z,\brm)$ is 
a weak solution to the associated SDER with initial condition $z$  if a certain local reflection 
property, specified in \eqref{chiomega},  is satisfied.    The third step entails the verification of this local
reflection property.  This is the most involved step, and requires a
careful analysis of the behavior of $Z$ at the boundary of the
domain.  This step is carried 
out in Sections \ref{sec:case1} and \ref{sec:case2} 
for the cases when $z \in G$ and $z \in \MU$, 
respectively. 

For notational conciseness, throughout the rest of the paper, we will use the
following notation.  Let $\{\martf
(t), t \geq 0\}$ be the process given by 
\begin{equation}
\label{def-martf}
 \martf (t)  \doteq  f (\zee (t)) - f(\zee (0)) - \int_0^t
 {\mathcal L} f (\zee(u) ) \, du,  \qquad t \geq 0. 
\end{equation}
for functions $f$ for which the process is well defined. In particular, this is
well defined 
for all $f \in {\mathcal C}^2_c(\bar G)$.  
Also, let $\id$
denote the identity function on $\R^J$: $\id (x) = x$.   Note that 
the measurability and local boundedness of $b$ ensures
that the process 
\be
\label{def-shift}
  \shift (t) \doteq   \zee (t) -  \zee (0) - \int_0^t b ( \zee(u)) \,
du,  \quad t \geq 0, 
\ee
is well defined.

\subsection{Construction of a Brownian Motion} \label{sec:BM}

In the unconstrained case, that is, when constructing a weak solution 
to an SDE with drift $b$ and dispersion coefficient $a^{1/2}$ from a
solution $\Q_z$ to the corresponding martingale problem, 
appropriate test functions can be used to show that $\shift$ is a
martingale   (see, e.g.,
Proposition 4.6 in Chapter 5 of \cite{KarShrBook}).   The Brownian
motion  that drives the SDE can
then be obtained as  a stochastic 
integral with respect to $\shift$. 
 In contrast, in the constrained 
setting, as suggested by the expression in \eqref{def-z}, $\shift$ is
no longer a martingale (and it  need not even be a semimartingale when  
$\MV \neq\emptyset$).  Instead,  we need to use a slightly more
complicated construction.  
First, we define a suitable nested sequence of domains 
$G_m$ and, using  an argument  analogous to the unconstrained case, 
we  show in Lemma \ref{lem-mart} that when restricted to certain random time intervals during which the
process lies strictly inside  a domain $G_m$,  the process $\shift$ 
is a $\Q_z$-martingale for every $z \in \bar{G}$. 
This allows us to construct, for each $z\in \bar{G}$, a sequence $\{ \brm^m \}_{m \in \mathbb N}$
of $\Q_z$-martingales, 
which are then shown in  Lemma \ref{lem-qvbm} to converge along a
subsequence to  a process that is a standard $\Q_z$-Brownian
motion.

\begin{lemma}
\label{lem-mart}
Let $\MO_i$, $i = 1, 2,$ be  connected bounded open subsets of $G$ 
such that $\bar{\MO}_1 \subset \MO_2$ and $\bar{\MO}_2 \subset G$.  
Given any $\{{\mathcal M}_t\}$-stopping time $\stop$, define the two stopping times 
\begin{eqnarray}
\label{def-stopin}
\stopin & \doteq & \inf \{ t > \stop: Z (t) \in \bar{\MO}_1 \},
\\
\label{def-stopout}
\stopout & \doteq & \inf \{ t > \stopin:  Z(t) \notin \MO_2 \}. 
\end{eqnarray}
Then  for every $z \in \bar{G}$, 
$\{ \shift (t \wedge \stopout) - \shift (t \wedge \stopin), \MM_t, t \geq 0\}$ is a continuous $\Q_z$-martingale 
and for $t \geq 0, i, j = 1, \ldots, J$,   
\be
\label{shift-cov1}
[ \shift_i, \shift_j] (t \wedge \stopout)  -  [ \shift_i, \shift_j] (t
\wedge \stopin)  
= \int_{t\wedge \stopin}^{t\wedge \stopout}  a_{ij} (Z(u)) \, du. 
 \ee
\end{lemma}
\begin{proof} 
Since $\bar{\MO}_2 \cap \partial G = \emptyset$, for each $i = 1,
\ldots, J$, there exists $f^{(i)} \in \MC_c^2 (G)$ such that 
$f^{(i)} (x) = x_i$ for $x \in \bar{\MO}_2$ and $f^{(i)} (x) = 0$ for each $x$ in a
neighborhood of $\partial G$.   Then, for every $i, j = 1, \ldots, J$, 
the functions $f^{(i)}$,  $-f^{(i)}$, $f^{(i)} f^{(j)}$ and $-f^{(i)} f^{(j)}$
clearly lie in $\MH$.  Recall the notation $\martf$ introduced
in \eqref{def-martf}, and let $\mart^{(i)}$ and $\mart^{i,j}$ be equal to 
$\mart^f$,  when $f = f^{(i)}$ and $f=f^{(i)} f^{(j)}$, respectively.
Then for any $z \in \bar{G}$, by property 3 of the
submartingale problem and the optional sampling theorem,
$\mart^{(i)}$,  $\mart^{(i,j)}$, $\mart^{(i)} (\cdot \wedge \stopout) - \mart^{(i)} (\cdot \wedge
\stopin)$ and $\mart^{(i,j)} (\cdot \wedge \stopout) - \mart^{(i,j)}
(\cdot \wedge \stopin)$ are all 
continuous $\Q_z$-martingales.   
 Since $\zee (s) \in \bar{\MO}_2$ for $s
\in [t \wedge \stopin, t \wedge \stopout]$ and 
for  $x \in
\bar{\MO}_2$,  $f^{(i)} (x) = x_i$ and 
$\ML f^{(i)} (x) = b_i (x)$, it follows that $\shift_i (\cdot
\wedge \stopout) - \shift_i (\cdot \wedge \stopin)$ is equal to $\mart^{(i)} (\cdot \wedge
\stopout) - \mart^{(i)} (\cdot \wedge \stopin)$, and hence, is a continuous
$\Q_z$-martingale. 
In addition, observing that $\ML (f^{(i)} f^{(j)}) (x) = x_j
b_i (x) + x_i b_j (x) + \frac{1}{2} a_{ij} (x)$,   a standard 
argument [e.g.,  see (4.10)--(4.12) on p. 315 of 
\cite{KarShrBook},  where $M^{(i)}$  there plays the role of $\shift_i$ here] can  be used to show that $\shift_i
\shift_j (t \wedge \stopout) - \shift_i \shift_j (t \wedge \stopin) -
\int_{t \wedge \stopin}^{t\wedge \stopout} a_{ij} (\zee (u)) \, du$
is a continuous $\Q_z$-martingale. 
This establishes \eqref{shift-cov1}.  
\end{proof}

Let $\{G_m,\ m\in \N\}$ be
a sequence of bounded domains in $G$ such that $\bar{G}_m\subset
G_{m+1}$ for each $m\in \N$ and $\cup_{m\in \N}G_m=G$.  
Also, for each $m \in \N$, let $\stopout_0^m\doteq 0$ and let $\{\stopin_k^m: k \in \N\}$ and
$\{\stopout_k^m: k \in \N\}$ be nested sequences of stopping times
defined by 
\begin{eqnarray*}\stopin_k^m & \doteq & \inf \{t>\stopout_{k-1}^m: \
  \zee (t)\in \bar{G}_{m} \}, \\
 \stopout_k^m & \doteq &  \inf \{t>\varsigma_k^m:\ \zee (t)
 \notin G_{2m}\}. 
\end{eqnarray*}
For each $k \in \N$, applying Lemma \ref{lem-mart} with $\MO_1 = G_m$, $\MO_2 = G_{2m}$,
$\tau_0 = \tau_{k - 1}^m$, $\stopin = \stopin^m_{k}$ and
$\stopout  = \stopout^m_{k}$, it follows that 
$\{\shift (t \wedge \stopout_{k}^m) - \shift (t \wedge
\stopin_{k}^m), t \geq 0\}$ 
is a continuous $\Q_z$-martingale with  covariation processes 
\be
\label{shift-cov2}   [ \shift_i, \shift_j ] ( t \wedge \stopout_k^m) -  
 [\shift_i, \shift_j ] ( t \wedge \stopin_k^m) = \int_{t \wedge
  \stopin_k^m}^{t \wedge \stopout_k^m} a_{ij} (\zee (u) )
\, du,  \quad t \geq 0, i, j = 1, \ldots, J. 
\ee 
Now,  for $m \in \N$, define 
 \begin{eqnarray} \label{q12m}\brm^m (t)&\doteq & 
\ds  \sum_{k = 1}^\infty \int_{t \wedge \stopin^m_k}^{t \wedge
  \stopout^m_k} \ind_{G_m} (\zee (u)) a^{-1/2}(\zee (u)) d \shift (u),
\qquad t \geq 0. 
\end{eqnarray}
In what follows, recall that $\delta_{ij}$ represents the Kronecker delta: 
 $\delta_{ij}=1$ if $i=j$, and $\delta_{ij}=0$ otherwise.

\begin{lemma}
\label{lem-qvbm} For every $z \in \bar{G}$, 
for each $m \in \N$, the  process $\{\brm^m (t), 
t \geq 0\}$ is well defined and is a continuous $\Q_z$-martingale with  covariation
processes given by 
\be \label{QV:Bm} [\brm^m_i,\brm^m_j](t)= \delta_{ij} \int_0^{t} \ind_{G_m}
(\zee (u)) \, du, \qquad t \geq
0, 1\leq i,j\leq J.    \ee 
 Moreover, there exists a process $\{\brm(t),
{\mathcal M}_t\}$  that is a Brownian
motion under $\Q_z$ and a  
subsequence
$\{\brm^{m_n},  n \in \N\}$ such that, as $n \ra \infty$, $\brm^{m_n}$
$\Q_z$-almost surely converges uniformly on 
bounded intervals to $\brm$. 
\end{lemma}
\begin{proof}  Fix $z\in \bar{G}$. 
The uniform ellipticity condition (\ref{eq-ue}) implies that the
(random) eigenvalues of $a(Z(t))$ are bounded below by $\bar{a}$
uniformly in $t\geq 0$ and hence the (random) eigenvalues of
$a^{-1/2}(Z(t))$ are bounded above by $\bar{a}^{-1/2}$ uniformly in
$t\geq 0$. Thus, $\{a^{-1/2} (\zee (t)), t \geq 0\}$ is a bounded,
$\{{\mathcal M}_t\}$-adapted process. 
Now,  for each $m, k \in \N$,  by (\ref{shift-cov2})  the covariation of the continuous $\Q_z$-martingale
$\shift ( \cdot \wedge \stopout_k^m) - \shift (\cdot \wedge
\stopin_k^m)$  is absolutely  continuous with respect to Lebesgue
measure, and hence, the process  
\be
\label{ehkm}
 \eh^{k,m} (t) \doteq \int_{t \wedge \stopin_k^m}^{t
  \wedge \stopout_k^m} \ind_{G_m} (\zee (u) ) a^{-1/2} (\zee
(u)) d\smart (u),  \qquad t \geq 0, 
\nonumber \ee 
 is a well defined,  continuous $\Q_z$-martingale. 
Moreover, \eqref{shift-cov2} also implies that 
for $i, j, k, k' \in \N$, 
\begin{align}
\nonumber
[\eh_i^{k,m}, \eh_j^{k',m}] &= \delta_{kk'} 
\sum_{\ell, \ell' =    1}^J 
\int_{t \wedge \stopin_k^m}^{t \wedge \stopout_k^m} \ind_{G_m} (\zee (u)) (a^{-1/2})_{i\ell} (\zee(u)) (
  a^{-1/2})_{j\ell'} (\zee(u))  d[\shift_\ell, \shift_{\ell'}] (u)\\ 
\nonumber
 &=  \delta_{kk'} \sum_{\ell, \ell' = 1}^J \int_{t \wedge \stopin_k^m}^{t
  \wedge \stopout_k^m} \ind_{G_m} (\zee (u)) (a^{-1/2})_{i\ell} (\zee(u)) (
  a^{-1/2})_{j\ell'} (\zee(u)) a_{\ell \ell'} (\zee (u)) \, du \\
\nonumber
  &=  \delta_{ij}  \delta_{kk'} \int_{t \wedge
  \stopin_k^m}^{t \wedge \stopout_k^m}\ind_{G_m} (\zee(u)
) \, du \\
\label{eh-cov}
& =    \delta_{ij}  \delta_{kk'}  \int_{t \wedge \stopin_k^m}^{t \wedge \stopin_{k+1}^m}\ind_{G_m} (\zee(u)
) \, du, 
\end{align}
where the last equality uses the fact that $\zee(u) \notin G_m$ for
$u \in [\stopout_k^m, \stopin_{k+1}^m]$.  

We next show that for a fixed $t> 0$ and $1\leq i\leq J$, 
$\sum_{k=1}^n \eh^{k,m}_i (t)$ converges in $\mathbb{L}^2(\Q_z)$  to $\brm^m_i (t)$ as $n\rightarrow \infty$. 
Applying Fatou's lemma, and using 
 \eqref{eh-cov} and the fact that  $\stopout_{k}^m \ra
\infty$ as $k \ra \infty$, we see that for $n \in \N$, 
 \begin{eqnarray*}
\E^{\Q_z}\left[\left|\brm^m_i (t)-\sum_{k=1}^n \eh^{k,m}_i
    (t)\right|^2\right] 
& \leq &  \sum_{k=n+1}^\infty \E^{\Q_z}\left[(\eh^{k,m}_i (t))^2\right] 
\\
& = & \E^{\Q_z}\left[ \int_{t \wedge \stopin_{n+1}^m}^t
    \ind_{G_m} (\zee (u)) du \right], 
\end{eqnarray*}
where the last equality uses the  monotone convergence theorem and the fact that $\zee (u) \notin
G_m$ for $u \in [\stopout_k^m, \stopin_{k+1}^m]$. 
Sending $n \rightarrow \infty$, the right-hand side above converges to
zero by the dominated convergence theorem and the fact that 
$\stopin_{n+1}^m \rightarrow \infty$ as $n\rightarrow \infty$. 
Thus,  we conclude that 
$\sum_{k=1}^n \eh^{k,m}_i (t)$ converges in $\mathbb{L}^2(\Q_z)$ 
to $\brm^m_i (t)$.  Since each $\sum_{k=1}^n \eh^{k,m}_i$ is a
continuous $\Q_z$-martingale, by Proposition 1.3 of  \cite{ChuWilBook90},  $\brm^m (t) = \sum_{k = 1}^\infty \eh^{k,
  m}( t)$ is well-defined and also a continuous $\Q_z$-martingale.   Moreover, using
\eqref{eh-cov}, the fact that $\stopin_{n+1}^m \ra \infty$ as $n \ra
\infty$, and the dominated convergence theorem,  
it is easy to see that 
its covariation process is equal to 
\begin{eqnarray*}
  [\brm^m_i, \brm^m_j] (t)  =    \delta_{ij}    \sum_{k = 1}^\infty  \int_{t \wedge
  \stopin_k^m}^{t \wedge \stopin_{k+1}^m}\ind_{G_m} (\zee (u))du =
\delta_{ij}    \int_{0}^{t}\ind_{G_m} (\zee (u) ) \, du, 
\end{eqnarray*}
which proves \eqref{QV:Bm}.  

We  now extract a convergent subsequence of 
$\{\brm^m, m \in \N\}$.  
Let  $\tilde{m} > m$, $\tilde{m}, m \in \N$.  
Using the fact that  $\bar{G}_m \subset G_{\tilde{m}}$,  it is easy to see that for any $k \in \N$, there exists  
$\tilde{k} \in \N$ such that $[\stopin_{k}^m, \stopout_{k}^m]
\subset [\stopin_{\tilde{k}}^{\tilde{m}},
\stopout_{\tilde{k}}^{\tilde{m}}]$. 
Moreover, for any $\tilde k\in \N$, $\zee (u) \in G_m$ for $u \in
[\stopin_{\tilde{k}}^{\tilde{m}}, \stopout_{\tilde{k}}^{\tilde{m}}]$
implies that $u \in [\stopin_{k}^m, \stopout_{k}^m]$ for some $k$.   
Together, this implies that 
\begin{eqnarray}
\label{diff-eq}
\brm^{\tilde{m}} (t) - \brm^m (t) & = &  \sum_{\tilde k =1}^\infty \int_{t \wedge
  \stopin_{\tilde k}^{\tilde m}}^{t \wedge \stopout_{\tilde k}^{\tilde
    m}} \ind_{G^{\tilde{m}}\sm G_m} (\zee (u)) a^{-1/2} (\zee 
  (u)) d \smart (u). 
\nonumber\end{eqnarray}
The argument that was used to establish \eqref{QV:Bm} can
also be used to  show that 
\[  [\brm^{\tilde{m}}_i - \brm^m_i] (t) = \int_0^t
\ind_{G_{\tilde{m}}\setminus G_m} ( \zee (u))\,du,  \quad i = 1,
\ldots, J. 
\]
For any $k \in \N$, $T < \infty$, and $i = 1,\ldots, J$,
by Doob's maximal inequality 
we have  
\begin{eqnarray*}\Q_z\left(\sup_{t \in [0,T]}|\brm^{\tilde{m}}_i(t)-\brm^{m}_i(t)| \geq
    \frac{1}{2^k}\right) &\leq & 2^{2k}
  \E^{\Q_z}\left[|\brm^{\tilde{m}}_i(T)-\brm^{m}_i(T)|^2\right]\\
  & = & 2^{2k} \ \E^{\Q_z}\left[\int_0^T\ind_{G_{\tilde{m}}\setminus
      G_{m}}(\zee (u))\,du\right]. 
\end{eqnarray*} 
Since $\cup_{m \in \N} G_m =  G$,   taking first $\tilde{m} \ra
\infty$ and then $m \ra \infty$, the 
 expectation on the right-hand side 
converges to zero by 
the bounded convergence theorem. 
Hence, there exists 
 a subsequence $\{\brm^{m_k}, k\in \N\}$ such that
 for $i = 1, \ldots, J$, 
\[\Q_z\left(\sup_{t \in [0, T]}|\brm^{m_{k+1}}_i(t)-\brm^{m_k}_i(t)| \geq
  \frac{1}{2^k}\right)\leq 
 2^{2k}\E^{\Q_z}\left[\int_0^T \ind_{G_{m_{k+1}}\setminus G_{m_k}}(\zee (s))\,ds\right]\leq\frac{1}{2^{k}},\] 
 and  by the Borel-Cantelli lemma and the completeness of $\ccspace$, it follows that 
there exists a set $\ccspace^0$ with $\Q_z(\ccspace^0)=1$ 
such that for  $\omega\in \ccspace^0$, 
the sequence of continuous functions $\{\brm^{m_k}(\omega,t),t \geq
0\}$, $k\in \N,$ converges uniformly on bounded intervals to a continuous
process $\brm = \{\brm(\omega,t), t \geq 0\}$.  
Furthermore,  for $\omega \in \ccspace^0$, by \eqref{QV:Bm} and the fact
that $\cup_{m \in \N} G_m =  G$,   
we have for $i, j = 1, \ldots, J$, 
\be
\label{betan-conv} \lim_{k \ra \infty}  \sup_{t \in [0,T]} \left| [ \brm_i^{m_k}, \brm_j^{m_k}] (t,\omega) -
\delta_{ij} \int_0^t \ind_{G} (\zee (u,\omega)) du \right| = 0. 
\nonumber\ee
The boundary property \eqref{prop-bdary} shows that $\Q_z$-almost surely,
$\int_0^t \ind_{G} (\zee (u)) du = t$ and thus, that  $(\brm^m, [\brm^m,
\brm^m])$ converge jointly to $(\brm, I_J t)$, where $I_J$ is the $J
\times J$ identity matrix.   Thus, 
by \cite[Theorem 2.4; p. 528]{JacShiBook},  $\brm$ is
a continuous $\Q_z$-martingale with $[\brm_i, \brm_j]  = \delta_{ij} t$, and so 
by L\'{e}vy's characterization theorem, $\brm$ is a $J$-dimensional
standard Brownian motion under $\Q_z$. 
\end{proof}

\subsection{A Local Reflection Property}
\label{subs:locref}

Fix $z\in \bar{G}$. Given the $\{{\mathcal M}_t\}$-adapted process $\zee$  as in
\eqref{Z-weak},  and the standard $\{{\mathcal M}_t\}$-Brownian motion $\brm$ 
constructed in Section \ref{sec:BM}, define the candidate
unconstrained process $\exx$ in terms of $Z$ and
$W$ via \eqref{def-x}, and let $\y$ be the corresponding candidate ``pushing'' or local time process: 
\be \label{def-y} 
\y(t) \doteq \zee(t)-\exx(t),  \quad t \geq 0. 
\ee  
Recalling the definition of $\shift$ in  \eqref{def-shift}, and using
\eqref{def-x} we can also write 
\be \label{eq:y}
\y(t) = \smart(t) - \int_0^t a^{1/2}
(\zee (u)) d \brm(u), \quad t \geq 0. 
\ee
Now, observe that property 1 of the weak
solution in  Definition \ref{def-SDER}  holds 
by the definition of $W$, property 2 follows due to the continuity of $Z$ and 
the local boundedness of $b$ and $\sigma$, and property 
4 is a consequence of the boundary property established in 
Proposition \ref{prop-bdary}.   Thus, to show that  $(\ccspace,
{\mathcal M}, \{{\mathcal M}_t\})$, $\Q_z$, 
$(\zee, \brm)$  is a weak solution to the SDER with initial condition 
$z$, 
it only remains to verify the third property, namely, to show that
$\Q_z$-almost surely, $(\zee,\y)$ solves  
the ESP for $\exx$.   However, property 1 of the ESP holds trivially
by the definition \eqref{def-y} of $Y$ and
property 2 of the ESP is a direct consequence of property 2 of the submartingale
problem and the definition \eqref{def-z} of $Z$. Thus, to prove Theorem \ref{th:existence}, it suffices 
to verify that $\y$ and $\zee$ satisfy the following ``reflection property'' embodied in
property 3 of the ESP: $\Q_z$ almost
surely, 
\be
\label{chiomega}
 \y (t) - \y(s)  \in \conv \left[ \huge\cup_{u \in [s,t]} d( Z(u))
 \right]\ \mbox{ for every }0\leq s<t<\infty. 
\ee

We now show that the proof of Theorem \ref{th:existence} 
can be reduced to the verification of a local version of
\eqref{chiomega}. Recall the stopping time $\tau_\MV$ defined by 
\eqref{def-taumv}.

\begin{proposition}
\label{prop-chiomega}
The triplet $(\ccspace, {\mathcal M},\{{\mathcal M}_t\})$,
$\Q_{z}$, $(Z(\cdot \wedge \tau_{\MV}),W (\cdot \wedge \tau_{\MV}))$  
is a weak solution to the SDER with initial condition $z$ associated 
with $(G, d(\cdot))$, $\MV$, $b(\cdot)$ and $\sigma(\cdot)$
if the process  $\y$ in \eqref{def-y} satisfies the following two properties: 
\begin{enumerate}
\item[1.] 
For every $\barz \in G$,  $\Q_{\barz}$-almost surely, $Y(t) = 0$ for $t \in [0,\tau_1]$, where 
\be
\label{def-tau}
\tau_1 \doteq \inf \{ t \geq 0: Z(t) \in \partial G\};  
\ee 
\item[2.] 
For every  $\barz \in \MU$, there exists an open neighborhood ${\mathcal O}_{\barz}$ of $\barz$ such that ${\mathcal O}_{\barz} \cap \MV = \emptyset$ and 
\be
\label{eq-locref}
 \Q_{\barz}\left(Y(t\wedge T_{\barz}) \right)\in \conv \left[    \cup_{u \in [0,t\wedge T_{\barz}]} d(\zee(u))  \right]=1,
\ee
 where 
\[  T_{\barz} \doteq \inf \{t \geq 0: Z(t) \not \in {\mathcal O}_{\barz} \}. 
\]
\end{enumerate}
\end{proposition}

The proof of Proposition \ref{prop-chiomega} will rely on two
preliminary results that we now state.  
We will make frequent use of a conditioning argument, which is encoded in 
the following lemma.

\begin{lemma} \label{lem:submartperv}
Fix $y \in \bar{G}$,  let $\tau$ be a stopping time on $(\ccspace, {\mathcal M}, \{{\mathcal M}_t\})$, 
and let $\{\Q_{\omega'}, \omega' \in  \ccspace\}$ be a family of regular conditional probability
distributions of $\Q_y$ given $\MM_{\tau}$. 
 There then exists a set $\MN\in \MM_{\tau}$ with $\Q_y(\MN)=0$ such that
 for every $\omega'\notin \MN$  with  $\tau(\omega')<\infty$ and each
$f\in \MH$, \[f(\omega(\tau(\omega')+t))-\int_{\tau(\omega')}^{\tau(\omega')+t}\ML
f(\omega(u))\,du\] is a $\Q_{\omega'}$-submartingale and, moreover,
$\omega(\tau(\omega')+\cdot)$ under $\Q_{\omega'}$ is a solution to
the submartingale problem starting from $\omega'(\tau(\omega'))$.    
\end{lemma}

For a fixed $f \in \MH$, this is a simple consequence of Theorem 1.2.10 of \cite{StrVarBook06}.
The main content of Lemma \ref{lem:submartperv} is the statement that a fixed exceptional set
$\MN$ can be chosen to work simultaneously for all $f \in \MH$.   
This is  a simple consequence of the separability of $\MH$ established
in \cite{KanRam14}.  The details of the proof are  
deferred to Appendix \ref{ap-regcondprob}. 

The next result is a covering lemma that is used to verify that certain random times
constructed in the proof of Proposition \ref{prop-chiomega} 
are stopping times.  The proof of the covering lemma is deferred to
Appendix \ref{subs-cover}.

\begin{lemma}[{\bf Covering Lemma}]
\label{lem:cover}
Given a collection of open sets $\{{\mathcal O}_y$, $y \in \partial G\}$ such that ${\mathcal O}_y \cap \MV = \emptyset$ if $y\in \MU$,  there exists a countable
set of points $\S \subset \MU$ such that
\be \label{cover}\MU \subseteq  \bigcup_{y \in \S} {\mathcal O}_y, \nonumber \ee
and there exists a measurable mapping $\kappa$ from $\MU$
onto $\S$ such that  $y\in {\mathcal O}_{\kappa(y)}$ and
$\MI(y)=\MI(\kappa(y))$ for each $y\in \MU$.
\end{lemma}

The following definition (see Section 2.3 and Theorem 2.3.1 of
\cite{WonBook96}) is used in the proof of Proposition
\ref{prop-chiomega}. 

\begin{defn} \label{def:esssup}
Let $\{Z_{\alpha}:\ \alpha \in I\}$ be a collection of random variables on a probability space $(\Omega, \MF,P)$, with $I$ an arbitrary index set. Then there exists a unique (up to $P$-null sets) random variable $Z^* : \Omega \rightarrow \hat{R} = R \cup \{-\infty,\infty \}$ such that:
\begin{enumerate}
\item[(i)] $P(Z_\alpha \leq Z^*)=1$, $\forall \alpha\in I$,
\item[(ii)] if $Y:\ \Omega\rightarrow \hat{R}$ is another random
  variable that satisfies $\P(Z_\alpha \leq Y)$, then $P(Z^*\leq Y)=1$.
\end{enumerate}
We call $Z^*$ the \textit{essential supremum} of $\{Z_\alpha: \alpha \in I\}$, and write
\[ Z^*= \esssup_{\alpha \in I} Z_\alpha. \]
\end{defn}

\begin{proof}[Proof of Proposition \ref{prop-chiomega}]
Suppose $Y$ satisfies properties 1 and 2 of Proposition \ref{prop-chiomega}. 
 From the discussion prior to Proposition \ref{prop-chiomega}, it is 
 clear that to  show that $(\ccspace, {\mathcal M},\{{\mathcal
   M}_t\})$, $\Q_z$,  
$(Z(\cdot \wedge \tau_{\MV}),\brm(\cdot \wedge \tau_{\MV}))$ is a weak
 solution to the SDER with initial condition $z$, it suffices to show 
 that  for every $0 \leq s \leq t$, \eqref{chiomega}  holds with $Y$ and $Z$ replaced by 
$Y(\cdot \wedge \tau_\MV)$ and $Z(\cdot \wedge \tau_\MV)$, 
respectively.   Applying the regular
 conditional probability distribution of $\Q_z$ given $\MM_s$,  and using the continuity of the paths of $\y$ and    
 Lemma \ref{lem:submartperv},  
it suffices to show that for each $\barz \in \MU$ and $0<t <\infty$, we have
 \be  \label{final} \Q_{\barz}\left(\y (t\wedge \tau_\MV)  \in \conv \left[ \huge\cup_{u
       \in    [0,t]} d( \zee (u\wedge \tau_\MV)) \right]\right) =1.\ee

Let $\Psi_S$ be the set of stopping times on $(\ccspace, {\mathcal M},
\{{\mathcal M}_t\})$, and  define 
\[\Psi\doteq \{\tau: \tau \in \Psi_S: \tau \leq \tau_\MV \mbox{ and } \Q_z\left(\y (t\wedge \tau\wedge \tau_\MV)  \in \conv \left[ \huge\cup_{u
       \in    [0,t\wedge \tau]} d( \zee (u\wedge \tau_\MV))
   \right]\right) =1\}.\]    
Note that $\Psi$ is closed under $\Q_z$-almost sure limits. 
We now use an argument by contradiction to show that \be \label{esssup} \Q_z\left(\esssup_{\tau\in \Psi} \tau =\tau_\MV\right)=1.\ee
Suppose (\ref{esssup}) does not hold.  Then $\Q_z\left(\esssup_{\tau\in \Psi} \tau <\tau_\MV\right)>0$. 
It is well known (cf. Theorem 2.3.1 of \cite{WonBook96}) that there exists
a sequence  $\{\tau^n, n\in \N\}
\subset \Psi$ such that $\Q_z$-almost surely, $\tau^n$ is increasing
and 
\be
\label{taustar} 
\Q_z \left(\esssup_{\tau\in \Psi} \tau = \lim_{n \ra \infty}\tau^n\right) = 1. 
\ee
Now, consider 
the stopping time $\tau^* \doteq \liminf_{n \ra \infty} \tau^n$ (where
we use the limit inferior to deal with the set of $\Q_z$-measure zero 
on which $\tau_n, n \in \N,$ need not be monotonic).    Then, clearly 
$\Q_z (\tau^* = \esssup_{\tau\in
  \Psi} \tau ) = 1$ and 
 since each $\tau^n \in \Psi$,  \eqref{taustar} and the closure
 property of $\Psi$ implies $\tau^* \in \Psi$.  
Due to  the hypothesis that \eqref{esssup} does not hold and
\eqref{taustar},  it also follows that 
$\Q_z\left(\tau^* <\tau_\MV\right)>0$.
       
Now, let ${\mathcal O}_z$, $z \in \MU$, be the collection of open sets
defined in property 2 of Proposition \ref{prop-chiomega},  set
${\mathcal O}_z \doteq G$ for $z \in G$, and let   $\kappa: \MU
\mapsto \S$ be the measurable mapping from Lemma \ref{lem:cover}. 
Let $\bar \varrho$ be a stopping time defined by  
\begin{eqnarray}
\label{tau}
\bar \varrho \doteq \left\{\begin{array}{ll} \inf\{t\geq  \tau^*: \zee (t\wedge \tau_\MV)\notin
{\mathcal O}_{\kappa(\zee( \tau^*))}\} & \mbox{ if } \tau^*<\tau_\MV, \\ \tau_\MV & \mbox{ if } \tau^*\geq \tau_\MV. \end{array}  \right.
\nonumber\end{eqnarray}
By the property of $\kappa$ stated in Lemma \ref{lem:cover}, $\zee( \tau^*)\in G\cup\MU$ lies strictly in the interior of ${\mathcal O}_{\kappa(\zee( \tau^*))}$ on $\{\tau^*<\tau_\MV\}$. It then follows that $\tau_\MV\geq \bar \varrho \geq \tau^*$ and $\Q_z(\bar \varrho > \tau^*) >0$. 
Given the regular conditional 
probability distribution $\{\Q_{\omega'}\}$ of $\Q_z$ given
$\MM_{\tau^*}$,  Lemma \ref{lem:submartperv} implies that 
there exists a set $\MN \in \MM_{\tau^*}$ with $\Q_z(\MN)=0$
such that for every $\omega'\notin \MN$ with
$\tau^*(\omega')<\bar \varrho(\omega')\wedge \tau_\MV(\omega')$, $\Q_{\omega'}$ is a
solution to the submartingale problem starting from
$\omega'(\tau^*(\omega'))\in G\cup\MU$.  Applying 
properties 1 and 2 of Proposition \ref{prop-chiomega} with $\barz = \omega'(\tau^*(\omega'))$, we see that if
$\omega'(\tau^*(\omega'))\in G$, then \[\Q_{\omega'}(\y(t\wedge
\bar \varrho\wedge \tau_\MV)-\y(t\wedge \tau^*\wedge \tau_\MV) = 0
)=1,\] whereas if $\omega'(\tau^*(\omega'))\in \MU$, then  \[\Q_{\omega'}(\y(t\wedge
\bar \varrho\wedge \tau_\MV)-\y(t\wedge \tau^*\wedge \tau_\MV) \in  \conv \left[ \huge\cup_{u \in [t\wedge \tau^*,t\wedge \bar \varrho]} d( \zee (u\wedge \tau_\MV)) \right])=1.\] 
Combining the last two displays with the fact that $\tau^* \in \Psi$, it follows that 
\[\Q_z\left(\y (t\wedge \bar \varrho\wedge \tau_\MV)  \in \conv \left[ \huge\cup_{u
       \in    [0,t\wedge \bar \varrho]} d( \zee (u\wedge \tau_\MV)) \right]\right) =1. \]
Thus, $\bar \varrho \in \Psi$. Since $\Q_z(\bar \rho > \tau^*) >0$ and
$\Q_z(\tau^*=\esssup_{\tau\in \Psi} \tau)=1$, it follows that
$\Q_z(\bar \rho > \esssup_{\tau\in \Psi} \tau)>0$. This contradicts
property (i) of Definition \ref{def:esssup}. Thus, (\ref{esssup})
holds, which in turn  implies that (\ref{final}) holds. 
 This completes the proof of  Proposition \ref{prop-chiomega}.

\end{proof}

\begin{proof}[Proof of Theorem \ref{th:existence}] 
The above discussion show that to complete the proof of Theorem \ref{th:existence}, it only
remains to verify properties  1 and 2 of Proposition
\ref{prop-chiomega}.  This is carried out  
in Sections \ref{sec:case1} and
\ref{sec:case2}, respectively.  In particular, Property 1 of Proposition
\ref{prop-chiomega} follows from Lemma \ref{lem:sderep}  and property 2 of Proposition
\ref{prop-chiomega} follows from Lemma \ref{lem-locrefu}.
\end{proof}

\section{Local Reflection Property for $\barz \in G$} \label{sec:case1}

We need to prove property 1 of Proposition \ref{prop-chiomega} for all 
$\barz \in G$.   This is established in Lemma \ref{lem:sderep}. 
 Define $\stopouti$ as in \eqref{def-tau}, 
let $\brm$ be the $\Q_{\barz}$-Brownian motion constructed in Section
\ref{sec:BM}, and let $\shift$ and $\y$ be the processes defined in \eqref{def-shift} and
\eqref{def-y}, respectively.

 \begin{lemma} \label{lem:sderep}  For $\barz \in G$,  the process
   $\smart(\cdot\wedge \stopouti)$ is a $\Q_{\barz}$-martingale and
   satisfies 
 \begin{eqnarray}
\label{rep-rho}
\smart (t \wedge \stopouti) = \int_0^{t \wedge \stopouti}  a^{1/2}
(\zee (u)) d \brm (u),  \quad  t \geq 0.  
\end{eqnarray}
Moreover,  $\Q_{\barz}$-almost surely, 
$\y(t)=0$ for each $t\in [0, \stopouti]$.
 \end{lemma}
\begin{proof} 
Let $G_m, m \in \N,$  be the sequence of nested domains introduced in  Section
 \ref{sec:BM}, and let $\stopin_1^m$ and $\stopout_1^m$, $m \in \N$,
 denote the associated sequence of stopping times defined in 
\eqref{def-stopin} and \eqref{def-stopout}, respectively, with $\MO_1=G_m$ and $\MO_2=G_{2m}$. 
Since $\barz \in G$ and $\cup_m G_m = G$,  there exists  $m_0 < \infty$
 such that $\barz \in G_m$ for all $m \geq m_0$.    Then  
$\stopin_1^m = 0$  for $m \geq m_0$. Let $\tilde \stopout_i^{m} \doteq \inf\{t\geq 0:\ Z(t)\notin G_m\}$. It follows that  $\tilde \stopout_i^{m} \leq \tau_i^{m}$ for $m \geq m_0$ and  
\be
\label{stopout-conv}
 \stopouti =  \lim_{m \ra \infty} \stopout_1^m =  \lim_{m \ra \infty} \tilde \stopout_1^m.
\ee
From the discussion preceding \eqref{shift-cov2} it follows that 
$\smart(\cdot \wedge \stopout_1^m)$, $m \in \N$, is a sequence of 
continuous $\Q_{\barz}$-martingales. By \eqref{stopout-conv},  this sequence converges uniformly on
compact intervals to $\smart (\cdot \wedge
\stopouti)$. Since the dispersion matrix $\sigma(\cdot)$ is
locally bounded, then $a_{ij}(\cdot)$ is also locally bounded and hence
$a_{ij}(Z(\cdot))$ is locally integrable. Together with 
\eqref{shift-cov2} and \eqref{stopout-conv}, this shows that
$\Q_z$-almost surely, for every $T < \infty$, 
\[  \lim_{m \ra \infty}  \sup_{t \in [0,T]}\left| [\smart_i, \smart_j] (t
  \wedge 
\stopout_1^m)   -  \int_0^{t \wedge \stopouti} a_{ij} (\zee (u)) \,
du \right| = 0. 
\]
Thus, by \cite[Theorem 2.4; p. 528]{JacShiBook},  $\smart
(\cdot \wedge \stopouti)$ is a continuous $\Q_{\barz}$-martingale with 
\be
\label{shift-cov3}
[\smart_i, \smart_j](t \wedge \stopouti)  =  \int_0^{t  \wedge
  \stopouti} a_{ij} (\zee (u)) du,
\quad t \geq 0. 
\ee

Next, note that by \eqref{q12m},
\[ 
 \brm^m (t \wedge \stopout_1^{m}) = \int_0^{t \wedge \stopout_1^{m}}
\ind_{G_m} (\zee (u)) a^{-1/2} (\zee (u) ) d \smart (u). 
\] Since $\zee (u) \in
G_m$ for $u \in [0, \tilde \stopout_1^{m})$, we have 
\[ 
 \brm^m (t \wedge \tilde \stopout_1^{m}) = \int_0^{t \wedge \tilde \stopout_1^{m}}
\ind_{G_m} (\zee (u)) a^{-1/2} (\zee (u) ) d \smart (u) = \int_0^{t
  \wedge \tilde \stopout_1^{m}} a^{-1/2} (\zee (u) ) d \smart (u). 
\]
For any $t > 0$, using \eqref{shift-cov3} and \eqref{stopout-conv},  it is straightforward to
show that the right-hand side of the last equality converges in 
$\mathbb{L}^2(\Q_{\barz})$ to $\int_0^{t \wedge \stopouti} a^{-1/2} (\zee (u)) d \smart
(u)$.   On the other hand, by Lemma \ref{lem-qvbm} and \eqref{stopout-conv},  the left-hand side, 
$\brm^m (t \wedge \tilde \stopout_1^{m})$, converges 
$\Q_{\barz}$-almost surely  to $\brm ( t \wedge \stopouti)$ along a subsequence.  
This proves 
\[  \brm (t\wedge \stopouti)=\int_0^{t\wedge \stopouti}
  a^{-1/2}(\zee (u)) \, d \shift(u). 
\]
Taking the stochastic integral of the martingales on both sides with
respect to $a^{1/2}$, 
we obtain \eqref{rep-rho}.  
\end{proof}

\section{The Local Reflection Property for $\barz \in \MU$} 
\label{sec:case2}

Throughout this section we fix $\barz \in \MU$. 
The proof of the reflection property \eqref{chiomega} is considerably more involved 
in this case, and is broken down into several steps.  
   First, in Section \ref{subs-prelims} we establish some preliminary
   results on the existence of certain test functions, which is then used
   in Section \ref{subs-locsm} (see Proposition
   \ref{prop-locsm}) to  show that the stopped 
process $\shift(\cdot \wedge \stopr_r)$, for a  suitable stopping time $\stopr_r$ is a
continuous $\Q_{\barz}$-semimartingale, and thus admits a decomposition into a continuous $\Q_{\barz}$-local
martingale $\martin$ and a continuous finite variation process $\fvproc$. 
 (In contrast, note that $\shift$ is typically not  a
semimartingale on any positive time interval when $\barz \in \MV$.)
Then, in Lemma \ref{lem:Urep1} of Section \ref{subsub-submart}, we characterize the behavior of
$A$ and $M$ in the interior of the domain $G$, showing in particular
that $A$ is constant in the interior of the domain and that the
covariation of 
$M_i$ and $M_j$ in the interior of the domain is 
$\int a_{ij}(\zee (s)) \, ds$, as desired. 
  In Section \ref{subs-martbound} we first show in Proposition
\ref{lem:q8}  that the trace of the quadratic
variation process of $M$ vanishes on the boundary
$\partial G$.  The proof of this property is non-trivial due to the
geometry of the domain and directions of reflection.   It uses
properties of a certain family of processes $\pushproc^f$ that are
first established in Section \ref{subs-integrep} and Appendix
\ref{ap-chig}. This property is then used in Lemma \ref{lem:Urep2} to
show that $M(\cdot)$ is equal to the
stochastic integral $\int_0^{\cdot \wedge \stopr_r} a (Z(u)) \,
dW(u)$, which allows us to identify 
$\fvproc(\cdot)$ with $Y(\cdot \wedge \stopr_r)$, the  (stopped) candidate local
time process.  
 Finally, in Section
\ref{subsub-locesp} we show that the increments of $Y(\cdot
\wedge \stopr_r)$ satisfy the
desired reflection property stated in \eqref{eq-locref}.
This again uses properties of the family of processes $\pushproc^f$ 
as well as  geometric properties of the domain  established
in Section \ref{subs-cover}. 

\subsection{Existence of Test Functions}
\label{subs-prelims}

Our proof will  
makes use of certain geometric
properties of the directions of reflection and the existence of certain test
functions, which we  summarize  in Lemmas \ref{lem:cutoff}--\ref{lem:htest2}.

\begin{lemma}\label{lem:cutoff}
Suppose $(G,d(\cdot))$ is a piecewise ${\mathcal C}^2$ domain with 
continuous reflection.  Then, 
for each $y \in  \MU$,  there exist $\alpha_y >0$ 
and  $0< R_y<\dist(y,\MV)$  such that 
\begin{enumerate}
\item[1.] 
$\MI(x)\subseteq \MI(y)$ for all $x\in
B_{R_y}(y)\cap \partial G;$ 
\item[2.]  $\sup_{n \in n(y): |n|=1}\inf_{x\in B_{R_y}(y)\cap \partial G}\inf_{d\in d(x): |d|=1}\lan n, d \ran \geq
\alpha_y$;  
\item[3.] 
 There exist $r_y < R_y$, an increasing, continuous function
$\kappa_y:(0,r_y] \mapsto
(0,\infty)$ that satisfies $\kappa_y < r$ and a collection of
functions  $\{f^{y,r}, r \in (0,r_y]\}$ on $\R^J$  such that 
\begin{enumerate}
\item $-f^{y,r} \in \MH \cap \C^2_c(\bar G)$; 
\item $\supp[f^{y,r}] \cap \bar G \subset B_{r}(y) \cap \bar{G}$;
\item $0 \leq f^{y,r}(x)\leq 1$ for all $x\in \bar{G}$; 
\item $f^{y,r}(x)  =1$ for all $x\in B_{\kappa_y(r)}(y)\cap \bar G$. 
\end{enumerate}
\end{enumerate}
\end{lemma}
\begin{proof} The proof is deferred to Appendix \ref{ap-test}.
\end{proof}

\begin{lemma} \label{lem:frs}
There exist $0<s<r<\infty$ such that $B_{\kappa_{\barz}(r)}(\barz)\subseteq \MU_{r,s} \doteq \{ x\in \partial G: |x| \leq r, d(x, \MV)
\geq s\}$ and a function $f_{r,s}\in \MH$ such that $\left<\vect,\nabla f_{r,s}(x) \right>\geq
1$ for each $x\in \MU_{r,s}$, $\vect\in d(x)$ and
$|\vect|=1$.
\end{lemma}
\proof The existence of $s,r$ follows directly from the property of $B_{\kappa_{\barz}(r)}(\barz)$ and the existence of $f_{r,s}$ follows from Theorem 2 of \cite{KanRam14}. Note that the proof the existence of $f_{r,s}$ in Theorem 2 of \cite{KanRam14} (which is the verification of part 2 of Assumption 1 of \cite{KanRam14}) does not require Assumption 2 therein to hold. \endproof

Recall that $\barz \in \MU$ and define the stopping time 
\be
\label{def-stopr}
\stopr_r  \doteq \inf \left\{ t > 0:  \zee (t) \not \in B_{\kappa_{\barz}(r)} (\barz)
\right\},  \quad r  > 0. 
\ee
From property 3 of the submartingale problem it follows that 
$\mart^h$ is a $\Q_{\barz}$-submartingale for every $h \in {\mathcal H}$. 
The next result identifies a slightly broader class of functions $g$ for which the 
stopped process $\mart^g (\cdot \wedge \theta_r)$ is a
$\Q_{\barz}$-submartingale.

\begin{lemma}[{\bf Localization Lemma}]
\label{lem:htest2}
Let $r \in (0,r_{\barz})$ and let $\stopr_r$ be
defined as in \eqref{def-stopr}.  
Given $g \in {\mathcal C}^2 (\R^J)$ such that $\langle \nabla g
(x), d \rangle \geq 0$ for all $d \in d(x)$ and $x \in B_r({\barz})
\cap \partial G$, there exists a function $h\in \MH$ such that $h(x)-h({\barz})=g(x)-g(\barz)$ for each $x\in B_{\kappa_{\barz}(r)}({\barz})$ and the
process 
$\{\mart^g (t \wedge \stopr_r), t \geq 0\}$ is a continuous $\Q_{\barz}$-submartingale.  
\end{lemma}
\begin{proof}
Let $f = f^{\barz,r}$ be a function that satisfies property 3 of Lemma
\ref{lem:cutoff} and extend the definition of $f$ to all of $\R^J$ by
setting it to be 
zero outside its compact support.  Next, given $g$ as in the statement
of the lemma, define 
\[
h(x) \doteq \left( g(x) -  \sup_{y \in B_r(\barz)} g(y)
  \right) f(x),  \quad x \in \R^J. 
\]
Then, clearly $h \in {\mathcal C}^2 (\bar G)$, with $\supp[h] \subset \supp
[f] \subset B_r (\barz)$,  and  \[ \nabla h(x) =  f(x) \nabla g(x)   + \left(g(x) -
\sup_{y \in B_r (\barz)} g(y)\right) \nabla f(x).\]    The properties of
$g$ stated in the lemma, together with the fact that 
 $f \geq 0$, $\supp[ \nabla f]
 \subset B_r(\barz)$ and  $-f \in \MH$,  imply 
 that $h \in \MH$.    By property 3 of the
  submartingale problem and the optional stopping theorem, it follows
  that $\mart^h$ and $\mart^h (\cdot \wedge \stopr_r)$ are  continuous
  $\Q_{\barz}$-submartingales.   Since $h(x) - h(\barz)   = g(x) - g(\barz)$ and $\ML f(x)=\ML g(x)$ for all $x \in B_{\kappa_{\barz}(r)}
(\barz)$, it follows  that $\mart^g ( \cdot \wedge \stopr_r)$ is equal to $\mart^h (
\cdot \wedge \stopr_r)$ under $\Q_{\barz}$, and is therefore also a continuous 
$\Q_{\barz}$-submartingale.  
\end{proof}

In what follows, for  $\barz \in \MU$, we will fix  $r_{\barz} < \infty$ and  $r \in
(0,r_{\barz})$ as in Lemma \ref{lem:cutoff}, and let $\stopr_r$ be defined as in \eqref{def-stopr}.

\subsection{A Semimartingale Property} 
\label{subs-locsm}

We now show that a suitably stopped version
of $\smart$ is a continuous $\Q_{\barz}$-semimartingale, and 
introduce some auxiliary processes that will be used in the sequel. 
 Recall the definition
of $\shift$ given in \eqref{def-shift}.  

\begin{prop}[{\bf Local Semimartingale Property}]
\label{prop-locsm} 
$\smart (\cdot \wedge
\stopr_r)$ is a continuous $\Q_{\barz}$-semimartingale, that is,  there exist a continuous 
$\Q_{\barz}$-local martingale $\martin$ with $\martin(0)=0$ and  a continuous process $\fvproc$
with $\fvproc(0) = 0$ that is  of locally bounded variation such that 
\be \label{semi} 
\smart(t\wedge \stopr_r) = 
\martin(t\wedge \stopr_r)+\fvproc(t\wedge \stopr_r), \quad t \geq 0. \ee 
Furthermore,  $\zee (\cdot \wedge \stopr_r)$ is also a continuous $\Q_{\barz}$-semimartingale. 
\end{prop}
\begin{proof}
By properties 1 and 2 of Lemma
\ref{lem:cutoff},  there exists $\alpha_{\barz} > 0$ and $n^{{\barz}}\in n(\barz)$ such that
$\lan n^{{\barz}}, d\ran \geq \alpha_{{\barz}}$ for all  $d\in d(x)$  with $|d|=
1$ and 
$x\in B_{r}(\barz) \cap \partial G$.   Let $\{e^{\ell}, \ell=1,\ldots,J\}$ be an orthonormal
basis of $\R^J$ and for $\tilde \varepsilon_{z} > 0$, let 
$\vect^\ell \doteq n^{{\barz}}+\tilde\varepsilon_{z} e^\ell$. 
Then,  we can choose $\tilde \varepsilon_{{\barz}} > 0$ small
enough such that for
each $\ell = 1,\ldots, J$, $\lan \vect^\ell, d\ran >0$ for all non-zero $d\in d(x)$ and $x\in
B_{r}(\barz) \cap \partial G$.   For each $\ell=1,\ldots, J$, define $g^{\ell}$ via 
\be
\label{def-gl}
g^\ell (x) = \langle x, v^\ell\rangle,  \quad  x \in \R^J.  
\nonumber\ee
Then, clearly $g^{\ell} \in \C^2(\R^J)$ and $\lan \nabla g^{\ell} (x),
d\ran \geq 0$ for $d \in d(x)$, $x \in
B_{r}(\barz) \cap \partial G$.  Therefore, by 
Lemma \ref{lem:htest2}, $\mart^{g^{\ell}}(\cdot
\wedge \stopr_r)$ is a continuous 
$\Q_{\barz}$-submartingale for each $\ell = 1, \ldots, J$.   
Now, for each $\ell$, 
 $\mart^{g^{\ell}} (\cdot \wedge \stopr_r)$ is equal to $\langle \shift  (\cdot \wedge \stopr_r),
 \vect_{\ell}\rangle$ and, due to  the local    
boundedness of $b (\cdot)$,  is a bounded continuous $\Q_{\barz}$-submartingale.  
Hence, by the Doob-Meyer decomposition theorem, each $\langle \shift  (\cdot \wedge \stopr_r),
 \vect^\ell\rangle$ is a continuous $\Q_{\barz}$-semimartingale. 
Since  the vectors $\vect^{\ell}, \ell = 1, \ldots, J,$ are linearly
independent,  
$\smart (\cdot \wedge \stopr_r)$ is a continuous $\Q_{\barz}$-semimartingale.  
Let $\martin^r$ and $\fvproc^r$ denote the continuous local martingale and continuous, locally bounded
variation components in the semimartingale decomposition of $\smart(\cdot\wedge \stopr_r)$.  
Since $\stopr_r \leq \stopr_{r_{\barz}}$ for $r \leq r_{\barz}$,  by uniqueness of
the semimartingale decomposition, it follows that $\Q_{\barz}$-almost
surely, $\martin^r (t) = \martin^{r_{\barz}} (t)$ for $t \leq \stopr_r$ and, since
$\smart(\cdot\wedge \stopr_r)$ is stopped at $\stopr_r$,  $\martin^r (t)
= \martin^{r_{\barz}} (\stopr_r)$ for $t \geq \stopr_r$, with an analogous
relation  holding for $\fvproc^r$.  Define  
$\martin (t) \doteq \martin^{r_{\barz}} (t)$ and $\fvproc (t) \doteq
\fvproc^{r_{\barz}}(t)$ for $t \in [0,\infty)$. 
    By uniqueness of the semimartingale decomposition,
\eqref{semi} holds.  
Lastly, since, by
\eqref{def-shift},  the process $\smart - \zee$ has almost surely finite
variation on bounded intervals, 
it follows that $\zee (\cdot \wedge \stopr_r)$ is also a continuous 
$\Q_{\barz}$-semimartingale. 
\end{proof}

As a  consequence of the semimartingale decomposition, we establish a property that will be
used in the subsequent analysis.  To state the property, we  need to introduce some
notation.  
   Let $\martin$ and $\fvproc$ be the 
processes in Proposition \ref{prop-locsm}. 
Then, for $f \in {\mathcal C}^2 (\R^J)$,  define  
\begin{eqnarray}
\nonumber 
\pushproc^f(t) &\doteq & \int_0^t  \langle
\nabla f( \zee (u)), d\fvproc (u)  \rangle + 
\frac{1}{2} \sum_{i,j=1}^J \int_0^t\displaystyle\frac{\partial^2 f}{\partial x_i \partial x_j} 
 (\zee (u))  d[\martin_i, \martin_j] (u)  \\ & & 
\label{def-chig}
- \frac{1}{2}  \sum_{i,j=1}^J   \int_{0}^{t}
a_{ij} (\zee (u))
 \frac{\partial^2 f}{\partial  x_i \partial x_j} (\zee(u))  \, du,
 \qquad t \geq 0, 
\end{eqnarray}
and 
\begin{eqnarray}
\label{chi-mart}
N^f(\cdot \wedge \stopr_r)  & \doteq  &  \mart^f(\cdot \wedge \stopr_r)
 - \pushproc^f(\cdot \wedge \stopr_r)  \\
\nonumber 
& = & f ( \zee (\cdot\wedge \stopr_r)) -f(\zee(0))- \int_0^{\cdot\wedge \stopr_r}
{\mathcal L} f (\zee (u)) \, du - \pushproc^f (\cdot \wedge
\stopr_r), 
\end{eqnarray}
where the latter equality follows  from \eqref{def-martf}. 
We also introduce some localizing stopping times.  For each $c>0$, let
\be \label{zeta} \zeta_c\doteq \left\{t\geq 0:\ |A|(t)\geq c \mbox{ or
  } \sum_{i,j=1}^J |[\martin_i,\martin_j]|(t) \geq c\right\}, \ee
where recall that $|A|$ and $|[\martin_i,\martin_j]|$ denote the total
variation processes associated with  $A$ and $[\martin_i,\martin_j]$,
respectively.  It is clear that 
\be \label{zetabd} \E^{\Q_{\barz}}[|A|(t\wedge \zeta_c)]\leq c \quad
\mbox{ and } \quad \sum_{i,j=1}^J \E^{\Q_{\barz}}[|[\martin_i,\martin_j]|(t\wedge \zeta_c)] \leq c. \nonumber \ee

\begin{lemma}
\label{lem-gtest}  
For 
each $f \in {\mathcal C}^2(\R^J)$, 
 $N^f(\cdot \wedge \stopr_r)$  is  a continuous $\Q_{\barz}$-martingale and for each $c <  \infty$, $N^f(\cdot \wedge \stopr_r \wedge \zeta_c)$ is
uniformly bounded on every finite interval of $[0,\infty)$. 
Furthermore, $\Q_{\barz}$-almost surely, 
for each $g \in {\mathcal C}^2 (\R^J)$ such that $\langle \nabla g
(x), d \rangle \geq 0$ for all $d \in d(x)$ and $x \in B_r(z)
\cap \partial G$, the process $\pushproc^g(\cdot \wedge
\stopr_r)$ is continuous and increasing. 
\end{lemma}
\proof  
First observe that 
 from (\ref{chi-mart}),  (\ref{def-chig}), \eqref{def-stopr}  and
\eqref{zeta} it is clear that for each $f\in
\C^2(\bar G)$, $c>0$ and $T <  \infty$, \begin{eqnarray} \label{Nf} &
  & \sup_{t\in [0,T]} |N^f(t \wedge \stopr_r\wedge \zeta_c)| \\ &\leq
  & 2 \sup_{x\in \bar{B}_r({\barz})}|f(x)| + T \sup_{x\in \bar{B}_r(\barz)} |\ML
  f(x)| + c \sup_{x\in \bar{B}_r({\barz})} |\nabla f(x)| \nonumber \\ & & +
  \frac{c}{2} \sup_{x\in \bar{B}_r({\barz})}
  \left|\sum_{i,j=1}^J\frac{\partial^2 f}{\partial x_i \partial x_j}
    (x)\right| + \frac{T}{2} \sup_{x\in \bar{B}_r({\barz})}
  \left|\sum_{i,j=1}^J a_{ij}(x)\frac{\partial^2 f}{\partial
      x_i \partial x_j} (x)\right|. \nonumber \end{eqnarray} 
Thus, $N^f(\cdot \wedge \stopr_r\wedge \zeta_c)$ is uniformly bounded
on every finite interval of $[0,\infty)$.  

By  \eqref{def-shift} and Proposition \ref{prop-locsm}, with $M$ and $A$ as in
\eqref{semi}, it follows that $\zee (\cdot \wedge \stopr_r)$  
is a continuous $\Q_{\barz}$-semimartingale with local martingale component
$M(\cdot\wedge \stopr_r)$ and locally finite variation component $A(\cdot\wedge \stopr_r) + \int_0^{\cdot\wedge \stopr_r} b
(Z(u)) \, du$.     Then, for $f \in {\mathcal
  C}^2 (\bar{G})$, 
 It\^{o}'s formula  shows that for $t \geq 0$, 
\begin{eqnarray}
\nonumber
f(\zee (t \wedge \stopr_r)) & = &  f(\zee (0)) + \int_0^{t \wedge \stopr_r}  \langle
\nabla f( \zee (u)), b (\zee (u))  \rangle \, du  + \int_0^{t \wedge \stopr_r} \langle
\nabla f( \zee (u)), d\martin (u)  \rangle \\
\label{form-ito} 
& & + \int_0^{t \wedge \stopr_r}  \langle
\nabla f( \zee (u)),  d \fvproc(u)  \rangle 
+ \frac{1}{2}  \sum_{i,j=1}^J \int_0^{t \wedge \stopr_r}
\frac{\partial^2 g}{\partial x_i \partial x_j} (\zee (u)) 
d [\martin_i, \martin_j ] (u). 
\end{eqnarray}
Substituting the definitions 
of $\mart^f$ and $\pushproc^f$ from \eqref{def-martf} and
\eqref{def-chig}, respectively, into \eqref{form-ito}, we obtain 
\begin{eqnarray}
\label{chig-equality}
 \mart^f (\cdot \wedge \stopr_r ) & = & 
\pushproc^f (\cdot \wedge \stopr_r) + \int_0^{t \wedge \stopr_r}
\langle \nabla f (\zee (u)), d \martin_u \rangle. 
\end{eqnarray}
If $f \in \MH$, then  property 3 of the submartingale problem,  
the optional stopping theorem, the
definition of $\stopr_r$ and the uniform 
boundedness of $f$ and ${\mathcal L} f$ 
imply that $\mart^f(\cdot \wedge \stopr_r)$ is a  continuous $\Q_{\barz}$-submartingale
 of class DL. 
 Thus, by the Doob-Meyer 
decomposition theorem,   $\mart^f(\cdot \wedge \stopr_r)$ admits a (unique) decomposition into a continuous $\Q_{\barz}$-martingale
and a continuous increasing process.   Now, the stochastic integral    $\int_{0}^{\cdot \wedge \stopr_r} \langle
\nabla f (\zee (u)), d \martin (u) \rangle$ on the right-hand side of
\eqref{chig-equality} is a continuous $\Q_{\barz}$-local martingale, and from 
definition \eqref{def-chig}  it is clear that
$\pushproc^f (\cdot\wedge \stopr_r)$ is  a continuous finite-variation process. 
Thus, the right-hand side of \eqref{chig-equality}  is in fact 
the Doob-Meyer decomposition of $\mart^f(\cdot \wedge \stopr_r)$.  
In particular, this implies that $\Q_{\barz}$-almost surely, $\pushproc^f (\cdot\wedge \stopr_r)$ is
a continuous increasing process and 
$N^f(\cdot \wedge \stopr_r)   = \int_0^{t \wedge \stopr_r} \langle \nabla f (\zee (u)), d \martin (u) \rangle$ is a continuous $\Q_{\barz}$-martingale.    
We now show that this property holds for all $f \in {\mathcal C}_c^2
(\bar{G})$, and not just $f \in \MH$. 
Let $f_{r,s}$ be the function in Lemma \ref{lem:frs}.  Therefore, for each $g\in \C_c^2(\bar G)$, there exists
a constant $C>0$ such that $g+C f_{r,s} \in \MH$.   Then, since $N^f(\cdot \wedge \stopr_r\wedge \zeta_c)$
is a continuous $\Q_{\barz}$-martingale for 
 both $f = f_{r,s}$ and $f = g+C f_{r,s}$, it also a continuous $\Q_{\barz}$-martingale for $f
 = g\in \C_c^2(\bar G)$. 
Since the process $\zee (\cdot \wedge \theta_r)$ lives in
$B_{\kappa_{\barz}(r)}(\barz)$, then for any function $f\in \C^2(\bar G)$, there
exists a function $g\in \C^2_c(\bar G)$ such that $f=g$ on
$B_{\kappa_{\barz}(r)}(\barz)$. Thus, $N^f(\cdot \wedge \stopr_r)$ is  a continuous
$\Q_{\barz}$-martingale for each $f\in \C^2(\bar G)$.

It only remains to  establish the last assertion of the lemma. 
Let $\MH_0$ be the countable dense subset of $\MH$ mentioned in Remark
\ref{rem-MH}. Since $\MH_0$ is
countable, the continuity and the monotonicity of $\pushproc^f
(\cdot\wedge \stopr_r)$ hold $\Q_{\barz}$-almost surely (simultaneously) for
all $f \in\MH_0$.   Now, note that  for each $T>0$, $\pushproc^{f_n} -
\pushproc^f = \pushproc^{f-f_n}$ and  $\sup_{t\in [0,T]}
|\pushproc^{f-f_n} (t \wedge \stopr_r \wedge \zeta_c)|$,
  is bounded above by (the sum of the last four terms on the) right-hand
side of (\ref{Nf}) with $f_n-f$ in place of $f$.  
Since any $f \in \MH$ can be
approximated by  a sequence $\{f_n\}$ in $\MH_0$ in the strong sense made precise 
in \eqref{fgk}, this implies that as $n\rightarrow \infty$, 
$\sup_{t\in [0,T]} |\pushproc^{f_n} (t \wedge \stopr_r\wedge
\zeta_c)-\pushproc^f (t \wedge \stopr_r\wedge \zeta_c)|$  converges to
$0$ both pointwise (that is, for each $\omega\in \ccspace$) and in
$L^1(\Q_{\barz})$.  It follows that $\pushproc^f (\cdot \wedge
\stopr_r\wedge \zeta_c)$ and hence $\pushproc^f (\cdot \wedge
\stopr_r)$ is continuous and increasing on $[0,T]$. Since $T$ is
arbitrary, the desired property holds $\Q_{\barz}$-almost surely for
all $f \in \MH$. For each $g \in {\mathcal C}^2 (\R^J)$ such that $\langle \nabla g
(x), d \rangle \geq 0$ for all $d \in d(x)$ and $x \in B_r({\barz})
\cap \partial G$, by Lemma \ref{lem:htest2}, there exists a function
$f\in \MH$ such that $f(x)-f(z)=g(x)-g(z)$ on $B_{\kappa_{\barz}(r)}({\barz})$.  Because the process $\zee (\cdot \wedge \theta_r)$ lives in
$B_{\kappa_{\barz}(r)}({\barz})$, it follows that $\pushproc^g (\cdot \wedge
\stopr_r)=\pushproc^f (\cdot \wedge \stopr_r)$ and hence, $\pushproc^g
(\cdot \wedge \stopr_r)$ is continuous and increasing. 
This completes the
proof of the lemma. 
\endproof

\subsection{Behavior of the Semimartingale in the Interior of the Domain}
\label{subsub-submart}

We now characterize the behavior of the components $\fvproc(\cdot \wedge \stopr_r)$ and $\martin(\cdot \wedge \stopr_r)$
of the semimartingale decomposition \eqref{semi}  of $\smart (\cdot
\wedge \theta_r)$  in $G$.  
As in Section \ref{sec:BM}, let $\{G_m,\ m\in \N\}$ be a  
sequence of bounded domains with $\bar{G}_m \subset G_{\tilde{m}}$ for
$m < \tilde{m}$ and  $\cup_{m \in \mathbb N} G_m = G$.   
For each $\omega \in \MC [0,\infty)$ 
and fixed $m \in \mathbb N$, set $\stopoutt_0^m (\omega) \doteq 0$ and
for
 $k \in \N$,  recursively define
\begin{eqnarray}
\label{def-stopin2}
\stopin_k^m = \stopin_{k}^m (\omega) & \doteq & \inf\{t\geq
\stopoutt_{k-1}^m: \zee (t)\in
 \bar{G}_m\} ,\\
\label{def-stopout2}
\stopoutt_k^m = \stopoutt_{k}^m (\omega) & \doteq & \inf \{ t \geq
\stopin_k^m:  \zee (t) \in
\partial G\}.
\end{eqnarray}

\begin{lemma}\label{lem:Urep1}
Let $\martin(\cdot \wedge \stopr_r)$ and $\fvproc(\cdot \wedge \stopr_r)$, respectively, be the continuous $\Q_{\barz}$-local martingale and continuous bounded variation
processes that arise in the local semimartingale decomposition of
$\smart(\cdot \wedge \stopr_r)$ given 
in (\ref{semi}).   Then $\Q_{\barz}$-almost surely, for every $k, m \in \N$, 
\begin{eqnarray} \label{disU}
\martin(t\wedge \stopoutt_k^m\wedge \stopr_r)-\martin(t\wedge
\stopin_k^m\wedge \stopr_r)  =\int_{t \wedge
\stopin_k^m \wedge \stopr_r}^{t \wedge
\stopoutt_k^m \wedge \stopr_r}
a^{1/2} (\zee (u)) d \brm (u), \ t \geq 0, 
\end{eqnarray}
Moreover, $\Q_{\barz}$-almost surely, for every $t \geq 0$, 
\be \label{bdjump}\int_0^{t \wedge \stopr_r} \ind_G( \zee (u))
d|\fvproc|(u) = 0,
\ee
and 
\be
\label{bdjump2}
 \int_0^{t \wedge \stopr_r} \ind_G( \zee (u))
\,d[\martin_i,\martin_j](u) = \int_{0}^{t \wedge \stopr_r} \ind_G(\zee (u))a_{ij}(
\zee (u))du.
 \ee
Moreover, $\Q_{\barz}$-almost surely, for each $f \in \C^2(\R^J)$, the process $\pushproc^f$ defined  in
\eqref{def-chig} satisfies 
\be
\label{chig} 
\int_0^{t \wedge \stopr_r} \ind_G (Z(u)) d \pushproc^f (u) = 0,  \quad
\forall t \geq 0. 
\ee
\end{lemma}
\begin{proof}
For $k, m \in \N$, 
  let $\{\Q_{\omega'}^{k,m}, \omega' \in  \ccspace\}$  be the regular conditional probability of $\Q_{\barz}$ 
given $\MM_{\stopin_k^m}$.    Then 
by Lemma \ref{lem:submartperv} there exists a set $\MN^{k,m} \in \MM_{\stopin_k^m}$ such that 
$\Q_{\barz} (\MN^{k,m}) = 0$ and for all $\omega \not \in
\MN^{k,m}$ with $\stopin_k^m (\omega')<\infty$, $\tau_k^m (\omega) = \tau_1 (\omega (\stopin_k^m (\omega')
+ \cdot))$ and so, by Lemma \ref{lem:sderep},
 for each $t \geq 0$, 
\[   \shift ( t\wedge \stopoutt_k^m \wedge \stopr_r) - \shift (t \wedge
\stopin_k^m \wedge \stopr_r) 
= \int_{t \wedge
\stopin_k^m \wedge \stopr_r}^{t \wedge
\stopoutt_k^m \wedge \stopr_r}
a^{1/2} (\zee (u)) d \brm (u), 
\]
and $Y ( t \wedge
\stopoutt_k^m \wedge \stopr_r) - Y ( t \wedge \stopin_k^m \wedge
\stopr_r) = 0$.   Comparing this with \eqref{semi}, 
we have \eqref{disU} 
and $\fvproc (t \wedge \stopr_r) - \fvproc(\stopin_k^m \wedge \stopr_r)= 0$ for  $t \in
[\stopin_k^m, \stopoutt_k^m]$. 
Since the latter equality holds for all $k \in \N$ and, because, for  
$u \in [0,\stopr_r]$, $\zee (u) \in
G_m$ implies $u \in \cup_k [ \stopin_k^m \wedge 
\stopr_r, \stopoutt_k^m\wedge \stopr_r]$, it follows that $\Q_{\barz}$-almost
surely, 
\[  \int_0^{\stopr_r}  \ind_{G_m} (\zee (u)) d|\fvproc|(u) = 0 \]  and 
\[  \int_0^{\stopr_r} \ind_{G_m}( \zee (u))
\,d[\martin_i,\martin_j](u) = \int_{0}^{\stopr_r} \ind_{G_m}( \zee
(u))a_{ij}( \zee (u))du. 
\] Taking limits as $m \ra \infty$, recalling that 
$\bigcup_m G_m
= G$ and applying 
the dominated convergence theorem,  we obtain \eqref{bdjump} and
\eqref{bdjump2}.  In turn, when combined with \eqref{def-chig},  \eqref{bdjump} and \eqref{bdjump2}
imply \eqref{chig}. 
\end{proof}

\subsection{A Random Measure}
\label{subs-integrep}

We would like to extend the result established in 
\eqref{disU} for when the process lies in $G$  
 to show that  $M$ is equal to 
$\int_0^{\cdot \wedge \stopr_r} a^{1/2} (Z(u)) dW(u)$.  
For this, we need to characterize the boundary behavior of $M(\cdot \wedge \stopr_r)$ and
$A(\cdot \wedge \stopr_r)$.   As a first step,  in Proposition \ref{prop-chig} below,  we introduce a random measure and
establish a convenient integral representation for the process
$\pushproc^f$.  
We then establish  some additional properties of the random measure 
which are used in the next section to prove the stochastic integral representation for $M(\cdot \wedge \stopr_r)$. 

\begin{prop}
\label{prop-chig} 
For each $\omega \in \ccspace$, there exists a $\sigma$-finite measure 
 $\tilde{\mu} (\omega, \cdot)$ on $(\R_+ \times
S_1(0), {\mathcal B} (\R_+ \times S_1(0))$, and 
there exists a subset  $\Omega_0 \subset \ccspace$ 
with $\Q_{\barz} (\Omega_0) = 1$ such that 
for every $\omega \in \Omega_0$,  for all  $f \in \MH$ and $t \geq 0$, 
\be \label{xif31} \pushproc^f(\omega, t\wedge \stopr_r(\omega)) = \int_{\intset_t(\omega)}
  \left<\vect,\nabla f(\zee (\omega,u))\right>  \tilde \mu(\omega,du,d\vect), 
 \ee 
where, for $t \geq 0$, 
\be
\label{def-intset}   \intset_t(\omega) \doteq 
\{ (u,\vect ) \in [0,t]
  \times S_1(0): \zee(\omega,u) \in \partial G\setminus \MV, \vect \in d(\zee (\omega,u)) \}. 
\nonumber\ee
Moreover, 
 $\int_{\intset_t}  \left< \vect, g(\zee(u))\right>  \tilde
 \mu(\cdot,du,d\vect), t \geq 0,$ is a continuous stochastic
 process for any continuous function
$g: \R^J \ra \R^J$. 
\end{prop}

The proof of Proposition \ref{prop-chig} is 
functional analytic in nature, involving an application of the
Hahn-Banach and Riesz Representation theorems for random linear
functionals, and 
 is deferred to Appendix \ref{ap-chig}.    
 We now identify  
a family of martingales associated with the random measure
$\tilde{\mu}$
 from Proposition \ref{prop-chig}. 
For $t \geq 0$, define 
\be
\label{def-alpha}
 \alpha (t) \doteq \smart (t \wedge \stopr_r) - \int_{\intset_{t}}   \vect \tilde{\mu} (\cdot,du, d\vect) \mbox{ with } \alpha(0)=0.
\ee
Applying the last assertion of Proposition \ref{prop-chig}
with $g (x) = e^\ell$ for $\ell=1,\ldots, J$, it follows that 
$\alpha$ is a well defined continuous stochastic process. 
We now identify some exponential martingales
associated with the process $\alpha$.   Recall the family of stopping times
$\{\zeta_c, c > 0\}$ defined by \eqref{zeta}.

\begin{lemma}
\label{lem-alpha}
For each $c>0$ and every bounded, ${\mathcal M}_t$-adapted process 
$\{\vartheta (t), t \geq 0\}$, 
\be
\label{def-expmart}  \exp \left\{  \int_0^{t \wedge \stopr_r\wedge \zeta_c} \langle \vartheta (u), d
  \alpha (u) \rangle - \frac{1}{2} \int_0^{t \wedge \stopr_r\wedge \zeta_c} \langle 
\vartheta (u), a (\zee (u)) \vartheta (u) \rangle \, du \right\},  \quad
t \geq 0, 
\ee
is a continuous $\Q_{\barz}$-martingale. 
\end{lemma}
\begin{proof}  
Fix $c>0$. We first reduce the proof of the lemma to  showing the 
result for constant $\vartheta(\cdot)$, namely, to showing that 
for all $\vartheta \in \R^J$, \be
\label{alpha-mart}
\exp\left\{\left<\vartheta, \alpha(t\wedge \theta_r\wedge \zeta_c) \right> - \frac{1}{2}\int_0^{t\wedge \theta_r\wedge \zeta_c}
  \left<\vartheta, a(\zee(u))\vartheta\right> du\right\}
\ee is a continuous 
 $\Q_{\barz}$-martingale. 
Indeed, given the local boundedness of $a$,  the non-degeneracy condition
\eqref{eq-ue},  and the continuity of $\alpha (\cdot \wedge
\stopr_r)$, it follows that the conditions of Theorem 3.1 (and
therefore Theorem 3.2) of
\cite{StrVar69} are fulfilled with $P, \xi$ and  $s$ therein replaced by 
$\Q_{\barz}, \alpha(\cdot \wedge \stopr_r\wedge \zeta_c)$ and $0$, respectively.  
Therefore, we can apply part (v) of Theorem 3.2 of \cite{StrVar69},  with $\xi = \alpha$ and
$\theta = \vartheta$, to conclude that  for every bounded adapted
process $\vartheta(\cdot)$, the process in \eqref{def-expmart} is a continuous
 $\Q_{\barz}$-martingale. 

To show that the process in \eqref{alpha-mart} is a continuous
$\Q_{\barz}$-martingale, we  establish a slightly more general result.  
Suppose $f \in {\mathcal C}^2 (\bar G)$  is uniformly positive, that
is, there exists $c_f > 0$ with $\inf_{x \in \bar{G}} f(x) >
c_f$.  Then for $t \geq 0$ and $\omega \in {\mathcal C}$, define
\[  
 V^f (\omega,t)\doteq  \exp\left\{-\int_0^{t} \frac{\ML
    f(\zee(\omega,u))}{f(\zee(\omega,u))}du - 
\int_{\intset_t(\omega)}
 \left<\vect,\frac{    \nabla f(\zee(\omega,u))}{f(\zee(\omega,u))}\right>  \tilde \mu(\omega,du,d\vect)
\right\},
\]
 and  
\begin{eqnarray}
\label{def-procf}
 \Proc^f (t) & \doteq & 
f(\zee(t\wedge \stopr_r\wedge \zeta_c)) V^f(t\wedge \stopr_r\wedge \zeta_c),  \quad t \geq 0, 
\end{eqnarray} 
Applying the last assertion of Proposition \ref{prop-chig} with
$g(x) = \nabla f(x)/f(x)$, it follows that $V^f$ and therefore
$\Proc^f$ are 
well-defined, continuous 
stochastic processes.  
We now claim that for any  uniformly positive $f \in {\mathcal C}^2(\bar G)$,   $\Proc^f$  is a positive continuous 
$\Q_{\barz}$-martingale starting from $f(z)$.  
Suppose the claim were true. Then  for fixed  $\vartheta\in \R^J$, define $f (x) = f^\vartheta (x) = \exp \{\langle \vartheta,
x \rangle\}$, $x \in \R^J$.  Then $f$ is clearly  uniformly positive,  lies in ${\mathcal
  C}^2(\bar{G})$ and satisfies 
\[\frac{\ML f(x)}{f(x)}= \left< \vartheta, b(x) \right> + \frac{1}{2}
\left< \vartheta, a(x)\vartheta \right> \quad \mbox{ and } \quad  \frac{\nabla 
  f(x)}{f(x)} = \vartheta.\]   Substituting this into the definition
of $V^f$ and $\Proc^f$ and recalling the definition of $\alpha$ in 
\eqref{def-alpha}, it is easy to verify that  the process in
\eqref{alpha-mart} is equal to $\Proc^f (t), t \geq 0,$ and hence, is a
continuous $\Q_{\barz}$-martingale.   

Thus, it only remains to establish the claim.  Fix 
$f \in \C^2(\bar G)$ that is uniformly positive.  
To prove the claim, we will first establish the relation 
\be
\label{rel-Procf}
 \Proc^f (t)  - f(Z(0)) = \bar{N}^f (t \wedge \stopr_r \wedge \zeta_c),  \quad t \geq 0, 
\ee
where  
\be
\label{def-tildenf} 
 \bar{N}^f (t) \doteq  N^f (t) V^f(t) - \int_0^t N^f (u)  d V^f (u), \quad
 t \geq 0. 
\ee
and then show that $\bar{N}^f (\cdot \wedge \stopr_r\wedge \zeta_c)$ is a continuous $\Q_{\barz}$-martingale. 
Using the relation (\ref{chi-mart}) for $N^f$ in the first and last
lines below, the definition of
$V^f$, the representation \eqref{xif31} for $\pushproc^f$, and integration-by-parts, 
we obtain 
\begin{eqnarray*}
&&\int_0^t N^f (u\wedge \stopr_r\wedge \zeta_c) d V^f (u\wedge \stopr_r\wedge \zeta_c) \\ &&\quad= \int_0^{t\wedge \stopr_r\wedge \zeta_c} \left(f(\zee(u))-f(\zee(0)) -\int_0^{u}
{\mathcal L} f (\zee (s)) \, ds - \pushproc^f (u)\right) dV^f(u)  \\ &&\quad= -\int_0^{t\wedge \stopr_r\wedge \zeta_c} V^f(u)d\left(\int_0^u\ML
    f(\zee(s))ds +\pushproc^f (u)\right)   \\ && \quad \quad - \int_0^{t\wedge \stopr_r\wedge \zeta_c} \left(f(\zee(0)) +\int_0^{u}
{\mathcal L} f (\zee (s)) \, ds + \pushproc^f (u)\right) dV^f(u) \\ &&\quad= - V^f(t\wedge \stopr_r\wedge \zeta_c)\left(\int_0^{t\wedge \stopr_r\wedge \zeta_c}\ML
    f(\zee(u))du +\pushproc^f (t\wedge \stopr_r\wedge \zeta_c)\right) \\ &&\quad\quad  + \int_0^{t\wedge \stopr_r\wedge \zeta_c} \left(\int_0^{u}
{\mathcal L} f (\zee (s)) \, ds + \pushproc^f (u)\right) dV^f(u) \\
&&\quad \quad -\int_0^{t\wedge \stopr_r\wedge \zeta_c} \left(f(\zee(0)) +\int_0^{u}
{\mathcal L} f (\zee (s)) \, ds + \pushproc^f (u)\right) dV^f(u) \\ &&\quad= f(\zee(0))-  V^f(t\wedge \stopr_r\wedge \zeta_c)\left(f(\zee(0))+\int_0^{t\wedge \stopr_r\wedge \zeta_c}\ML
    f(\zee(u))du +\pushproc^f (t\wedge \stopr_r\wedge \zeta_c)\right)
  \\
&&\quad= f(Z(0))  - V^f(t \wedge \stopr_r \wedge \zeta_c)\left( -N^f (t
  \wedge \stopr_r \wedge \zeta_c) + f (Z(t \wedge \stopr_r \wedge
  \zeta_c))\right). 
\end{eqnarray*}
Together with the definitions of $\bar{N}^f$ and $\bar{H}^f$ in
\eqref{def-tildenf} and \eqref{def-procf}, respectively, this proves \eqref{rel-Procf}.

We now show that $\bar{N}^f (\cdot \wedge \stopr_r\wedge \zeta_c)$ is a $\Q_{\barz}$-martingale. 
By  Lemma  \ref{lem-gtest}, 
$N^f(\cdot \wedge \stopr_r\wedge \zeta_c)$ is a continuous $\Q_{\barz}$-martingale that is uniformly
bounded on every finite interval and $V^f(\cdot \wedge \stopr_r\wedge \zeta_c)$ is a continuous finite
variation process.   
Thus, the desired result will follow from Lemma 2.1 of \cite{StrVar71} with $\phi = N^f(\cdot \wedge \stopr_r\wedge \zeta_c)$ and
$\psi = V^f(\cdot \wedge \stopr_r\wedge \zeta_c)$ if we can show that 
$\E[|V^f|(t \wedge \stopr_r\wedge \zeta_c)] < \infty$ for every $t >
0$,  where recall that $|V^f| (t)$ denotes the total variation of $V^f$ on
$[0,t]$.   Let  $C_f < \infty$ be the maximum of
the supremum of $f$ and the suprema of its
first and second partial derivatives over $\bar{B}_{\kappa_{\barz}(r)}(\barz)$. 
  Choose $f_{r,s}$ from Lemma \ref{lem:frs}. It follows from (\ref{xif31}) and (\ref{def-chig}) that \begin{eqnarray*}  \tilde \mu(\omega,\intset_{t\wedge \stopr_r(\omega)\wedge \zeta_c(\omega)}(\omega))  &\leq &  \pushproc^{f_{r,s}}(\omega, t\wedge \stopr_r(\omega)\wedge \zeta_c(\omega)) \\ &\leq & \left(c J  + \frac{cJ^2}{2} + \frac{tJ^2}{2} \sum_{i,j=1}^J\sup_{x\in \bar{B}_r(\barz)}  |a_{ij}|(x) \right) C_{f_{r,s}}.
 \end{eqnarray*}
 In turn, this implies that for all $t > 0$, 
\begin{eqnarray*} & & \E[ |V^f| (t \wedge \stopr_r\wedge \zeta_c) ] \\
  &\leq & t \exp \left\{  \frac{t J \sup_{x\in \bar{B}_r(\barz)}  |\ML
     f(x)|}{c_f} +   \frac{J C_f}{c_f}\left(\frac{3cJ^2}{2} +
      \frac{tJ^2}{2} \sum_{i,j=1}^J\sup_{x\in \bar{B}_r(\barz)}
      |a_{ij}|(x) \right) C_{f_{r,s}} \right\} 
\\ &<& \infty, 
\end{eqnarray*}
as desired.   This completes the proof of the lemma. 
\end{proof}

\subsection{A Boundary Property of the Martingale Component}
\label{subs-martbound}

We start by showing that the trace of the quadratic variation of $M$
vanishes on the boundary $\partial G$. 

\begin{prop} \label{lem:q8}  
The continuous $\Q_{\barz}$-local martingale $\martin(\cdot \wedge
  \stopr_r)$ in
  the decomposition \eqref{semi} for $\shift (\cdot \wedge
  \stopr_r)$ satisfies $\Q_{\barz}$-almost surely, 
\be \label{q8} \int_{0}^{\stopr_r}
  \ind_{\partial
    G}( \zee (u))\,d[\martin_i,\martin_i](u)=0,  \quad i = 1, \ldots,
  J. 
\ee
\end{prop}
\begin{proof}  
For each $\vartheta \in \R^J$, 
choosing $\vartheta(\cdot) = \vartheta \ind_{\partial
  G}(\zee (\cdot))$ in Lemma \ref{lem-alpha}, we see 
that for each $c>0$, \[\exp\left\{\left< \vartheta, \int_0^{t\wedge \theta_r\wedge \zeta_c}
    \ind_{\partial G}(\zee (u)) d\alpha(u) \right> -
  \frac{1}{2}\int_0^{t\wedge \theta_r\wedge \zeta_c} \ind_{\partial G}(\zee (u))
  \left<\vartheta, a(\zee (u))\vartheta\right> du\right\}\] is a continuous 
$\Q_{\barz}$-martingale. Since $\Q_{\barz}$-almost surely $Z$ spends zero Lebesgue time on
the boundary by Proposition \ref{lem-bdary}, this implies that for each $\vartheta \in \R^J$
and $t \geq 0$,
 \begin{eqnarray*} 
& & \E^{\Q_{\barz}}\left[\exp\left\{\left< \vartheta, \int_0^{t\wedge \theta_r\wedge \zeta_c}
      \ind_{\partial G}(\zee (u)) d\alpha(u) \right>\right\}\right] =
  1.
\end{eqnarray*}
 Hence,  for each $t\geq 0$, we have  $\Q_{\barz}$-almost
 surely, $\int_0^{t\wedge \theta_r\wedge \zeta_c} \ind_{\partial G}(\zee(u)) d\alpha(u) =
0$, 
which in turn implies that  for each $i=1,\ldots, J$,
$\int_0^{\theta_r\wedge \zeta_c} \ind_{\partial G}(\zee (u))
d[\alpha_i,\alpha_i](u) = 0.$ 
By letting $c\rightarrow \infty$, we have 
\be \label{alpha0} \int_0^{\theta_r} \ind_{\partial G}(\zee (u))
d[\alpha_i,\alpha_i](u) = 0, \quad  i=1,\ldots, J.  \ee
From (\ref{semi}) and \eqref{def-alpha} we know that  \[\alpha(t\wedge
\theta_r) = \martin(t \wedge
  \stopr_r)+\fvproc(t \wedge
  \stopr_r) -
\int_{\intset_t}
\vect\tilde \mu(du,d\vect ). \] 
Since $\fvproc(t \wedge
  \stopr_r) -\int_{\intset_t}$ 
$\vect\tilde \mu(du,d\vect )$ is a process with locally bounded variation, it follows
that $[\alpha_i, \alpha_i] = [\martin_i, \martin_i]$, and 
(\ref{q8}) follows directly from (\ref{alpha0}). This completes the proof of
Proposition \ref{lem:q8}. 
\end{proof} 

\begin{lemma}\label{lem:Urep2}
We have  $\Q_{\barz}$-almost surely, for $t \geq 0$, 
 \be \label{Urep2a}
\martin(t \wedge
  \stopr_r)=\int_0^{t\wedge \stopr_r}a^{1/2}(
\zee (u))d\brm(u) 
\ee
and
\be
\label{Urep2b}
\fvproc(t \wedge
  \stopr_r)= Y(t\wedge \stopr_r). 
 \ee 
\end{lemma}
\begin{proof}
From \eqref{eq:y} and \eqref{semi}, it is clear that \eqref{Urep2b}
follows from \eqref{Urep2a}.   
To establish \eqref{Urep2a}, 
let the sequences of stopping times 
$\stopin_k^m\uparrow \infty$, $m \in \N$, and $\stopoutt_k^m \uparrow \infty$, $m \in
\N$, be defined by \eqref{def-stopin2} and \eqref{def-stopout2}. 
We use the fact that  $\Q_{\barz}$-almost
surely $\stopin_k^m \uparrow \infty$ and $\stopoutt_k^m \uparrow \infty$
as $k \ra \infty$,  to conclude  that for any $t \geq
0$, 
 \begin{eqnarray} 
 \label{Udecomp1}\
\martin(t \wedge
  \stopr_r) &=& M^{1,m} (t) + M^{2,m} (t),   \quad m \in \N,  
\end{eqnarray}
where for $m \in \N$, 
\begin{eqnarray*}
M^{1,m} (t) & \doteq &   \sum_{k\in \N} \left[\martin(t\wedge \stopoutt_k^m\wedge
    \stopr_r)-\martin(t\wedge \stopin_k^m\wedge \stopr_r) \right], \\
M^{2,m} (t) & \doteq &  \sum_{k\in \N} \left[\martin(t\wedge
    \stopin_k^m\wedge \stopr_r)-\martin(t\wedge \stopoutt_{k-1}^m\wedge
    \stopr_r) \right].
\end{eqnarray*}
Now,  by (\ref{disU}) we have  for any $m \in \N$ and $t \geq 0$, 
\begin{eqnarray*}
M^{1,m} (t) - \int_0^{t \wedge \stopr_r} a^{1/2} (\zee (u)) d\brm(u) = 
-\sum_{k\in \N}\int_{t \wedge \stopoutt_{k-1}^m\wedge
  \stopr_r}^{t\wedge \stopin_k^m\wedge \stopr_r}a^{1/2}( \zee (u)) d
\brm(u). 
\end{eqnarray*}
The last term is  a square integrable $\Q_{\barz}$-martingale, with 
covariation 
\be
\label{qov-stop} \int_0^{t \wedge \stopr_r}  \ind_{\bigcup_{k \in \N} [\stopoutt_{k-1}^m,
  \stopin_k^m]} (u) \ind_G (\zee (u)) a_{ij} (\zee (u)) \, du,  \quad
i, j = 1, \ldots, J. 
\ee
Each integral in \eqref{qov-stop} converges almost surely to zero as $m
\ra \infty$ 
because 
\be
\label{lim-ind}  \lim_{m \ra \infty}  \ind_{\cup_{k \in \N} [\stopoutt_{k-1}^m,
  \stopin_k^m]} (u)  =  \ind_{\partial G} (\zee (u)). 
\ee
Thus, we have shown that for any $t > 0$,  as $m \ra \infty$, 
$M^{1,m}(t)$ converges in $\mathbb{L}^2(\Q_{\barz})$ to $\int_0^{t \wedge
  \stopr_r} a^{1/2} (\zee (u)) \, d\brm(u)$.     In view
of \eqref{Udecomp1}, to
complete the proof of \eqref{Urep2a}, it suffices to show that
$M^{2,m}(t)$ converges  to zero 
in $\mathbb{L}^2(\Q_{\barz})$, as $m \ra \infty$.  
 Now, for each $i=1,\ldots, J$, by \eqref{lim-ind} and the bounded
 convergence theorem, 
\begin{eqnarray}
 \lim_{m\rightarrow \infty}
\E^{\Q_{\barz}}\left[|M^{2,m}(t)|^2\right]  & = 
  &  \sum_{i=1}^J\lim_{m\rightarrow \infty} \E^{\Q_z}\left[\int_0^{t\wedge
      \stopr_r}\ind_{\cup_{k \in \N}[\stopoutt_{k-1}^m,\stopin_k^m]}(u) d [\martin_i,\martin_i](u)\right] \nonumber \\
& = & \sum_{i=1}^J\E^{\Q_z}\left[\int_0^{t\wedge \stopr_r}\ind_{\partial  G}( \zee(u))\, d [\martin_i,\martin_i](u)\right], \label{Udecomp2}
\nonumber\end{eqnarray}
which is identically zero due to  Proposition \ref{lem:q8}. This  completes the proof. 
\end{proof}

\bigskip 

\subsection{Proof  of the Reflection Property for $\barz\in \MU$} 
\label{subsub-locesp}

In this section we establish the reflection property
\eqref{eq-locref}. The proof relies on the following simple geometric
property.

\begin{lemma} \label{lem:ind}
Let $\Theta$ be a convex cone with vertex at $0$ and let
\be \label{Xi}  \conset \doteq \{ v \in \R^J:\ \lan v,b\ran \geq 0 \mbox{ for
  each } b\in \Theta\}.\ee
If there exists $\mvect \in \R^J$ such that $\lan v, \mvect \ran \geq
0$  for all $v\in \conset$, then $\mvect \in \Theta$.
\end{lemma}
\begin{proof} We use an argument by contradiction to establish  the lemma. Suppose that there exists
$\mvect \in \R^J\sm \Theta$ such that $\lan v, \mvect \ran \geq 0$ for
all $v \in \conset$.
Let $P_{\Theta} : \R^J \rightarrow\Theta$ be the metric projection onto
the cone $\Theta$ (which assigns to each point $x\in \R^J$  the point
on $\Theta$ that is closest to $x$).
Since $P_{\Theta}(\mvect )-\mvect$  is the inward normal to $\Theta$ at
$P_{\Theta}(\mvect)$, $\Theta$
is convex and has vertex at the origin, we have \[\lan
P_{\Theta}(\mvect)-\mvect, b\ran \geq 0
\mbox{ for each } b\in \Theta.\] This implies that $P_{\Theta}(\mvect )-\mvect \in
\conset$ and hence, by the assumed property of $\mvect$, $\lan
P_{\Theta}(\mvect)-\mvect, \mvect \ran \geq 0$.   On the other hand, since $P_{\Theta}$
is non-expansive and $P_{\Theta} (\mvect)\neq \mvect$ because $\mvect \not \in \Theta$, it follows that  $\lan P_{\Theta}(\mvect), \mvect\ran
< \lan \mvect, \mvect \ran$, which yields a contradiction.  \end{proof}
\bigskip

We now use this to establish the local reflection
property when $\barz \in \MU$.  

\begin{lemma}
\label{lem-locrefu}
\[ \Q_{\barz}\left( \y(t\wedge \stopr_r)   \in \conv
   \left[ \huge\cup_{u \in [0,t \wedge \stopr_r]} d(\zee(u))
   \right]\right) =1.\]
\end{lemma}
\begin{proof}
 For each $\varepsilon>0$ and $y\in \MU$, by Lemma \ref{lem:cutoff} 
 we can choose $r_{y,\varepsilon}<\varepsilon$ 
such that properties 1--3 of Lemma \ref{lem:cutoff} are satisfied for $r =
r_{y,\varepsilon}$.  By applying
Lemma \ref{lem:cover} with $\mathcal{O}_y= B_{\kappa_y(r_{y,\varepsilon})}(y)$, $y\in \MU$, and $\mathcal{O}_y=B_\varepsilon(y)$, $y\in \MV$, there exist a countable set $\S^\varepsilon$ and a measurable map $\kappa^\varepsilon$ such that the first part of Lemma \ref{lem:cover} holds. Let $\{ \stopinnewt_{k}^\varepsilon,\ k\in \N \cup \{0\}\}$ and
$\{ \stopoutnew_{k}^\varepsilon,\ k\in \N\}$ be the two nested
sequences of  stopping times defined by $\stopinnewt_{0}^\varepsilon =
0$ and for each $k\in \N$, 
recursively define the following two nested sequences of stopping times:
\begin{eqnarray}
\label{tsigma-ne}
\stopoutnew_{k}^\ve & \doteq & \inf \left\{ t \geq
  \stopinnew_{k-1}^\ve:  \zee (t)
\notin
B_{\ell_{\ptt_{k-1}^\ve}} ( \ptt_{k-1}^\ve) \right\},   \\
\label{ttau-ne}
\stopinnew_{k}^\ve & \doteq & \inf \left\{t\geq \stopoutnew_{k}^\ve:
  \zee (t)\in
\partial G\right\},
\nonumber\end{eqnarray} 
where $\ptt_{k-1}^\ve
\doteq  \kappa^\ve( \zee (\stopinnew_{k-1}^\ve))$ and $\ell_{\ptt_{k-1}^\ve}=  \kappa_{\ptt_{k-1}^\ve}(r_{\ptt_{k-1}^\ve,\varepsilon})$. 
Note that when $\zee (\stopoutnew_k^\ve) \in \partial B_{\ell_{\ptt_{k-1}^\ve}} ( \ptt_{k-1}^\ve) \cap \MU_{\delta}^\MJ$, one could have $\stopinnew_{k}^\ve = \stopoutnew_k^\ve$.   However, because Lemma \ref{lem:cover} implies $\zee
(\stopoutnew_k^\ve)$ lies strictly in the interior of $B_{\ell_{\ptt_{k}^\ve}} ( \ptt_{k}^\ve)$, we always have $\stopoutnew_{k+1}^\ve >\stopoutnew_{k}^\ve$.  

First, observe that the relations $\int_0^{\stopr_r} \ind_{G}
(\zee (u)) d|\fvproc|(u) = 0$ and $\fvproc  (t
\wedge \stopr_r) = \y (t
\wedge \stopr_r)$ established in \eqref{bdjump} and \eqref{Urep2b},
respectively, along with the fact that $\zee (t) \in G$ for $t \in
\cup_{k \in \N}[\stopoutnew_{k}^\ve, \stopinnew_k^\ve)$  imply that
for every $t \geq 0$, 
\be
\label{eta-zero}
\y (t \wedge \stopinnew_k^\ve \wedge \stopr_r) - \y( t \wedge
\stopoutnew_{k}^\ve \wedge \stopr_r) = 0, \quad k \in \N. 
\ee
Next,  for each $k \in \N$, if $\stopinnewt^\ve_{k-1}<\infty$, let\[
\Theta^\ve_{k-1} \doteq \conv \left[\bigcup_{y\in
    B_{\ell_{\ptt_{k-1}^\ve}} ( \ptt_{k-1}^\ve)}d(y)\right], 
\] with
$\ptt_{k-1}^\ve$ defined as in \eqref{tsigma-ne},  and otherwise, let 
$\Theta_{k-1}^\ve \doteq \{0\}$. 
We now show that for every $\ve > 0$, 
\be \label{dis:v0u1} \Q_{\barz}\left(\y (t\wedge
  \stopoutnew_k^\varepsilon\wedge \stopr_r)  - \y( t \wedge
  \stopinnew_{k-1}^\varepsilon \wedge \stopr_r)  \in \Theta_{k-1}^\ve
\right) = 1,  \quad k \in \N. 
\ee
First,  let $\conset_{k-1}^\varepsilon$ be
the set defined in (\ref{Xi}) with 
$\Theta_{k-1}^\ve$ in place of $\Theta$.
For each $v \in \conset_{k-1}^\varepsilon$,  define $g^{v} (x) \doteq \langle v, x
\rangle$, $x \in \R^J$.   Then $g^v \in {\mathcal C}^2(\R^J)$ and 
$\langle \nabla g^v(x), d \rangle = \langle v, d \rangle \geq 0$  for $d \in
d(x)$, $x \in  B_{\ell_{\ptt_{k-1}^\ve}} ( \ptt_{k-1}^\ve)$. 
Thus, Lemma \ref{lem:submartperv}, Lemma \ref{lem:htest2} and Lemma \ref{lem-gtest} 
imply that  $S^{g^v}(\cdot \wedge \stopoutnew_k^\ve\wedge \stopr_r)-S^{g^v}(\cdot \wedge \stopinnew_{k-1}^\varepsilon\wedge \stopr_r)$ is a continuous $\Q_{\barz}$-submartingale and $\Q_{\barz}$-almost surely, 
$R^{g^v}(\cdot \wedge \stopoutnew_k^\ve\wedge \stopr_r)-R^{g^v}(\cdot \wedge \stopinnew_{k-1}^\varepsilon\wedge \stopr_r) = \langle v, \fvproc (\cdot \wedge \stopoutnew_k^\ve\wedge \stopr_r)- 
\fvproc (\cdot \wedge \stopinnew_{k-1}^\varepsilon\wedge \stopr_r) \ran$
is a continuous increasing process. 
In turn, since $\fvproc (\cdot \wedge \stopr_r)= \y(\cdot \wedge \stopr_r)$ on $[0,\stopr_r]$ by \eqref{Urep2b}, this implies 
$\lan v, \y(\cdot\wedge 
\stopoutnew_k^\varepsilon\wedge \stopr_r) -  \y(\cdot\wedge 
\stopinnew_{k-1}^\varepsilon\wedge \stopr_r)\ran$ is a continuous increasing
process. In particular,  for every $t \geq 0$, $\lan v, \y(t\wedge 
\stopoutnew_k^\varepsilon\wedge \stopr_r) -  \y(t\wedge 
\stopinnew_{k-1}^\varepsilon\wedge \stopr_r)\ran\geq \lan v, \y(0)-\y(0)\ran=0$ 
for every $v\in \conset_{k-1}^\varepsilon$. 
Thus,   (\ref{dis:v0u1}) follows from Lemma \ref{lem:ind}. 

Note that   \eqref{eta-zero},  (\ref{dis:v0u1}) and the fact that
$Y(0) = 0$ together imply that for every
$\ve > 0$,  
\be \label{dis:v0u4} \Q_{\barz}\left( \y (t\wedge \stopr_r)  \in
  \conv\left[\bigcup_{k\in \N: \stopinnew_{k-1}^\varepsilon < t \wedge \stopr_r}\Theta_{k-1}^\ve   \right]\right) = 1. 
\ee
Now, for each $k \in \N$,  $\ve > 0$ and $\omega \in \ccspace$, 
define 
\[ F_{k-1}^\ve (\omega) \doteq  \bigcup_{u \in [\stopinnewt_{k-1}^\varepsilon(\omega), \stopinnewt_k^\varepsilon(\omega))}\left\{y\in \bar G:\ |y-\zee (\omega,u)|\leq
4\varepsilon\right\},
\] 
and, recalling that $r_{y,\ve}
< \ve$, note that 
\[\Theta_{k-1}^\ve(\omega)\subseteq 
\conv \left[ \bigcup_{y\in F_{k-1}^\ve(\omega)} d(y) \right]. \]
Since the graph $d(\cdot)$ is closed, we have $\Q_{\barz}$-almost surely, 
\[ \bigcap_{\ve > 0}\bigcup_{k\in
  \N: \stopinnew_{k-1}^\varepsilon < t \wedge \stopr_r}\conv\left[\bigcup_{y\in F_{k-1}^\ve} d(y)\right] = \conv\left[\bigcup_{u \in
  [0,t \wedge \stopr_r]} d(\zee (u))\right].\]   
  Thus,   sending $\varepsilon \downarrow 0$ on both sides of
  (\ref{dis:v0u4}), we obtain  \eqref{eq-locref}. 
\end{proof}

\appendix
\numberwithin{equation}{section}

\beginsec

\section{Proof of a Measurability Property for the ESP}
\label{ap-esp}

We now establish 
Lemma \ref{lem-spesp}. 
Let the processes  $X, Y, Z$,  and the stopping times  $\theta_1,
\theta_2$ be as in Lemma \ref{lem-spesp}, let 
the stopped shifted processes $\tilde{Y}$ and $\tilde{Z}$  be defined as in
\eqref{def-tildey} and \eqref{def-tildez}, respectively, 
and also define the corresponding 
process $\tilde{X}(u) = Z (\theta_1) + X
((u+\theta_1) \wedge \theta_2) - X (\theta_1)$, $u \in [0,\infty)$.  
Also, given any $\R^J$-valued process $H$, 
recall that $|H|(u)$ represents the 
total variation of $H$ on $[0,u]$. 
 It follows from Lemma 2.3 of   
\cite{Ram06}   that $\P_z$-almost surely on the set $A = \{\theta_1 < \theta_2\}$, 
$(\tilde{Z}, \tilde{Y})$ satisfies the ESP
for $\tilde{X}$.  On the other hand, Theorem 2.9 of \cite{Ram06}  shows that for $\omega \in A$ such that 
$\tilde Z(\omega,s) \not \in \MV$ for all $s \in  [0,
\theta_2 (\omega)-\theta_1 (\omega)]$ (which should be interpreted as $s \in [0, \infty)$ when
$\theta_2(\omega) = \infty$), 
the total variation $|\tilde Y|$ of $\tilde{Y}$ is
finite on every bounded interval $[0,t]$, $t < \infty$. 
It then follows from property 2 of Theorem 1.3 of \cite{Ram06} 
(which is restated in Remark \ref{rem-esp} of this paper) that for each $\omega$, one can find 
a Borel measurable function $\gamma(\omega, \cdot)$ on $[0,\infty)$,
with the desired properties stated in \eqref{eq-espmeas}.  
However, to prove the assertion in Lemma \ref{lem-spesp}, we need to
show the existence of a version of $\gamma$  that is jointly
measurable  in $\Omega \times [0,\infty)$. 

To show this, for each $N\in  \mathbb{N}$, we define the stopping time 
\[\tau^N\doteq
\inf\{t\geq 0:\ |\tilde Y|(t) \geq N\},\] 
 let $\theta^N \doteq \tau^N \wedge \theta_2$, 
and let  $\tilde{Y}^{\theta_2}(\cdot)$ and $\tilde{Y}^{\theta^N} (\cdot)$,
respectively, be the stopped processes 
$\tilde{Y}(\cdot \wedge \theta_2)$ and 
$\tilde{Y}(\cdot \wedge \theta^N) = \tilde{Y}(\cdot \wedge \theta_2 \wedge \tau^N)$.

Consider the measure $\bar \mu^N$  on $(\Omega \times \R_+, {\cal F} \times \MB(\R_+))$
defined  by 
 \begin{eqnarray}
\label{def-mutvn}
\bar \mu^N(A\times (s,t]) & \doteq & \E^{\P_z}\left[\ind_A (|\tilde
   Y|(t\wedge \theta^N)-|\tilde Y|(s\wedge \theta^N)) \right] \\
\nonumber
& =  & \E^{\P_z}\left[ \ind_A \int_{(s,t]} d
   |\tilde{Y}^{\theta^N}| (u) \right] 
\end{eqnarray}
 for $A\in {\cal F}$ and  $0\leq s<t< \infty$. 
Since $|\tilde{Y}|$ is almost surely non-decreasing, the definition of
$\tau^N$ implies $\bar{\mu}^N(\Omega \times \R_+) \leq N$, and thus, each  $\bar{\mu}^N$ is  a
finite measure.  
 In an analogous fashion, for each $i=1,\ldots,J$,  define $\mu^N_i$
 to be 
the finite signed  measure on $(\Omega \times \R_+,{\cal F} \times
\MB(\R_+))$ that satisfies, 
\be
\label{def-muitvn}
\mu^N_i(A\times
 (s,t] \doteq \E^{\P_z}\left[\ind_A (\tilde Y_i(t\wedge
   \theta^N)-\tilde Y_i(s\wedge \theta^N)) \right], \quad  A\in {\cal
   F},  0\leq s<t< \infty.  
\ee
From \eqref{def-mutvn} and \eqref{def-muitvn}, it is  clear that
$\mu^N_i\ll \bar{\mu}^N$ 
 for $i = 1, \ldots, J$. Let $\mu^N$ denote the $J$-dimensional vector of
finite signed measures whose $i$th entry is $\mu^N_i$.  
By the Radon-Nikod\`{y}m Theorem, there exists a measurable function
$\gamma^N : (\Omega \times \R_+, {\mathcal F} \times \MB(\R_+))
\mapsto (\R^J, {\mathcal B}(\R^J))$ such that 
\be
\label{mun-rn}
\mu^N(A\times (s,t]) =  \int_{A\times (s,t]}\gamma^N(\omega, u)
d \bar{\mu}^N(\omega, u), \quad A\in {\cal F}, 0\leq s<t<\infty. 
\ee
 Moreover, it is also clear that $\gamma^N(\omega, u)=\gamma^N(\omega,
 u\wedge \tau^N(\omega))$ for  each $\omega\in \Omega,\ 0\leq u<\infty$.

Now, from \eqref{def-muitvn}, \eqref{mun-rn} and \eqref{def-mutvn}, it follows that for each random
variable $\rv$ and measurable function $h$ defined on $\R_+$, 
\begin{eqnarray*}
 \E^{\P_z}\left[\rv \int_{[0,\infty)} h(u) d \tilde
   Y^{\theta^N}(u) \right] &=& \int_{\Omega \times \R_+} \rv(\omega)h(u) d\mu^N(\omega, u)\\ &=& \int_{\Omega \times \R_+}\rv(\omega)h(u) \gamma^N(\omega, u) d\bar{\mu}^N(\omega, u) \\ &=&  \E^{\P_z}\left[\rv (\cdot)\int_{[0,\infty)} h(u) \gamma^N(\cdot, u) d |\tilde Y^{\theta^N}|(u) \right].
\end{eqnarray*}
Hence,  for each $0\leq s<t<\infty$,  since the above display holds for each $\xi$, by choosing $h(u)=\ind_{(s,t]}(u)$, we see that $\P_z$-almost surely, 
\[\tilde Y(t\wedge \theta^N)-\tilde Y(s\wedge \theta^N) = \int_{(s,t]}\gamma^N(\cdot, u) d |\tilde Y^{\theta^N}|(u),\]
and the continuity of $\tilde Y$  implies  that  $\P_z$-almost surely, 
\be
\label{YRN1}
\tilde Y(t\wedge \theta^N)-\tilde Y(s\wedge \theta^N) =
\int_{[s,t]}\gamma^N(\cdot, u) d |\tilde
  Y^{\theta^N}|(u),\quad  0\leq s<t< \infty. 
\ee
In turn, this 
 shows that for $\P_z$-almost  every $\omega$,  $\gamma^N(\omega, \cdot)$ is a
version of the Radon-Nikod\`{y}m derivative of
$d\tilde{Y}^{\theta^N}(\omega, \cdot)$
with respect to $d|\tilde{Y}^{\theta^N}|(\omega, \cdot)$.

We now show that the sequence $\gamma^N, N \in \N,$ is consistent in the sense 
that for $N \in \N$, 
\be
\label{gamma-consist}
\gamma^{N+1}(\cdot, u\wedge \tau^N) = \gamma^{N}(\cdot, u) \quad
\mbox{ for } d|\tilde Y^{\theta_2}|-a.e.\   u \in [0,\infty).
\ee
 Indeed, first note that for each $N\in \N$ and $0\leq s<t<
\infty$,  \eqref{YRN1}, with $N$ replaced by $N+1$, $t$ replaced by
$t \wedge \tau^N$ and $s$ replaced by
$s \wedge \tau^N$,  yields 
 \[\tilde Y(t\wedge
\tau^N\wedge \theta^{N+1})-\tilde Y(s\wedge \tau^N\wedge \theta^{N+1})
= \int_{[s\wedge \tau^N, t\wedge \tau^N]}\gamma^{N+1}(\cdot, u) d
|\tilde Y^{\theta^{N+1}}|(u),\] 
which can be  equivalently rewritten as 
\begin{eqnarray*}
\tilde Y(t\wedge \theta^{N})-\tilde Y(s\wedge \theta^{N}) & =
& \int_{[s,t]}\gamma^{N+1}(\cdot, u\wedge \tau^N) d |\tilde
  Y^{\theta^N}|(u).
\end{eqnarray*}
A  comparison with  \eqref{YRN1} shows that 
\[\gamma^{N+1}(\cdot, u\wedge \tau^N) = \gamma^{N}(\cdot, u) \quad
\mbox{ for } d|\tilde Y^{\theta^{N}}|-a.e.\  u\in
[0,\tau^N],  \] 
which is equivalent to \eqref{gamma-consist}.  
Next,  define 
\[\gamma (\omega, u) = \limsup_{N\rightarrow \infty} \gamma^N (\omega,
u), \quad  (\omega, u) \in \Omega \times [0,\infty). \] 
It follows that $\gamma$ is a measurable function from $(\Omega \times
\R_+, {\mathcal F} \times \MB(\R_+))$ to $(\R^J, \MB (\R^J))$, and
$\P_z$-almost surely,  $\gamma(\cdot, u\wedge \tau^N)
= \gamma^N(\cdot, u)$ for $d|\tilde Y^{\theta_2}|-$ almost every $u \in [0,\infty)$  and 
for each $0\leq s<t< \infty$, 
\begin{eqnarray*}\tilde Y(t\wedge \theta_2\wedge \tau^N)-\tilde
  Y(s\wedge \theta_2\wedge \tau^N) &=& \int_{[s,t]}\gamma(\cdot, u\wedge
  \tau^N) d |\tilde Y^{\theta_2\wedge \tau^N}|(u)\\ &=&
  \int_{[s\wedge \tau^N,t\wedge \tau^N]}\gamma(\cdot, u) d
  |\tilde Y^{\theta_2}|(u).
\end{eqnarray*}
Sending  $N\rightarrow \infty$ and using the continuity of $\tilde
Y$ and the fact that $\tau^N \ra \infty$ because $\tilde{Y}$ has
finite variation on every bounded interval, we have  $\P_z$-almost surely, for each $0\leq s<t< \infty$, 
\begin{eqnarray*}\tilde Y(t)-\tilde Y(s)=\tilde Y(t\wedge \theta_2)-\tilde Y(s\wedge \theta_2)= 
  \int_{[s,t]} \gamma(\cdot, u) d |\tilde Y^{\theta_2}|(u)=\int_{[s,t]} \gamma(\cdot, u) d |\tilde Y|(u).
\end{eqnarray*}
Finally, since the pair $(Z,Y)$ solves the ESP for $X$, it follows from Remark \ref{rem-esp} that 
 $\P_z$-almost surely, $\gamma (u) \in d(\tilde Z(u))$ for $d|\tilde
 Y|$-almost every $u \in [0,\infty)$. This completes the proof.

\beginsec

\section{Regular Conditional Probabilities}
\label{ap-regcondprob}

\begin{proof}[Proof of Lemma \ref{lem:submartperv}] 
Fix $y\in \bar{G}$. 
Let $\MH_0$ be the countable dense subset of $\MH$ mentioned in Remark
\ref{rem-MH}. 
Then there exists a set $\MN\in \MM_{\tau}$ with $\Q_y(\MN)=0$ such that for every $\omega'\notin \MN$ with $\tau(\omega')<\infty$, each $f\in \MH_0$ and each $N\in \N$, \[f(\omega((\tau(\omega')+t)\wedge \varsigma^N))-\int_{\tau(\omega')\wedge \varsigma^N}^{(\tau(\omega')+t)\wedge \varsigma^N}\ML
f(\omega(u))\,du\] is a $\Q_{\omega'}$-submartingale, where \[\varsigma^N=\inf\{t\geq 0:\ \omega(t)\notin B_N(0)\},\qquad
 \omega\in \ccspace. \] It follows from (\ref{fgk}) that for every $\omega'\notin \MN$ with $\tau(\omega')<\infty$, each
$f\in \MH$ and $N\in \N$, \[f(\omega((\tau(\omega')+t)\wedge \varsigma^N))-\int_{\tau(\omega')\wedge \varsigma^N}^{(\tau(\omega')+t)\wedge \varsigma^N}\ML
f(\omega(u))\,du\] is a $\Q_{\omega'}$-submartingale. By passing to the limit as
$N\rightarrow \infty$, we conclude that for every $\omega'\notin \MN$ with $\tau(\omega')<\infty$ and each
$f\in \MH$, \[f(\omega(\tau(\omega')+t))-\int_{\tau(\omega')}^{\tau(\omega')+t}\ML
f(\omega(u))\,du\] is a $\Q_{\omega'}$-submartingale. This shows that property 3 of Definition \ref{def-smg} holds for $\omega(\tau(\omega')+\cdot)$ under $\Q_{\omega'}$. At last, properties 1, 2 and 4 of Definition \ref{def-smg} hold for $\omega(\tau(\omega')+\cdot)$ under $\Q_{\omega'}$ by enlarging, if needed, $\MN\in \MM_\tau$ with $\Q_x(\MN)=0$ using properties of regular conditional probability distributions. This completes the proof of the lemma. \end{proof}

\beginsec

\section{Proof of the Covering lemma}
\label{subs-cover}

In this section, we prove Lemma  \ref{lem:cover}. 
Fix a collection of open sets $\{{\mathcal O}_y,\ y \in \partial G\}$ such that ${\mathcal O}_y \cap \MV = \emptyset$ if $y\in \MU$. To prove Lemma \ref{lem:cover}, we introduce a countable open covering of $\MU$ from $\{{\mathcal O}_y,\ y \in \partial G\}$ that has certain desirable geometric properties. 
 For each $\delta>0$, let  
\be
\label{def-delta}
 \MU_\delta\doteq \left\{y\in \MU:\  \begin{array}{c}
\MI(z)\subseteq\MI(y) \mbox{ for all } z\in B_\delta(y)\cap \partial
G \mbox{ and } \exists\ n \in n(y) \\
 \mbox{ such that } n=\sum_{i\in \MI(y)}  \coeff_i n^i(y),\mbox{ where }
 \sum_{i\in \MI(y)}s_i=1,\\\coeff_i \geq 0, i \in \MI (y) \mbox{ and } \lan n, d
 \ran \geq \delta |d| \mbox{ for all } d \in d(y) \end{array}
\right\},
\nonumber \ee
and for each  $\J \subseteq \MI$, let
\be
\label{def-deltaJ}
\MU^{\J}_\delta\doteq \{y\in \MU_\delta: \MI(y)=\J\}.
\ee
It follows from Lemma 6.2 of \cite{KanRam14} that $\MU^{\J}_\delta$ is a
closed set for each $\delta > 0$, and $\J \subseteq \MI$. It is immediate from (\ref{def-deltaJ}) that any two elements in
$\{\MU^{\J}_\delta, \J \subseteq \MI\}$ are
 disjoint, and
\be
\label{ginc}
\MU = \bigcup_{\delta > 0} \MU_{\delta}, \qquad  \MU_\delta =
\bigcup_{\J \subseteq \MI}\MU^{\J}_\delta.  \nonumber
\ee

Recall that $B_r(x)$
denotes an open ball of radius $r$ around the point $x$. Suppose that $\MU_\delta^{\J} \neq \emptyset$. Since $\MU^{\J}_\delta$ is a closed set, then $\MU^{\J}_\delta \cap
\bar{B}_n(0)$ is compact for each $n \in \N$.    Since $\{ {\mathcal O}_y,\ y \in \MU^{\J}_{\delta}\cap \bar{B}_n(0)\}$
clearly covers the compact set $\MU^{\J}_\delta \cap
\bar{B}_n(0)$,  there exists a finite subset $\S_{n,\delta}^\MJ$ of $\MU^{\J}_{\delta}
\cap \bar{B}_n(0)$
such that $\{ {\mathcal O}_y,\ y \in \S_{n,\delta}^\MJ\}$ covers $\MU^{\J}_{\delta} \cap
\bar{B}_n(0)$.
 The countable set $\S^\MJ_\delta \doteq \cup_{m \in \N}
 \S_{n,\delta}^\MJ$ clearly satisfies 
\be \label{coverd}
\MU^\MJ_\delta \subseteq  \bigcup_{y \in  \S^\MJ_\delta}
{\mathcal O}_y. 
\nonumber \ee 
We can further choose the set $\S^\MJ_\delta$ to be minimal in the
sense that for each strict subset $C$ of $\S^\MJ_\delta$, $\cup_{y \in
  C} {\mathcal O}_y$ does not cover $\MU_\delta^\MJ$. Let $\S^\MJ \doteq
\cup_{n=1}^\infty\S^\MJ_{1/2^n}$. Then $ \cup_{y \in \S^\MJ}
{\mathcal O}_y$ covers $\MU^\MJ \doteq \{y\in \MU:\ \MI(y)=\MJ\}$.  
We enumerate the elements of $\S^\MJ=\{y_i,\ i\in \N\}$. Let $D_k={\mathcal O}_{y_k}\setminus (\cup_{i=0}^{k-1} {\mathcal O}_{y_i})\cap  \MU^\MJ$ for each $k$, $\{D_k\}$
is a partition of $\MU^\MJ$, and for each $y\in\MU^\MJ$ there is a unique index $i(y)$ such that
$y\in D_{i(y)}$. Define $\kappa^\MJ(y) = y_{i(y)}$. Then $\kappa^\MJ$
is a measurable mapping from $\MU^\MJ$ onto $\S^\MJ$. Let
$\S=\cup_{\MJ\subseteq \MI}\S^\MJ$. By patching $\kappa^\MJ,
\MJ\subseteq \MI,$ together, we get a measurable mapping $\kappa$
from $\MU$ onto $\S$. For each $y\in \MU$, $y\in \MU^{\MI(y)}$. Since
$\kappa(y)\in \S^{\MI(y)}$, we have $\MI(\kappa(y))=\MI(y)$.

\section{Construction of Test Functions}
\label{ap-test}

This section is devoted to the proof of Lemma \ref{lem:cutoff}. 
The first property of the lemma follows from the upper-semicontinuity of
  the set function $\MI(y) = \{i \in \MI: y \in \partial G_i\}$
  introduced in Definition \ref{ass:G}. 
Since $y \in \MU$, property 
2 follows from property 1, the definition of $\MU$ in 
\eqref{def-mu} and the continuity of  $n^i(\cdot)$ and
$\gamma^i(\cdot)$, which holds by  Definition \ref{ass:G}. 
We devote the rest of the section to the proof of property 3. 

We now establish two results in preparation for the proof of property 3.  We first paraphrase a result from \cite{Ram06}. 

\begin{lemma}
\label{prop:TF1}
Let $\Cone$  be a closed convex cone with vertex at the
origin and a boundary that is ${\cal C}^\infty$, except possibly at
the vertex. 
  Given any  closed, convex, compact subset  ${\cal  K}$ of the interior of
$\Cone$, constants $0 < \csto < \cstt < \infty$ and $\ve > 0$, there
exist $\theta > 0$ and 
a ${\cal C}^\infty$ function $\ell$ on the set
\[  \Lambda \doteq \{ x \in \R^J: \csto < \dist (x, \Cone)
< \cstt \}
\]
that satisfy the following properties:
\begin{enumerate}
\item
$\sup_{x \in \Lambda} \left( \left|\ell(x) - \dist (x, \Cone) \right|
\vee \left||\nabla \ell(x)| - 1 \right|\right)\leq  \ve$;
\item
there exists $\theta > 0$ such that
\[  \lan \nabla \ell (x), p \ran \leq - \theta, \qquad \mbox{ for } p
\in {\cal K} \mbox{ and } x \in \Lambda.
\]
\end{enumerate}
Moreover, if  $\Cone$ is a half-space, given any subset ${\cal K}$  of
$\Cone$,   the
function $\ell (x) \doteq \dist (x, \Cone)$, $x \in \Lambda$, is  a ${\cal C}^2$ function
on $\Lambda$ that satisfies property (1) and
property (2) with $\theta = 0$ above.
\end{lemma}
\begin{proof}
The function $\ell$ with the  properties
stated above can be constructed as a suitable mollification of the
distance function to the cone $\Cone$.  Indeed, 
Lemma \ref{prop:TF1} can be deduced from the proof of Lemma 6.2 of
\cite{Ram06},  with  $g_C$, $L_{C,\delta_C}$, $K_{C}^{\delta_C/3}$,
$\tilde \eta_C$, $\tilde \lambda_C$ and $\tilde
 \varepsilon_C$ therein replaced by  $\ell(\cdot)$, $\Cone$,  ${\cal
   K}$, $\csto$,  $\cstt$ and $\ve$, respectively.
 \end{proof}

Next, we introduce some geometric objects associated with
the directions of reflection,  similar to those introduced in
 Section 6.1 of \cite{Ram06} in the context of polyhedral domains.
For $y \in \MU$,  let
\be \label{Kx}
K_y \doteq \left\{-\sum_{i\in \MI(y)}a_id^i(y):
  a_i\geq 0, i \in \MI(y), \sum_{i\in \MI(y)}a_i=1 \right\}. \ee
Note that $K_y$ is a convex, compact subset of $\R^J$. 
  Therefore,  
 there exist $\delta_y >0$ and a compact, convex
set $K_{y,\delta_y}$ such that $K_{y,\delta_y}$ has $\C^\infty$
boundary and satisfies 
\be \label{func1}K_y^{\delta_y/2}\subset
(K_{y,\delta_y})^\circ\subset K_{y,\delta_y} \subset
K_y^{\delta_y}, \ee
where $K_y^\ve \doteq \{x\in \R^J:\ \dist(x,K_y)\leq \ve\}$ for every $\ve > 0$.
It is easy to see (cf. Lemma 6.1 of \cite{Ram06})  that
$0 \not \in K_y$ and
\be \label{sep0} \min_{i\in \MI(y)}\lan n^i(y),d\ran <0 \qquad
 \mbox{for every } d\in K_y. \nonumber\ee
Therefore,  $\delta_y > 0$ can be chosen such that $0 \notin K_{y,\delta_y}$ and
\be \label{est:test11}\min_{i\in \MI(y)}\lan
n^i(y),d\ran <0 \qquad \mbox{for every } d\in K_y^{\delta_y}. \ee

\begin{lemma}\label{lem:A1} For $y \in \MU$,
there exist $\bar{R}_y\in (0,1)$ and $\beta_y>0$ such that
\be \label{sep2} \min_{i\in \MI(y)}\lan n^i(y),d\ran <-2\beta_y|d|
\qquad \mbox{for every } d\in \cup_{t\in [0, \bar R_y]}tK_{y,\delta_y}\ee
and
\be \label{sep1}
\left(y+\cup_{t\in [0, \bar R_y]} tK_{y,\delta_y}\right) \cap \bar G = \{y\}. \ee
\end{lemma}
\begin{proof} We first use an argument by contradiction to prove that
\be \label{sep3}\sup_{d\in K_{y,\delta_y}} \min_{i\in \MI(y)}\lan n^i(y),d/|d|\ran  < 0. \ee
Since $K_{y,\delta_y}$ is compact and $d \mapsto \min_{i\in
  \MI(y)}\lan n^i(y),d/|d|\ran$ is continuous, the supremum can be
replaced by a maximum in \eqref{sep3}.  Thus, if (\ref{sep3}) does not
hold, then there exists  $d \in
K_{y,\delta_y}$ such that $\min_{i\in
  \MI(y)}\lan n^i(y),d/|d|\ran$   $\geq 0$. But this contradicts
(\ref{est:test11}).  
Thus, (\ref{sep3}) holds, which, in particular, implies that  there exists $\beta_y>0$
such that  (\ref{sep2})  holds for any $\bar{R}_y>0$.

In addition, for each $i \in
\MI(y)$,
since $\partial G_i$ is $\C^1$ near $y$, it follows that
\[\lim_{\delta \rightarrow 0}\inf_{x\in \bar G:\ |x-y|\leq
  \delta} \min_{i\in \MI(y)}\lan n^i(y), (x-y)/|x-y| \ran \geq 0. \]
Together with (\ref{sep3}) this shows that there exists $\bar{R}_y\in (0,1)$ such that
\be \label{sep4}
\inf_{\stackrel{x\in \bar G:}{|x-y|\leq \bar{R}_y(\sum_{i\in
    \MI(y)}|d^i(y)|+\delta_y)}} \min_{i\in \MI(y)} \left\lan
n^i(y), \frac{x-y}{|x-y|} \right\ran > \sup_{d\in K_{y,\delta_y}} \min_{i\in
  \MI(y)}\left\lan n^i(y),\frac{d}{|d|}\right\ran.\ee
 We use this to prove (\ref{sep1}) by contradiction.  Suppose that
 (\ref{sep1}) does not hold.  
Then there exists $d\in \cup_{t\in [0,\bar{R}_y]}tK_{y,\delta_y}$ such that $d\neq 0$ and $y+d\in \bar
 G$.   We can write $d = t^* d^*$ for some 
 $t^*\leq \bar{R}_y$ and $d^*\in K_{y,\delta_y}$.  Also, by the
definition of $K_{y,\delta_y}$ 
in \eqref{Kx} and \eqref{func1}, 
 $|d|\leq \bar{R}_y(\sum_{i\in \MI(x)}|d^i(y)|+\delta_y)$. 
Hence,  (\ref{sep4}) implies that
\[ \min_{i\in \MI(y)}\left\lan n^i(y), \frac{d^*}{|d^*|} \right\ran= \min_{i\in
  \MI(y)}\left\lan n^i(y), \frac{d}{|d|} \right\ran > \sup_{d\in K_{y,\delta_y}}
\min_{i\in \MI(y)}\left\lan n^i(y), \frac{d}{|d|} \right\ran, \] which contradicts the fact
that $d^* \in K_{y,\delta_y}$. Thus,  (\ref{sep1}) holds for the chosen
$\bar{R}_y \in (0,1)$.
\end{proof}

 For each $i\in \MI(y)$,
since $\partial G_i$ is of $\C^1$ near
$y\in \partial G$, the hyperplane $\{x\in \R^J:\ \lan n^i(y),x-y\ran = 0\}$ is  the tangent plane to $\partial G_i$ at $y$. Let
\be \label{Sx} \hyp_y \doteq \cap_{i\in \MI(y)}\{x\in \R^J:\ \lan n^i(y),x-y\ran
\geq 0\}. \nonumber\ee Then  $\bar G$ can be locally approximated near $y$
by the polyhedral cone $\hyp_y$ in the sense that for each $N>0$,
\be \label{approxG}\{y+(x-y)/r \in \R^J:\ x\in \bar G,\ |x-y|\leq
N r\} \rightarrow \hyp_y \cap B_N(y)\ \mbox{ as } r\rightarrow 0,\ee
where the convergence is with respect to the Hausdorff distance.
 In view of (\ref{sep1}), it follows that there exist $0<r_y<\dist(y,
 \MV \cup
 \cup_{i\notin \MI(y)}(\partial G\cap \partial G_i))$ and
 $\lambda_y\in (0,1)$ small enough (independent of $r_y$) such that for each $r\in (0,r_y)$,
 \be \label{include1}\left\{x\in \R^J:\ \dist(x,y+\cup_{t\leq
     \bar{R}_y}tK_{y,\delta_y})\leq 3\lambda_y r\right\}\cap \partial G
 \subset B_r(y)\cap \partial G\ee and
\be \label{include2}\left\{x\in \R^J:\ \dist(x,y+\cup_{t\leq
    \bar{R}_y}tK_{y,\delta_y})\leq 3\lambda_y r\right\}\cap \bar G \cap \partial B_r(x)=\emptyset.\ee
It follows from Lemma \ref{prop:TF1} with $\Cone = \cup_{t\geq 0}tK_{y,\delta_y}$,   ${\cal K} = K_y^{\delta_y/3}$, $\lambda =
2\lambda_y$,  $\eta = \eta_y \in (0,\lambda_y)$, $\Lambda  =
\Pi_y\doteq \{x\in \R^J:\ \eta_y<\dist\left(x, \Cone\right)\leq 2\lambda_y\}$  and
$\varepsilon_y=\lambda_y/12\wedge \eta_y/2$ that  there exists a
function $\ell_y:\ \Pi_y \rightarrow \R$, that satisfies all the
properties stated in Lemma \ref{prop:TF1}.

Let $L_{x,\delta_x}$ be a truncated (half) cone with vertex at the
origin defined by
\be \label{Lx} L_{y,\delta_y}\doteq \cup_{t\leq \bar{R}_y/2}tK_{y,\delta_y}.\nonumber \ee
Then (\ref{sep1}) implies
\[(y +L_{y,\delta_y})\cap\bar G = \{y\}.\]
Due to the fact that $y \in \MU$, there exists 
 a unit vector  $q_y$ in the set $K_y$
defined in (\ref{Kx}) such that $-q_y$ points into $G$ from $y$. 
For each $r\in (0,1)$, define \be \label{Mxr} M(y,r)\doteq y-\lambda_y
\frac{\bar{R}_y}{2}rq_y+rL_{y,\delta_y}. \ee For each $\ve\geq 0$, let 
\[ M^{\ve}(y,r)\doteq \{x\in \R^J:\ \dist(x,M(y,r))\leq \ve\}.\]
Since $\bar{R}_y < 1$, it is clear that for each $x\in
M^{2\lambda_y r}(y,r)$, \[\dist(x,y+\cup_{t\leq \bar{R}_y}tK_{y,\delta_y})
\leq 2\lambda_y r + \left|\lambda_y \frac{\bar{R}_y}{2}rq_y\right|< 3 \lambda_y r.\]
Thus, \begin{eqnarray*}M^{2\lambda_y r}(y,r)\subseteq \{x\in \R^J:\
  \dist(x,y+\cup_{t\leq \bar{R}_y}tK_{y,\delta_y})< 3\lambda_y
  r\}\end{eqnarray*}
and hence, by (\ref{include1})--(\ref{include2}) we have
\be \label{sub1} M^{2\lambda_y r}(y,r) \cap \bar G \cap \partial
B_r(y)=\emptyset \mbox{ and } M^{2\lambda_y r}(y,r)\cap
\bar G \subset B_r(y) \cap \bar G. \ee
Let \be \label{Oxr} {\mathcal O}(y,r)\doteq \bar G \cap (M^{2\lambda_y r}(y,r)
\setminus M^{\eta_y r}(y,r))\subset B_r(y) \cap \bar G.\ee
For each $x\in {\mathcal O}(y,r)$, it is clear that $x\in \bar G$ and \be
\label{star}
\eta_y<\dist\left(\frac{x-y}{r}+
  \frac{\lambda_y \bar{R}_y}{2} q_y, L_{y,\delta_y}\right) \leq 2\lambda_y.
\ee

\noi
Since $\bar G$ can be locally approximated at $y$ by $\hyp_y$ as 
in (\ref{approxG}),  by choosing $r_y$ and $\lambda_y$ sufficiently small,
we can ensure that for each $r\in (0,r_y)$ and $x\in {\mathcal O}(y,r)$, the projection of 
$(x-y)/r+\lambda_y (R_y/2)q_y$ to $L_{y,\delta_y}$
coincides with the projection of $(x-y)/r+\lambda_y (R_y/2)q_y$ to
$\cup_{t\geq 0}tK_{y,\delta_y}$ since $L_{y,\delta_y}$ is the portion
of $\cup_{t\geq 0}tK_{y,\delta_y}$ truncated near its vertex.  Hence,
for each $x\in {\mathcal O}(y,r)$, we have 
\[\dist\left(\frac{x-y}{r}+
  \frac{\lambda_y \bar{R}_y}{2} q_y, L_{y,\delta_y}\right) =\dist\left(\frac{x-y}{r}+
  \frac{\lambda_y \bar{R}_y}{2} q_y, \cup_{t \geq 0} t
K_{y,\delta_y}\right).  \]  Together with \eqref{star} this shows that
for each  $x\in {\mathcal O}(y,r)$, $\frac{x-y}{r}+\lambda_y
  \frac{\bar{R}_y}{2} q_y \in \Pi_y$.
Let
$k_{y,r}$ be the function on ${\mathcal O}(y,r)$ given by
\be  \label{kxr} k_{y,r}(x)\doteq \ell_y\left(\frac{x-y}{r}+
  \frac{\lambda_y \bar{R}_y}{2} q_y\right), \quad x \in {\mathcal O}(y,r). \nonumber \ee  Then the properties of
  $\ell_y$ stated in Lemma \ref{prop:TF1} and (\ref{Oxr}) imply that
  $k_{y,r} \in \C^\infty ({\mathcal O}(y,r))$ and $k_{y,r}$
  satisfies \be
\label{star2}
\sup_{x\in {\mathcal O}(y,r)} \left(\left|k_{y,r}(x)-\dist\left(\frac{ (x-y)}{r}+
  \frac{\lambda_y\bar{R}_y}{2}q_y,L_{y,\delta_y}\right)\right|\vee (r|\nabla k_{y,r}(x)|-1)\right)\leq
  \frac{\lambda_y}{12}, \ee 
and there exists $\theta_y >0$ (independent of $r_y$) such that
$ \lan r \nabla  k_{y,r} \left( x\right), p \ran \leq -\theta_y$ for each $p\in K_y^{\delta_y/3}$ and $x\in O(y,r)$.
From the second property of $k_{y,r}$, it follows that
 \[\lan r\nabla k_{y,r}(x),d^i(y)\ran \geq \theta_y\mbox{ for $i\in \MI(y)$ and $x\in O(x,r)$}.\]
Since $d^i(\cdot)$ is  continuous for each $i\in \MI$, by possibly making $r_y$ yet smaller and using the first property of $k_{y,r}$, we have
for each $r\in (0,r_y)$,
\be \label{func2}\lan r\nabla k_{y,r} (y),d^i(y)\ran \geq \theta_y/2\mbox{ for $i\in \MI(x)$ and $x\in O(y,r)$}.\ee

Now, choose $\fntwo_y\in \C^\infty(\R)$ to be a decreasing function such
that \be
\label{def-hx}
\fntwo_y(s)=\left\{\begin{array}{ll} 1 & \mbox{ if } s\in
    (-\infty,5\lambda_y/4],\\ \mbox{ strictly decreasing } & \mbox{ if
    } s\in (5\lambda_y/4, 23\lambda_x/12], \\ 0 & \mbox{ if } s\in
    (23\lambda_y/12,\infty), \end{array}\right.
\ee
and define $f^{y,r}:\ \R^J\rightarrow \R_+$ to be
\be
\label{gxr2}
f^{y,r}\doteq \left\{\begin{array}{ll} \fntwo_y(k_{y,r}(x)) & \mbox{
      if } x\in {\mathcal O}(y,r), \\
1 & \mbox{ if } x\in \bar G \cap M^{\eta_y r}(y,r), \\
0 & \mbox{ otherwise.} \end{array} \right.\ee
When combined with the definitions of $M(y,r)$ and ${\mathcal O}(y,r)$ given in
(\ref{Mxr}) and  (\ref{Oxr}), respectively, and properties
(\ref{sub1}) and \eqref{star2}, we infer that
\be
\label{support}
\begin{array}{l} \ds supp[f^{y,r}]\cap \bar G \\ \qquad \ds
  \subset \left\{x\in \bar G:\ k_{y,r}(x)\leq
    23 \lambda_y/12\right\} \\
\qquad \ds \subset   \{x\in \bar G:\
\dist\left(\frac{x-y}{r}+\frac{\lambda_y 
    \bar{R}_y}{2}q_y,L_{y,\delta_y}\right)\leq
\frac{23\lambda_y}{12}+\frac{\lambda_y}{12}\} \\
\qquad \ds =  \left\{x\in \bar G:\ \dist\left(x-y+\frac{\lambda_y
      \bar{R}_y}{2} rq_x, rL_{y,\delta_y}\right)\leq 2\lambda_y
  r\right\} \\ \qquad \ds = M^{2\lambda_y r}(y,r)\cap \bar G \\ \qquad \ds \subset B_r(y)\cap \bar G, \end{array} \ee
which establishes property 3(b) of Lemma \ref{lem:cutoff}.  In addition,
\be \begin{array}{l} \ds \{x\in \bar G:\ k_{y,r}(x)\geq
  5\lambda_y/4\} \\ \quad \ds \subset   \left\{x\in \bar G:\
    \dist\left(\frac{x-y}{r}+\frac{\lambda_y
        \bar{R}_y}{2}q_y,L_{y,\delta_y}\right)\geq
    \frac{5\lambda_y}{4}-\frac{\lambda_y}{12}\right\} \\ \quad \ds
  \subset  \left\{x\in \bar G:\ \dist((x-y)+\frac{\lambda_y 
     \bar{R}_y}{2}rq_y, rL_{y,\delta_y})\geq \lambda_y r\right\} \\
 \quad \ds \subset (M^{\eta_y r}(y,r))^c\cap \bar G, \end{array} \ee
where the last inclusion uses \eqref{star}. 
Thus, the set on which $f^{y,r}$ is neither $0$ nor $1$ is a strict subset of
${\mathcal O}(y,r)$. Combining this with (\ref{support}) and the properties  $\fntwo_y\in \C^\infty(\R)$
and $k_{y,r}\in \C^\infty({\mathcal O}(y,r))$,  it follows that $f^{y,r}\in
\C^\infty(\bar G)$.  

By the definition of $\fntwo_y$ in \eqref{def-hx},  $f^{y,r}$ clearly satisfies property 3(c) of Lemma
\ref{lem:cutoff}. Moreover, 
since $y$ is an interior point of $M(y,r)$, there exists
$\kappa_y(r) \in (0,r)$ such that $B_{\kappa_y(r)}(y)\subset M(y,r)$.
For each $x \in B_{\kappa_y(r)}(y)\cap \bar G$, the definition of $f^{y,r}$ in
(\ref{gxr2}) implies that $f^{y,r} (x) = 1$.
Thus, $f^{y,r}$ satisfies the property 3(d) of Lemma \ref{lem:cutoff}.
Finally, for each $x\in {\mathcal O}(y,r)$, a simple
calculation shows that $\nabla f^{y,r}(x) =
\fntwo_y'(k_{y,r}(x))\nabla k_{y,r} (x).$
Together with (\ref{func2}) and the property of $\fntwo_y$ in (\ref{def-hx}),
this implies that \[\lan \nabla f^{y,r}(y), d^i(x) \ran \leq
0\mbox{ for $i\in \MI(x)$ and $x\in \MO(y,r)$},\]
which proves that $-f^{y,r} \in {\mathcal H}$. Since $f^{y,r}$ has
compact support by property 
3(b), this implies property 3(a) of  
Lemma \ref{lem:cutoff}. This completes the
proof of Lemma \ref{lem:cutoff}.

\beginsec

\section{Proof of some Integral Representations}
\label{ap-chig}

This section is devoted to the proof of Proposition \ref{prop-chig}.  
First, in Section \ref{subs-ranlin}, we introduce a random positive linear
functional, which we  use  in Section \ref{apsub-integrep} to 
establish a preliminary integral representation for $\pushproc^f$ 
(see Lemma \ref{lem-chirep}).    The proof of Proposition
\ref{prop-chig} is then given in Section \ref{subs-proofchig}.

\subsection{A Random Positive Linear Functional} 
\label{subs-ranlin}

 In what follows, let
\be
\label{def-mk1}
\MK\doteq \left\{(x,\vect)\in \R^{2J}:\ x\in \partial G\setminus \MV,\ \vect\in
d(x),\ |\vect |=1\right\}. \nonumber
\ee
For each $f\in \MH$, let $h_f:\MK \mapsto \R$ be the function given by
\[ h_f(x,\vect)\doteq \langle \vect, \nabla f(x) \rangle,  \qquad (x,\vect)\in
\MK. \]
Clearly, $h_f\in \MC_c^1(\MK)$ for each $f\in \MH$. Note that
$\MC_c(\R_+\times \MK)$, equipped with the uniform norm, is a separable linear space.  Let $\MT_0$ be the linear subspace of $\MC_c(\R_+\times \MK)$ given by
 \[\MT_0\doteq \left\{g\in \MC_c(\R_+\times \MK) : \begin{array}{c}
     g(u,x,\vect)=\sum_{i=1}^n\ell_i(u)h_{f_i}(x,\vect), \\
   n \in \N,   f_i\in \MH, \ell_i\in \MC_c(\R_+),\ i = 1, \ldots,
   n \end{array}\right\}. \]

Now, let $\pushproc^f, f \in \MH,$ be the family of processes defined
in \eqref{def-chig}, 
and  recall from Lemma \ref{lem-gtest} that there exists a set $\Omega_0 \in
{\mathcal M}$ with $\Q_{\barz} (\Omega_0) = 1$ 
such that  for all $\omega \in \Omega_0$, $f\in \MH$, $t \mapsto \pushproc^f (\omega, t\wedge \stopr_r(\omega))$ is 
increasing  and the map $f \mapsto \pushproc^f(\omega, \cdot\wedge \stopr_r(\omega))$ is linear. 
 For each $g\in \MT_0$ that has a representation of the form
$g(u,x,\vect)=\sum_{i=1}^n\ell_i(u)h_{f_i}(x,\vect)$, with $\ell_i \in
{\mathcal C}_c (\R_+)$ and $f_i \in \MH$, $i = 1, \ldots, n$, define 
 \be
\label{def-Lambda}
\Lambda\left(\omega, g\right)\doteq \left\{ 
\begin{array}{ll}
\sum_{i=1}^n\int_0^{\infty} \ell_i(u) d\pushproc^{f_i}(\omega, u\wedge \stopr_r(\omega)), 
\qquad & \mbox{ if } \omega \in \Omega_0, \\
  0 & \mbox{ otherwise. }
\end{array}
\right. 
\ee
We will sometimes suppress the dependence of $\Lambda$ and
$\pushproc^f$ on $\omega$ and
simply write $\Lambda (g)$ and $\pushproc^{f} (t\wedge \stopr_r)$, respectively. 

We will show that $\Lambda$ is a  random positive linear functional on
$\MT_0$,  in a sense made precise below. 

\begin{definition}
\label{defn-ranlin}
Let $X$ be a topological linear space. A map $\Psi:\ \Omega \times X
\rightarrow \R$ is a random linear functional on $X$ if it satisfies
the following two properties: \\
i) $\Psi(\cdot, x)$ is a random variable for each $x\in X$; \\
 ii) $\Psi(\omega,\cdot)$ is a linear functional on $X$ for each
 $\omega \in \Omega$. 
\end{definition}

The positivity of $\Lambda$ will be shown with respect to a suitable positive
cone.  Define 
  \[\MP\doteq \{g\in
\MC_c(\R_+\times \MK):\ 0\leq g(u, x,\vect)\leq h_f(x,\vect), (x,\vect)\in \MK \mbox{ for some
} f\in \MH\}.\]
Consider the partial order $\preceq$ on $\MC_c(\R_+\times \MK)$ defined by
$h \preceq g$ if $g-h\in \MP$.

\begin{lemma} \label{lem:p}
The set  $\MP$ is a
positive cone in $\MC_c(\R_+\times \MK)$. Moreover,  for each $g
\in
\MC_c(\R_+\times \MK)$,  there exists $\hat{g}\in \MT_0$ such that $g \preceq \hat{g}$. 
Furthermore, if $g\in \MC_c(\R_+\times \MK)$ is non-negative,
then $g \in \MP$ and $0 \preceq g$.
\end{lemma}
\begin{proof}
Note that if $g, \tilde{g} \in \MP$, there exist $f,\tilde f\in \MH$
such that for $(x,\vect )\in \MK$, $0\leq g(u, x,\vect)\leq h_f(x,\vect )$
and $0\leq \tilde g(u, x, \vect)\leq h_{\tilde f}(x,\vect )$. Hence,
by the linearity of the mapping $f \mapsto h_f$, 
 $0\leq g(u, x,\vect)+ \tilde g(u,x,\vect)\leq h_f(x,\vect )+h_{\tilde
  f}(x,\vect ) = h_{f+\tilde{f}} (x, \vect)$, and for  $a>0$,
$0\leq a g(u, x,u)\leq a h_f(x,\vect ) = h_{af} (x, \vect)$. 
Thus,  $g+\tilde{g}\in \MP$ and $ag\in \MP$, showing that   $\MP$ is a
positive cone in $\MC_c(\R_+\times \MK)$. 

 We now turn to the proof of the second assertion of the lemma. 
Fix  $g\in \MC_c(\R_+\times \MK)$. Then there exists a
compact set $K\subset \MK$, an interval $[t_1,t_2)\subset \R_+$ and a
constant $0 < C < \infty$ such that $|g(u,x, \vect)|\leq
C\ind_{[t_1, t_2)}(u)\ind_{K}(x,\vect )$ for each $(u,x,\vect)\in
\R_+\times \MK$.  Since $\MK \cap \MV \times \R^J = \emptyset$ and $\MV$ is closed,
there exist $r, s  > 0$ such that
\be
\label{rs-cont}
\{ x\in \R^J: (x,\vect ) \in K \} \subseteq \MU_{r,s}\doteq \{ x\in \partial G: |x| \leq r, d(x, \MV)
\geq s\}.
\ee
Now, choose $f_{r,s}$ from Lemma \ref{lem:frs}.  Then    $f \doteq C f_{r,s}$ 
satisfies   $|g(u,x,\vect)|\leq   C\ind_{K}(x,\vect ) \leq h_f(x,\vect )=
\langle \vect, \nabla f(x) \rangle$ for each $(u,x,\vect)\in
\R_+\times \MK$. Let $\ell\in  \MC_c(\R_+)$ be a function such that
$\ind_{[t_1,t_2)}(u) \leq \ell(u)\leq 1$ for each $u\in \R_+$, and  choose $\hat{g}(u,x,\vect)=\ell(u)h_f(x,\vect )$. Then $\hat{g} \in \MT_0$ and  $0\leq  \hat{g}-g \leq 2\hat{g}\leq 2 h_f=h_{2f}$ on $\MK$.
Since $2f \in \MH$, this shows that  $\hat{g}-g\in \MP$ and hence, that  $g
\preceq \hat{g}$.  
Lastly, if $g\in \MC_c(\R_+\times \MK)$ and $g\geq 0$, the last
argument shows that  $0\leq g(u,x,\vect)\leq  
C\ind_{[t_1,t_2)}(u)\ind_{K}(x,\vect ) \leq \ell(u) h_f(x,\vect )$ for each $(x,\vect )\in
\MK$. This shows that $g\in \MP$ and $0 \preceq g$. \end{proof}

\begin{lemma} \label{lem:Lambda}
The map $\Lambda:\Omega \times \MT_0\rightarrow \R$ in  (\ref{def-Lambda})
defines a  random linear functional on $\MT_0$.  Moreover, 
$\Lambda:\ \MT_0\rightarrow \R$ defined by (\ref{def-Lambda}) is
positive in the sense that  
$\Lambda(g)\geq 0$ whenever $g\geq 0$. 
\end{lemma}
\begin{proof}  
 For $\omega \not \in \Omega_0$, $\Lambda (\omega,\cdot)$ is trivially
 well defined, and is positive and linear on $\MT_0$. 
So, fix $\omega \in \Omega_0$.    To show that $\Lambda(\omega, \cdot)$ is well
defined, we need to show that if $g \in \MT_0$
admits  two representations 
\be
\label{tworeps} g (u,x,\vect) = \sum_{i=1}^n \ell_i (u) h_{f_i} (x,\vect) = \sum_{j=1}^m
\tilde{\ell}_j (u) h_{\tilde{f}_j} (x,\vect),   \qquad (u,x,\vect) \in
\R_+ \times \MK, 
\ee
with $\ell_i, \tilde{\ell}_j \in {\mathcal C}_c (\R_+)$, $f_i, \tilde{f}_j
\in \MH$, $i = 1, \ldots, n$, $j = 1, \ldots, m$, $m, n \in \N$,
then 
\be
\label{chi-toshow} \sum_{i=1}^n \int_0^\infty \ell_i (u) d
\pushproc^{f_i} (\omega, u\wedge \stopr_r(\omega)) =
\sum_{j=1}^m 
\int_0^\infty \tilde{\ell}_j (u) d \pushproc^{\tilde{f}_j}
(\omega, u\wedge \stopr_r(\omega)). 
\ee
First, note that 
since $\pushproc^{f_i}(\omega, \cdot\wedge \stopr_r(\omega))$ and
$\pushproc^{\tilde{f}_j}(\omega, \cdot\wedge \stopr_r(\omega))$ are 
increasing functions, and 
 $\ell_i$ and $\tilde \ell_j$ are continuous with compact support, 
each of the  integrals in \eqref{def-Lambda} is
well defined as a Riemann-Stieltjes integral.   
In fact,  since the functions 
 $\ell_i$ lie in $\MC_c(\R_+)$, $i = 1, \ldots, n$, they are 
uniformly continuous and so for each $\varepsilon>0$,  there exists
$h>0$ such that $|\ell_i(u)-\ell_i(v)|<\varepsilon$ whenever 
$|u-v|\leq h$, $i =1, \ldots, n, j = 1, \ldots, m$.  Thus, for $T$ 
large enough such that  $[0,T]$ contains the supports of
every $\ell_i$, $\tilde{\ell}_j$, $i = 1,\ldots, n$, $j = 1, \ldots,
m$, we have 
\begin{eqnarray}
\label{estimate}
& & \left|\sum_{i=1}^n\int_0^\infty \ell_i(u)
   d\pushproc^{f_i}(u\wedge \stopr_r) - \sum_{i=1}^n \sum_{k=0}^\infty
   \ell_i(kh)(\pushproc^{f_i}((kh+h)\wedge \stopr_r)-\pushproc^{f_i}(kh\wedge \stopr_r))\right| \\ &\leq & \varepsilon
 \sum_{i=1}^n \pushproc^{f_i}(T\wedge \stopr_r), \nonumber
\end{eqnarray}
and, likewise, 
\begin{eqnarray*}
& & \left|\sum_{j=1}^m\int_0^\infty \tilde \ell_j(u) d\pushproc^{\tilde
    f_i}(u\wedge \stopr_r) - \sum_{j=1}^m \sum_{k=0}^\infty \tilde
  \ell_j(kh)(\pushproc^{\tilde f_j}((kh+h)\wedge \stopr_r)-\pushproc^{\tilde
    f_j}(kh\wedge \stopr_r))\right| \\ &\leq & \varepsilon \sum_{j=1}^m \pushproc^{\tilde
  f_j}(T\wedge \stopr_r). \end{eqnarray*}
We now claim that 
\be
\label{show-intermed} 
\sum_{i=1}^n \ell_i(u) \pushproc^{f_i} (\omega, t\wedge \stopr_r(\omega)) = \sum_{j=1}^m
\tilde{\ell}_j(u) \pushproc^{\tilde{f}_j} (\omega, t\wedge \stopr_r(\omega)), 
\qquad u, t \geq 0. 
\ee
If the claim  \eqref{show-intermed} holds, then in particular, 
\begin{eqnarray*}
& & \sum_{i=1}^n \sum_{k=0}^\infty
\ell_i(kh)(\pushproc^{f_i}(\omega, (kh+h)\wedge \stopr_r(\omega))-\pushproc^{f_i}(\omega,kh\wedge \stopr_r(\omega))) \\ &=& \sum_{j=1}^m
\sum_{k=0}^\infty \tilde \ell_j(kh)(\pushproc^{\tilde
  f_j}(\omega,(kh+h)\wedge \stopr_r(\omega))-\pushproc^{\tilde f_j}(\omega,kh\wedge \stopr_r(\omega))), 
\end{eqnarray*}
which, when combined with the two previous bounds, implies 
\begin{eqnarray*} & & \left|\sum_{i=1}^n\int_0^\infty \ell_i(u) d\pushproc^{f_i}(\omega,u\wedge \stopr_r(\omega)) -
  \sum_{j=1}^m\int_0^\infty \tilde \ell_j(u) d\pushproc^{\tilde f_i}(\omega,u\wedge \stopr_r(\omega))
\right|  \\ &\leq &\varepsilon \left(\sum_{i=1}^n \pushproc^{f_i}(\omega,T\wedge \stopr_r(\omega))+\sum_{j=1}^m
  \pushproc^{\tilde f_j}(\omega,T\wedge \stopr_r(\omega)) \right). \end{eqnarray*}
Sending $\varepsilon \downarrow 0$, we obtain \eqref{chi-toshow}. 

Thus, to prove \eqref{chi-toshow},  it suffices to establish \eqref{show-intermed}. 
Define 
\[ \Delta^u (x) \doteq  \sum_{i=1}^n \ell_i (u) f_i(x)  - \sum_{j=1}^m \tilde{\ell}_j
  (u) \tilde{f}_j(x), \quad  (u,x) \in \R_+ \times \bar{G}. 
\] 
Due to the linearity of the space $\MH$ and of the map $f \mapsto h_f$,
 \eqref{tworeps} implies that for each $u \geq 0$, 
$\Delta^u$  lies in $\MH$ and $h_{\Delta^u}(x,v)  =
\langle \nabla \Delta^u (x), v \rangle = 0$
for every $(x,v) \in \MK$.   
  In turn,  this implies that $\pushproc^{\Delta^u} (t\wedge \stopr_r)$ and $-\pushproc^{\Delta^u} (t\wedge \stopr_r)=\pushproc^{-\Delta^u} (t\wedge \stopr_r)$ are both increasing and hence  $\pushproc^{\Delta^u} (t\wedge \stopr_r) = 0$ for every $t \geq 0$. 
By  linearity of the mapping $f \mapsto
\pushproc^f(\cdot \wedge \stopr_r)$, this is equivalent to \eqref{show-intermed}. 
Thus, we have shown that  $\Lambda(\omega, \cdot)$ is a well defined functional on 
$\MT_0$.    The fact that 
$g \mapsto \Lambda(\omega, g)$ is linear is an immediate consequence of  the definition 
of $\Lambda$ in \eqref{def-Lambda}, and 
the fact that the sum of representations of two functions $g,
\tilde{g}$ in $\MT_0$
is a representation for the sum $g + \tilde{g}$. 
Furthemore, for any $g \in \MT_0$, given any representation for $g$ of
the form \eqref{tworeps}, each stochastic Riemann-Stieltjes integral 
$\int_0^\infty \ell_i(u) d\pushproc^{f_i} (u\wedge \stopr_r)$ is a random
variable, and so is its sum.  Since $\Omega_0$ is a measurable set, it follows immediately from
\eqref{def-Lambda} that $\Lambda (\cdot, g)$ is a random variable. 
Thus, $\Lambda$ satisfies 
both properties of Definition \ref{defn-ranlin} and is a random linear
functional on $\MT_0$. 

We now establish the  positivity of $\Lambda$. 
Let $g \in \MT_0$ be such that $g\geq 0$.  
Since $\MT_0 \subset {\mathcal C}_c(\R_+ \times \MK)$,  $g\in \MP$ by
the last assertion of 
Lemma \ref{lem:p}.  
Now, since $g \in \MT_0$, it also 
 admits a representation  of the form
 $g(u,x,\vect)=\sum_{i=1}^n\ell_i(u)h_{f_i}(x,\vect )$ for $\ell_i\in
 \MC_c(\R_+)$ and $f_i \in \MH$, $i = 1, \ldots, n$. 
For $\omega \not \in \Omega_0$, $\Lambda (\omega, \cdot) \equiv 0$. 
On the other hand,  for $\omega \in \Omega_0$ and each
 $u\geq 0$, $ \left< v,\nabla(\sum_{i=1}^n \ell_i(u) f_i)(x) \right> =\sum_{i=1}^n \ell_i(u) h_{f_i}(x,v) \geq 0$ for each $x\in \partial G \setminus \MV$, $v\in d(x)$ and $|v|=1$. So $\sum_{i=1}^n \ell_i(u) f_i\in \MH$ and hence
\[ \sum_{i=1}^n \ell_i(u)
(\pushproc^{f_i}(\omega, (s+h)\wedge \stopr_r(\omega))-\pushproc^{f_i}(\omega, u\wedge \stopr_r(\omega)))\geq 0.\]  
Together with the approximation \eqref{estimate} to the
Riemann-Stieltjes integral, this implies that for any $\ve > 0$, 
$\sum_{i=1}^n \int_0^\infty \ell_i (u) d \pushproc^{f_i} (u\wedge \stopr_r) \geq -\ve
\sum_{i=1}^n \pushproc^{f_i} (T\wedge \stopr_r)$. Sending  
$\varepsilon$ down  to $0$, we conclude that  for all $\omega \in
\Omega_0$, 
\[\Lambda(\omega, g)=\sum_{i=1}^n\int_0^\infty \ell_i(u)
d\pushproc^{f_i}(\omega, u\wedge \stopr_r(\omega))\geq 0.\]
This shows that $\Lambda$ is positive, and completes the proof of the lemma. 
\end{proof}

\subsection{An Integral Representation}
\label{apsub-integrep}

We now use the random positive linear functional $\Lambda$ to show
that $\pushproc^f(\cdot \wedge \stopr_r)$ admits a suitable integrable representation.

\begin{lemma}
\label{lem-chirep}
 There exists a
unique positive regular Borel measure $\mu(\omega,\cdot)$ on $\R_+\times \MK$ such that
for each $f \in \MH$ and $t\geq 0$, \be \label{xif10} 
\pushproc^f(\omega,t\wedge \stopr_r(\omega))= \int_{[0,t]\times \MK} \left<\vect,\nabla f(x)\right>  \mu(\omega,du,dx,d\vect).\ee 
\end{lemma}
 \begin{proof} Fix $\omega \in \ccspace$.  By Lemma \ref{lem:Lambda}, $\Lambda(\omega, \cdot)$ is
   a positive linear functional on $\MT_0$. 
Thus, by 
the positive cone version of
the Hahn-Banach theorem for positive linear functionals (see Theorem
2.1 of \cite{BagBook92}), $\Lambda(\omega, \cdot)$ can be extended to a positive
linear functional on $\MC_c(\R_+\times \MK)$, which we denote by
$\bar{\Lambda}(\omega, \cdot)$.  In turn, 
an application of the Riesz-Markov-Kakutani representation theorem for
positive linear functionals (see Theorem 2.14 of \cite{RudinBook})  shows
that there exists a
unique positive regular Borel measure $\mu(\omega,\cdot)$ on $\R_+\times \MK$ such that
\be
\label{rep-lambda}
\bar{\Lambda}(\omega, g) =\int_{\R_+\times \MK} g(u,x,\vect) \mu(\omega,du,dx,d\vect) \mbox{ for each } g\in
\MC_c(\R_+\times \MK).\ee
Now, 
for each $t>0$, let $\{\ell^n,\ n\in \N\}$ be a sequence of non-negative functions in $\C_c(\R_+)$ such that $\ell^n \uparrow \ind_{[0,t]}$ as $n\rightarrow \infty$.
For each $f \in \MH$, substituting $g^n(u,x,\vect) = \ell^n(u) h_f(x,\vect ) \in \MT_0$ into both the
definition (\ref{def-Lambda}) and the representation (\ref{rep-lambda}) of
$\bar{\Lambda}$,  taking limits as $n\rightarrow \infty$ and invoking
the monotone convergence theorem,  we
obtain \eqref{xif10}.  
 \end{proof}

We now establish some additional properties of the measure
$\mu(\omega, \cdot)$. 
For each $\omega \in \ccspace$, consider the set \[
\overline
\MK(\omega)\doteq \{(u,x,\vect)\in \R_+\times \R^{2J}:\ x=\zee
(\omega, u)
\in \partial G\setminus \MV,\ \vect\in d(x),\ |\vect|=1\},\] where 
we have written $\zee (\omega, u)$ instead of $\zee (u)$ to make clear
the dependence of the right-hand side on $\omega$.  

\begin{lemma} \label{lem:musupport} There exists $\Omega_0 \in
  {\mathcal M}$ with $\Q_z (\Omega_0) = 1$ such
  that
\be
\label{chif-int2}
\mu(\omega, [\R_+\times \MK] \setminus\overline \MK(\omega))=0, \qquad 
\mbox{ for }  \omega\in \Omega_0. 
\ee
\end{lemma}
\begin{proof}
Let $\Omega_0 \in {\mathcal M}$ be the set of full $\Q_z$-measure such that
the relation 
\eqref{chig} holds. 
For each $\omega\in \Omega_0$ and $\varepsilon>0$, let $\stopin_0^\ve(\omega)\doteq 0$ and for each $k\in \N$, let 
\begin{eqnarray*}
 \stopout_k^\ve(\omega) & \doteq &  \inf \{t\geq \varsigma_{k-1}^\ve(\omega):\ \zee(\omega,t)
 \in \partial G\}, \\ \stopin_k^\ve(\omega) & \doteq & \inf \{t\geq \stopout_k^\ve(\omega): \
  \dist(\zee(\omega,t),\zee(\omega,\stopout_k^\ve(\omega)))\geq \ve\}. 
\end{eqnarray*}
Note that for some $k$ and $\omega$,  we could have $\stopout_{k+1}^\ve (\omega)  =
\stopin_k^\ve (\omega)$. 
For $\omega\in \Omega_0$, $\ve > 0$ and $k \in \N$,  define the sets 
\[  \MA^\ve_k(\omega) \doteq  \left\{(u,x,\vect)\in  [\stopout_k^\ve(\omega),
  \stopin_k^\ve(\omega)] \times\MK:  \dist(x,\zee (\omega,\stopout_{k}^\ve(\omega)))\leq \ve\right\}
\]
and 
\[  \ME^\ve_k (\omega) \doteq  \left\{(u,x,\vect)\in 
[\stopout_k^\ve(\omega),
  \stopin_k^\ve(\omega)] \times \MK:  \dist(x,\zee
  (\omega,\stopout_{k}^\ve(\omega))) > \ve \right\}.
\]
Note that 
$\overline \MK(\omega) = \bigcap_{\ve > 0} \bigcup_{k \geq 1} \MA^\ve_k (\omega)$. In fact, it is easy to see that $\overline \MK(\omega) \subset \bigcap_{\ve > 0} \bigcup_{k \geq 1} \MA^\ve_k (\omega)$. For the converse, let $(u,x,\vect) \in \bigcap_{\ve > 0} \bigcup_{k \geq 1} \MA^\ve_k (\omega)$. Then for each $\ve > 0$, there exists $k(\ve)\geq 1$ such that $(u,x,\vect) \in \MA^\ve_{k(\ve)} (\omega)$. Thus, $u\in [\stopout_{k(\ve)}^\ve(\omega),
  \stopin_{k(\ve)}^\ve(\omega)] $ and $(x,v) \in \MK$ such that $\dist(x,\zee (\omega,\stopout_{k(\ve)}^\ve(\omega)))\leq \ve$. It follows that $\dist(\zee(\omega,u),\zee(\omega,\stopout_{k(\ve)}^\ve(\omega)))\leq \ve$. and hence $\dist(x, \zee(\omega,u))\leq 2\ve$. Since $\zee(\omega,\stopout_{k(\ve)}^\ve(\omega))\in \partial G$, we have $\dist(\zee(\omega,u),\partial G)\leq \ve$. By letting $\ve \downarrow 0$, it follows that $x=\zee(\omega,u) \in \partial G$. Since $(x,v)\in \MK$, we have $x=\zee(\omega,u) \in \partial G\setminus \MV$. Thus, $(u,x,v) \in \overline \MK(\omega)$. This completes the proof of $\overline \MK(\omega) = \bigcap_{\ve > 0} \bigcup_{k \geq 1} \MA^\ve_k (\omega)$.

Observe that for every $\ve > 0$, 
\[  [\R_+ \times \MK] \sm \left[ \bigcup_{k \geq 1} \MA^\ve_k (\omega) \right]
 = \left[ \bigcup_{k \geq 1}  \ME^\ve_k (\omega) \right] \bigcup
 \left[  \bigcup_{k \geq 0}
   (\stopin_k^\ve(\omega),\stopout_{k+1}^\ve(\omega) )\times \MK
 \right]. 
\]
Thus, to show \eqref{chif-int2}, it suffices to show that for each $\omega\in \Omega_0$ and $\ve > 0$, 
\be \label{c1} \mu\left(\omega, \ME^\ve_k(\omega)\right)  = 0,  \quad k \in
\N, 
\ee
and 
\be 
\label{c2}
\mu\left(\omega,   (\stopin_k^\ve(\omega),\stopout_{k+1}^\ve(\omega)
  )\times \MK\right) =0,   \quad k \in
\N. 
\ee

To establish (\ref{c1}) and (\ref{c2}), choose $r, s > 0$ such that 
$B_{\kappa_{\barz}(r)}(\barz) \subseteq \MU_{r,s}= \{ x\in \partial G: |x| \leq r, d(x, \MV)
\geq s\}$ and 
recall the definition of $\stopr_r$ in (\ref{def-stopr}). 
From Theorem 2 of  \cite{KanRam14},  
it follows that there exists a function  $f_{r,s} \in \MH \cap \MC_c^2
(\bar{G})$ such 
that for every $x\in \MU_{r,s}$, $\langle d, \nabla f_{r,s}(x)\rangle \geq 1$ for all $d\in d(x) \cap
  S_1(0)$. 
Together with (\ref{xif10}), (\ref{chig}) and the continuity of $\pushproc^{f_{r,s}}(\cdot\wedge \stopr_r)$, this  shows that 
 \begin{eqnarray*}
   0&=& \pushproc^{f_{r,s}}(\omega,
   \stopout_{k+1}^\ve(\omega)\wedge \stopr_r(\omega))-\pushproc^{f_{r,s}}(\omega, \stopin_{k}^\ve(\omega)\wedge \stopr_r(\omega))\\
   & =&  \int_{(\stopin_k^\ve(\omega),\stopout_{k+1}^\ve(\omega) ]
     \times \MK} \left<\vect,\nabla f_{r,s}(x)\right>
   \mu(\omega,du,dx,d\vect)\\ &\geq & \mu(\omega, (\stopin_k^\ve(\omega),\stopout_{k+1}^\ve(\omega) )\times \MK_{r,s}),\end{eqnarray*} where $\MK_{r,s}=\{(x,\vect)\in \R^{2J}:\ x\in \MU_{r,s}, \vect\in d(x),\ |\vect|=1\}$. Then $\mu(\omega, (\stopin_k^\ve(\omega),\stopout_{k+1}^\ve(\omega) )\times \MK_{r,s})=0$ for each $r,s>0$. Note that $\MK_{r,s}\uparrow \MK$ as $r\rightarrow \infty$ and $s\rightarrow 0$. 
This proves (\ref{c2}).  To show (\ref{c1}), by (\ref{xif10}) again, we have that for each
$f\in \MH$,  \be \label{xif1}
\pushproc^{f}(\omega,\stopin_{k}^\ve(\omega)\wedge \stopr_r(\omega)) -
\pushproc^{f}(\omega,\stopout_{k}^\ve(\omega)\wedge \stopr_r(\omega)) =
\int_{[\stopout_{k}^\ve(\omega),\stopin_{k}^\ve(\omega) ] \times \MK}
\left<\vect, \nabla f(x)\right>  \mu(\omega,du,dx,d\vect). \ee
Note that for each $f\in \MH$ with support outside
$B_\ve(\zee(\omega,\stopout_{k}^\ve(\omega)))$, by the definition of
$\pushproc^{f}$ in (\ref{def-chig}), we
have \[\pushproc^{f}(\omega,\stopin_{k}^\ve(\omega)\wedge \stopr_r(\omega)) -
\pushproc^{f}(\omega,\stopout_{k}^\ve(\omega)\wedge \stopr_r(\omega))=0.\] Thus, by running
over all $f\in \MH$ with support outside $B_\ve(\zee (\omega,\stopout_{k}^\ve(\omega)))$, we have that (\ref{c1}) holds. This completes the proof of the lemma. \end{proof}

\subsection{Proof of Proposition \ref{prop-chig}} 
\label{subs-proofchig}

We now present the proof of Proposition \ref{prop-chig}.  
Note that as a consequence of  (\ref{xif10}) and Lemma \ref{lem:musupport}, for every $\omega \in \Omega_0$, we have
\be \label{xif21} \pushproc^f(\omega, t\wedge \stopr_r(\omega)) = \int_{[0,t]\times \MK \cap
  \overline \MK(\omega)} \left<\vect, \nabla f(x)\right>
\mu(\omega,du,dx,d\vect), \qquad f \in \MH. \nonumber \ee 
Define a random measure $\tilde \mu: \ccspace \times \R_+
\times \MK \mapsto \R$ as follows: 
\begin{eqnarray}
\label{def-tildemu}
\tilde \mu(\omega,du,d\vect) \doteq 
\left\{ 
\begin{array}{ll}
 \ind_{\{x=\zee(\omega,u)\in \partial G\setminus
   \MV\}}\mu(\omega,du,dx,d\vect) & \mbox{ if } \omega \in \Omega_0, \\
0 & \mbox{ otherwise. }
\end{array}
\right.  \nonumber
\end{eqnarray} 
Then for $\omega \in \Omega_0$, clearly \eqref{xif31} holds. Since $\mu(\omega,\cdot)$ and hence $\tilde \mu(\omega,\cdot)$ is a Borel measure by Lemma \ref{lem-chirep}, it follows that for each compact set $K\subset \R_+\times \MK$, $\tilde \mu(\omega,K)<\infty$ and therefore $\tilde \mu(\omega,\cdot)$ is $\sigma$-finite. 
Since $\Q_z(\Omega_0)  = 1$, this proves the first part of 
Proposition \ref{prop-chig}. 

Next, note that $\zee (\omega,\cdot \wedge \theta_r(\omega))$ lives in
$B_{\kappa_{\barz}(r)}(\barz)$ and there exists a function $f_{r,s}\in \MH$ such that
$\left<\vect,\nabla f_{r,s}(x) \right>\geq 1$ for each $x\in \partial G \cap
B_{\kappa_{\barz}(r)}(\barz)$, $\vect\in d(x)$ and $|\vect|=1$. For each $f\in \C_c^2(\bar G)$,
there exists a constant $C>0$ such that $f+C f_{r,s} \in \MH$. Since
(\ref{xif31}) holds for both $f_{r,s}$ and $f+C f_{r,s}$, then
(\ref{xif31}) holds for $f$. For any function $f\in \C^2(\bar G)$,
there exists a function $g\in \C^2_c(\bar G)$ such that $f=g$ on
$B_{\kappa_{\barz}(r)}(\barz)$. Since $\zee(\omega,\cdot \wedge \theta_r(\omega))$ lives in
$B_{\kappa_{\barz}(r)}(\barz)$, (\ref{xif31}) also holds for each $f\in \C^2(\bar
G)$. It follows that for each $t\geq 0$ and $f\in \C^2(\bar G)$ that
is uniformly positive, \begin{eqnarray*} & & \int_0^t \int_{d(\zee(\omega,u))\cap S_1(0)}\ind_{\{Z(\omega,u)\in \partial G\setminus \MV\}}
\left<\vect,\frac{\nabla f(\zee (\omega,u))}{f(\zee (\omega,u))}\right>  \tilde
\mu(\omega,du,d\vect) \\ &=& \left\{\begin{array}{ll} \int_0^t \frac{1}{f(\zee (\omega,u))} d\pushproc^f(\omega,u \wedge \stopr_r(\omega)) & \mbox{ if }\omega \in \Omega_0 \\ 0 &  \mbox{ if }\omega \notin \Omega_0 \end{array} \right.\end{eqnarray*} as a function of $\omega$ is a random variable. 

Let $\vartheta\in \R^J$. Choose $f(x)=\exp\{\left<\vartheta,
  x\right>\}$.
Then $f\in \C^2(\bar G)$ is uniformly
positive. Simple calculations yield that $\frac{\nabla
  f(x)}{f(x)}= \vartheta$, then by plugging  $f(x)=\exp\{\left<\vartheta,
  x\right>\}$ into the previous display, we have
$\left<\vartheta, \int_{\intset_t} \vect  \tilde
  \mu(\cdot,du,d\vect) \right>$, $t\geq 0$, is a one-dimensional continuous stochastic process.  Since $\vartheta$
is arbitrary, then we have $R^\vect(t) \doteq \int_{\intset_t} \vect  \tilde
  \mu(\cdot,du,d\vect)$, $t\geq 0$, is a $J$-dimensional continuous stochastic process and the $i$th component of $R^\vect(t)$, denoted by $R^\vect_i(t)$, is $\int_{\intset_t} \vect_i  \tilde
  \mu(\cdot,du,d\vect)$.  Let $g:\ \R^J\rightarrow \R^J$ be a continuous function and $g_i$ denote its $i$th component. It follows that \[ \int_{\intset_t} \left<\vect, g(Z(u)) \right> \tilde
  \mu(\cdot,du,d\vect) = \sum_{i=1}^J \int_{\intset_t} \vect_i g_i(Z(u))  \tilde
  \mu(\cdot,du,d\vect) = \sum_{i=1}^J\int_0^t g_i(Z(u)) d R^\vect_i(u).\] Hence $\int_{\intset_t} \left<\vect, g(Z(u)) \right> \tilde
  \mu(\cdot,du,d\vect)$, $t\geq 0$, is also a continuous process.
This completes the proof of Proposition \ref{prop-chig}.

\end{document}